\documentclass[12pt, twoside]{article}
\usepackage{amsmath,amsthm,amssymb,graphicx,euscript, mathtools,mathrsfs}
\usepackage{times, url}
\usepackage{enumerate}

\def\titlerunning#1{\gdef\titrun{#1}}
\makeatletter
\def\author#1{\gdef\autrun{\def\and{\unskip, }#1}\gdef\@author{#1}}
\def\address#1{{\def\and{\\\hspace*{18pt}}\renewcommand{\thefootnote}{}%
\footnote {#1}}%
\markboth{\autrun}{\titrun}}
\makeatother
\def\email#1{e-mail: #1}
\def\subjclass#1{{\renewcommand{\thefootnote}{}%
\footnote{\emph{Mathematics Subject Classification (2010):} #1}}}
\def\keywords#1{\par\medskip
\noindent\textbf{Keywords.} #1}

\newtheorem{thm}{Theorem}[section]
\newtheorem{cor}[thm]{Corollary}
\newtheorem{lem}[thm]{Lemma}

\newtheorem{prop}[thm]{Proposition}

\theoremstyle{definition}
\newtheorem{defin}[thm]{Definition}
\newtheorem{rem}[thm]{Remark}

\newtheorem{proc}[thm]{Process}

\newcommand{\BB}{\EuScript{B}}

\newcommand{\DD}{\EuScript{D}}
\newcommand{\KK}{\EuScript{K}}

\newcommand{\R}{\mathbb{R}}
\newcommand{\Z}{\mathbb{Z}}

\newcommand{\calD}{\mathcal{D}}

\newcommand{\MF}{\mathop{\mathrm{MF}}\nolimits}

\newcommand{\Endop}{\mathop{\mathrm{Endop}}\nolimits}
\newcommand{\eu}{e^{\lambda}}
\newcommand{\el}{e_{\lambda}}
\newcommand{\modu}{\mathop{\mathrm{mod}}\nolimits}
\newcommand{\MCM}{\mathop{\mathrm{MCM}}\nolimits}
\newcommand{\Perf}{\mathop{\mathrm{Perf}}\nolimits}

\newcommand{\Inj}{\mathop{\mathrm{Inj}}\nolimits}
\newcommand{\id}{\mathop{\mathrm{id}}\nolimits}
\newcommand{\Hom}{\mathop{\mathrm{Hom}}\nolimits}
\newcommand{\Ker}{\mathop{\mathrm{Ker}}\nolimits}
\newcommand{\Tor}{\mathop{\mathrm{Tor}}\nolimits}

\newcommand{\Ext}{\mathop{\mathrm{Ext}}\nolimits}
\newcommand{\sg}{\mathop{\mathrm{sg}}\nolimits}

\newcommand{\op}{\mathop{\mathrm{op}}\nolimits}

\newcommand{\sy}{\mathop{\mathrm{sy}}\nolimits}
\newcommand{\nc}{\mathop{\mathrm{nc}}\nolimits}
\newcommand{\HH}{\mathop{\mathrm{HH}}\nolimits}
\newcommand{\Barr}{\mathop{\mathrm{Bar}}\nolimits}
\newcommand{\coker}{\mathop{\mathrm{Coker}}\nolimits}
\newcommand{\Imm}{\mathop{\mathrm{Im}}\nolimits}
\newcommand{\Spec}{\mathop{\mathrm{Spec}}\nolimits}
\newcommand{\THH}{\mathop{\mathrm{TH}}\nolimits}
\newcommand{\Cact}{\mathop{\mathscr{C}act}\nolimits}
\newcommand{\Brace}{\mathop{\BB\mathrm{race}}\nolimits}
\input xy
\xyoption{all}

\makeatletter
\newcommand{\colim@}[2]{%
  \vtop{\m@th\ialign{##\cr
    \hfil$#1\operator@font colim$\hfil\cr
    \noalign{\nointerlineskip\kern1.5\ex@}#2\cr
    \noalign{\nointerlineskip\kern-\ex@}\cr}}%
}
\newcommand{\colim}{%
  \mathop{\mathpalette\colim@{\rightarrowfill@\textstyle}}\nmlimits@
}
\makeatother

\numberwithin{equation}{section}

\begin{document}


\baselineskip=17pt


\titlerunning{Gerstenhaber algebra and Deligne's conjecture on Tate-Hochschild cohomology}
\title{Gerstenhaber algebra and Deligne's conjecture on Tate-Hochschild cohomology}
\author{Zhengfang Wang
}
\date{}
\maketitle
\address{Beijing International Center for Mathematical Research, Peking University, No. 5 Yiheyuan Road Haidian District, Beijing 100871, PR China;
\and Universit\'e Paris Diderot-Paris 7, IMJ-PRG CNRS UMR 7586, B\^atiment Sophie Germain, Case 7012, 75205 Paris Cedex 13, France;   \email{zhengfang.wang@imj-prg.fr}}
\begin{abstract}
Using non-commutative differential forms, we construct a complex called {\it singular Hochschild cochain complex}  for any associative algebra over a field. The cohomology of this complex is isomorphic to the Tate-Hochschild cohomology in the sense of Buchweitz.  
By a natural action of the cellular chain operad of the spineless  cacti operad, introduced by R. Kaufmann, on the singular Hochschild cochain complex,   we provide a proof of the Deligne's conjecture for this complex. More concretely, the complex is an algebra over the (dg) operad of chains of the little $2$-discs operad. By this action, we also obtain that the singular Hochschild cochain complex has a $B_{\infty}$-algebra structure and its cohomology ring is a Gerstenhaber algebra.

Inspired by the original definition of Tate cohomology for finite groups, we define a generalized Tate-Hochschild  complex with the Hochschild chains in negative degrees and the Hochschild cochains in non-negative degrees.  There is a natural embedding of this complex into  the singular Hochschild cochain complex. In the case of a self-injective algebra, this embedding becomes a quasi-isomorphism.  In particular,  for a symmetric algebra, this  allows us to show that the Tate-Hochschild cohomology ring, equipped with the Gerstenhaber algebra structure, is a Batalin-Vilkovisky algebra.


\end{abstract}

\section{Introduction}
Hochschild cohomology, introduced by Hochschild \cite{Hoc} in 1945,  is a cohomology theory for associative algebras. Motivated by Eilenberg-MacLane's approach to the cohomology theory of groups, Hochschild introduced a cochain complex $C^*(A, M)$ for an associative algebra $A$  and an $A$-$A$-bimodule $M$. The Hochschild cohomology groups  (with coefficients in $M$) of $A,$ denoted by $\HH^*(A, M)$, are defined as the cohomology groups of $C^*(A, M)$. Recall that $\HH^i(A, M)$ is isomorphic to the space of morphisms from $A$ to $s^iM$ in the bounded derived category $\DD^b(A\otimes A^{\op})$ of  $A$-$A$-bimodules, where $s^i$ is the $i$-th shift functor for $i\in\Z$.

Later in the 1960s, Gerstenhaber \cite{Ger} found that there is a rich algebraic structure on $C^*(A, A)$ when studying the deformation theory of associative algebras. There is a cup product,  which makes $C^*(A, A)$ into a differential graded (dg) associative algebra. This cup product has a remarkable property that it is not commutative on $C^*(A, A)$ but graded commutative up to homotopy. He also constructed a differential graded (dg) Lie algebra  (of degree $-1$) structure on $C^*(A, A)$.   The induced Lie  bracket on $\HH^*(A, A)$ satisfies the graded Leibniz rule with respect to the cup product. Nowadays we call $\HH^*(A, A)$, together with the Lie bracket (called Gerstenhaber bracket) and cup product, a Gerstenhaber algebra (cf. Theorem \ref{thm-ger}). Moreover,   Gerstenhaber  showed that the dg Lie algebra $C^{*+1}(A, A)$ controls the deformation theory of $A$.

Recall that the little $2$-discs operad  is a topological operad whose space in arity $n$ is the topological space of  standard embeddings (i.e.  translations composed with dilations) of the disjoint union of $n$ discs into a standard disc.  Cohen \cite{Coh} in 1973 found that if a topological space $X$ is an algebra over the little $2$-discs operad, then its singular homology $H_*(X)$ is a Gerstenhaber algebra. In 1993, Deligne asked  whether the Hochschild cochain complex $C^*(A, A)$ of an associative algebra $A$ has a natural action of the little $2$-discs operad. This is the original Deligne's conjecture for Hochschild cochain complexes, which has been proved by several researches  using different chain models of the little $2$-discs operad (cf. \cite{Tam98b, Kon, KoSo, Vor, McSm, Kau07a}). We also refer to \cite{Kon} for its connection with Kontsevich's deformation quantization theorem.

On the other hand, in the 1980s,  Buchweitz in an unpublished manuscript \cite{Buc} provided a general framework for Tate cohomology of Gorenstein algebras. To do this, he introduced the notion of {\it stable derived category} as the Verdier quotient of the bounded derived category by the full subcategory consisting of compact objects. This notion is also known as the {\it singularity category}, rediscovered by Orlov \cite{Orl04} in the study of homological mirror symmetry. 
Under Buchweitz's framework, for any Noetherian algebra $A$ (not necessarily commutative), it is very natural to define the Tate-Hochschild cohomology groups   as the morphism spaces from $A$ to  $s^iA$ ($i\in \Z$) in the singularity category $\calD_{\sg}(A\otimes A^{\op})$ of finitely generated $A\otimes A^{\op}$-modules.  The Tate-Hochschild cohomology  has been investigated by a few authors (cf. \cite{BeJo, EuSc, Ngu}) only in the case of Frobenius algebras.  We remark that Tate cohomology is also implicitly Tate-Vogel's cohomology \cite{Goi} and exposed in \cite{AvVe}. The Tate-Hochschild cohomology is also called singular Hochschild cohomology in \cite{RiWa, Wan}.

This paper attempts to provide a  more complete picture of Tate-Hochschild cohomology by describing  richer algebraic structures as it was done for Hochschild cohomology.   These algebraic structures might shed new light on the study of Tate-Hochschild cohomology not only in algebra but also in other fields such as noncommutative geometry, symplectic geometry, operad theory, and string topology.

We start with constructing a complex $C_{\sg}^*(A, A),$ called singular Hochschild cochain complex,  for any associative algebra $A$, which calculates the Tate-Hochschild cohomology of $A$. It is a colimit of Hochschild cochain complexes $C^{*}(A, \Omega^p_{\nc}(A))$  with coefficients in the non-commutative differential forms $\Omega_{\nc}^p(A)$ (concentrated in degree $-p\in\Z_{\leq 0}$) along natural embeddings    $\theta_p: C^{*}(A, \Omega_{\nc}^p(A))\hookrightarrow C^{*}(A, \Omega_{\nc}^{p+1}(A)), f \mapsto f\otimes \id_{s\overline{A}}$   (cf. Definition \ref{defin3.2}). In other words,  $C_{\sg}^*(A, A)$ has a filtration of cochain complexes $$0\subset C^*(A, A)\subset C^*(A, \Omega_{\nc}^1(A))\subset \cdots \subset C^{*}(A, \Omega_{\nc}^p(A)) \subset \cdots \subset C_{\sg}^*(A, A).$$ This yields a natural map from $\HH^*(A, A)$ to the cohomology, denoted by $\HH_{\sg}^*(A, A)$, of $C_{\sg}^*(A, A)$. 
 Moreover, this map   coincides with the one induced by the quotient functor from  $\DD^b(A\otimes A^{\op})$ to the singularity category $\DD_{\sg}(A\otimes A^{\op})$  (cf. Theorem \ref{thm1}).

A natural question is whether the Gerstenhaber algebra structure on $\HH^*(A, A)$ can be extended to $\HH_{\sg}^*(A, A)$. We give an affirmative answer to this question by an  explicit  construction of  a cup product and Lie bracket (of degree -1) on $C_{\sg}^*(A, A)$, which makes $\HH^*_{\sg}(A, A)$ into a Gerstenhaber algebra (cf. Corollary \ref{cor5.3}). We further show that the natural map from $\HH^*(A, A)$ to $\HH_{\sg}^*(A, A)$ is a morphism of Gerstenhaber algebras.

In the series of papers \cite{Kau05, Kau07a, Kau08}, the author introduced the (topological) operad $\Cact$ of spineless cacti. He proved that the cellular chain (dg) operad $CC_*(\Cact )$  is  equivalent to  the operad of chains of the little $2$-discs operad  (cf. \cite[Theorem 3.11]{Kau07a}).  There is a natural action of $CC_*(\Cact)$ on $C^*(A, A)$. Recall that the brace operations (cf. Definition \ref{defin5.41}) on $C^*(A, A)$, due to Kadeishvili \cite{Kad} and Getzler \cite{Get}, play a crucial role in almost all existing proofs of the Deligne's conjecture. The brace  operations, together with the cup product,  endow $C^*(A, A)$ with a $B_{\infty}$-algebra structure (cf. \cite[Theorem 3.1]{Vor}). Let $\Brace$ be the dg suboperad of the endomorphism operad $\Endop(C^*(A, A))$, generated by the cup product and the brace operations.  From \cite[Proposition 4.9]{Kau07a} it follows that  $\Brace$ is isomorphic to $CC_*(\Cact)$. In this paper, we will show  that the action of $CC_*(\Cact)$ on $C^*(A, A)$ can be naturally extended to  $C_{\sg}^*(A, A)$. As a consequence, we obtain that $C_{\sg}^*(A,A)$ is a $B_{\infty}$-algebra with a $B_{\infty}$-subalgebra $C^*(A, A)$ (cf. Theorem \ref{thm5.2}) and  that the Deligne's conjecture holds for $C^*_{\sg}(A, A)$ (cf. Theorem \ref{thm5.1}).

Motivated by the original definition of  Tate cohomology for finite groups, we construct another unbounded complex $\calD^*(A, A)$,  called generalized Tate-Hochschild complex, for any associative algebra $A$:
\begin{equation*}
\calD^*(A, A): \cdots\xrightarrow{b_2} C_1(A, A^{\vee})\xrightarrow{b_1} A^{\vee}\xrightarrow{\mu} A\xrightarrow{\delta^0} C^1(A, A)\xrightarrow{\delta^1} \cdots,
\end{equation*}
where  $A^{\vee}:=\Hom_{A\otimes A^{\op}}(A, A\otimes A^{\op})$ 
and the differential $\mu$   is given by $\mu\left(\sum_i x_i\otimes y_i\right)=\sum_i x_iy_i$. In general, this complex does not caculate the Tate-Hochschild cohomology,  but however there exists a natural embedding of $\calD^*(A, A)$ into $C_{\sg}^*(A, A)$. Moreover, this embeding becomes a quasi-isomorphism if $A$ is a self-injective algebra over a field. In particular, when $A$ is symmetric, this allows us to prove that the Tate-Hochschild cohomology,  equipped with the Gerstenhaber algebra structure,  is a Batalin-Vilkovisky (BV) algebra. The BV differential operator is induced by the Connes' $B$ operator and its dual on $\calD^*(A, A)$ (cf. Theorem \ref{thm6.17}).  Inspired by the cyclic Deligne's conjecture \cite{Kau08}, it is natural to ask  whether $C_{\sg}^*(A, A)$ (or equivalently, $\calD^*(A, A)$) is an algebra over the framed little $2$-discs operad if $A$ is a symmetric algebra.

{\bf Related and future works:}
It follows from \cite{LoVa05} that the Hochschild cohomology  of a dg algebra is isomorphic to the Hochschild cohomology of  dg enhancements of its derived category.  Inspired by this fact, it is interesting to  study the relationship between the Tate-Hochschild cohomology and the Hochschild cohomology of dg enhancements of its singularity category. This problem is closely related to the uniqueness (up to quasi-equivalences) of dg enhancements of a singularity category since  two quasi-equivalent dg categories have the same  Hochschild cohomology (cf. \cite{Kel, Toe}). 

Let $(R, \mathfrak{m})$ be a regular local ring. Suppose that $w\in\mathfrak{m}$ be a non-zero element such that the hypersurface $\Spec(R/w)$ has an isolated singularity at $\mathfrak{m}$. From \cite[Corollary 6.4]{Dyc} it follows that the Hochschild cohomology of the $2$-periodic dg category of matrix factorizations $\MF_{\Z}(R, w)$ is isomorphic to the Jacobian algebra $R/(\partial_1w, \cdots, \partial_nw)$  in even degree and vanishes in odd degree. This is a $\Z/2\Z$-graded version of Hochschild cohomology. But $\HH_{\sg}^*(R/w, R/w)$ is isomorphic to the Tyurina algebra $R/(w, \partial_1w, \cdots, \partial_nw)$ in each even degree.   
Thus in general, $\HH_{\sg}^*(R/w, R/w)$ is  not isomorphic to the Hochschild cohomology of the dg enhancement $\MF_{\Z}(R, w)$ of $\DD_{\sg}(R/w)$ after translating the $\Z/2\Z$-graded version to $\Z$-graded one.  On the contrary, let $Q$ be a finite quiver (not necessarily acyclic) without sources or sinks. Denote by $A_Q$  the radical square zero algebra $kQ/\langle Q_2\rangle$, where $Q_2$ is the set of paths of length $2$.  
We show \cite{ChLiWa} that  $\HH_{\sg}^*(A_Q, A_Q)$ is isomorphic to the Hochschild cohomology of the dg category $\KK_{ac}(\mbox{$A_Q$-$\Inj$})^c,$ the full subcategory of compact objects in  the dg category of acyclic complexes of injective $A_Q$-modules, of $\DD_{\sg}(A_Q)$. It is known that $\KK_{ac}(\mbox{$A_Q$-$\Inj$})^c$ is a dg enhancement of $\DD_{\sg}(A_Q)$ (cf. \cite{Kra}). 

In order to understand the relevance of the algebraic structures discussed in this paper in other fields such as  symplectic geometry and string topology, we  generalize our constructions  to the dg case in \cite{RiWa}.  As an application, we  provide  a rational homotopy invariant of topological spaces. More explicitly, suppose that $X$ is a topological space and $C^*(X)$ is its rational singular cochain complex, which is clearly a dg algebra. We show that the singular Hochschild cochain complex $C_{\sg}^*(C^*(X), C^*(X))$ gives a rational homotopy invariant of $X$.   We also obtain that $\HH_{\sg}^*(C^*(M), C^*(M))$  of a simply-connected closed manifold $M$ is isomorphic to the Rabinowitz-Floer homology of the unit disc cotangent bundle $D(T^*M)\subset T^*M$ with the canonical symplectic structure (cf. \cite[Theorem 7.1]{RiWa} and \cite[Theorem 1.10]{CiFrOa}). 
Inspired by the open-closed and closed-open string maps in symplectic geometry (cf. \cite{Sei}), it is interesting to wonder whether this isomorphism lifts to the chain level from a geometric point of view.

In deformation theory, there is a general guiding principle that every deformation problem in characteristic zero is governed by a dg Lie algebra, due to numerous researchers such as Deligne, Grothendieck, Drinfeld, and Kontsevich. Recently, this principle is formulated by Lurie \cite{Lur} via the language of $\infty$-categories.  
To the best of the author's knowledge,  it is still unclear which deformation problem the dg Lie algebra $C_{\sg}^{*+1}(A, A)$ controls. Motivated by the works \cite{LoVa06,  KeLo} on the deformation theory of abelian categories and triangulated categories, it is expected that  $C_{\sg}^{*+1}(A,A)$ is related to the deformation theory of the singularity category $\DD_{\sg}(A)$. 

Throughout this paper, we fix a field $k$. 
For simplifying the notation, we always write $\otimes$ instead of $\otimes_k$ and write $\Hom$ instead of  $\Hom_k$,  when no confusion may occur. For simplicity, we write $s\overline{a}_{i, j}:=s\overline{a}_i \otimes s\overline{a}_{i+1}\otimes \cdots \otimes s\overline{a}_j\in (s\overline{A})^{\otimes j-i+1}$ for $0\leq i\leq j,$ where $s$ is the shift functor in the category of complexes.

\section{Preliminaries}
\subsection{Hochschild homology and cohomology}
\subsubsection{Normalized bar resolution} \label{bar-resolution} Let $A$ be a unital associative algebra over a field $k$. Let $\overline{A}$ be the quotient $k$-module $A/(k\cdot 1)$ of $A$ by the $k$-scalar multiplies of the unit. Denote by $s\overline{A}$ the graded $A$-module concentrated in degree $-1$, namely, $(s\overline{A})^{-1}=\overline{A}$.  The {\it normalized bar resolution} $\Barr_*(A)$ is the complex of $A$-$A$-bimodules with $\Barr_p(A)=A\otimes s\overline{A}^{\otimes p}\otimes A$ for $p\in \Z_{\geq0}$
and the differentials $$d_p(a_0\otimes s\overline{a}_{1, p}\otimes a_{p+1}):=\sum_{i=0}^{p} (-1)^ia_0\otimes s\overline{a}_{1, i-1}\otimes s\overline{a_ia_{i+1}}\otimes s\overline{a}_{i+2, p}  \otimes a_{p+1}.$$ Here we remark that the term corresponding to $i=0$ in the  above formula should be $a_0a_1\otimes s\overline{a}_{2, p+1}$ and similarly the term of $i=p$ is $a_0\otimes s\overline{a}_{1, p-1} \otimes a_pa_{p+1}$. In order  to shorten the formula,  we write the  sum in such uniform way when no confusion may occur. 
It is well-known  (cf. e.g. \cite{Lod, Zim}) that $\Barr_*(A)$ is a projective bimodule resolution of $A$.


\subsubsection{Definitions of Hochschild (co)-homology}
Let $A$ be an associative algebra over a field $k$ and $M$ be a graded $A$-$A$-bimodule. The   {\it normalized Hochschild cochain complex} $C^*(A, M)$ with coefficients in $M$ is obtained by applying the functor $\Hom_{A\otimes A^{\op}}(-, M)$ to the normalized bar resolution $\Barr_*(A)$ and then using the canonical isomorphisms $\Hom_{A\otimes A^{\op}}(A\otimes s\overline{A}^p\otimes A, M) \cong \Hom(s\overline{A}^{\otimes p}, M)$. Therefore, $C^*(A, M)$ is the following complex$$\cdots\rightarrow C^{-1}(A, M) \rightarrow C^{0}(A, M)\xrightarrow{\delta^0} \cdots \rightarrow C^{p-1}(A, M)\xrightarrow{\delta^{p-1}} C^p(A, M)\xrightarrow{\delta^p} \cdots$$
where  $C^p(A, M):=\prod_{i\in \Z_{\geq 0}} \Hom((s\overline{A})^{\otimes i}, M)^p$ for $ p\in\Z$ and $$\Hom((s\overline{A})^{\otimes i}, M)^p:=\{f\in \Hom((s\overline{A})^{\otimes i}, M) \ | \ \mbox{$f$ is graded of degree $p$}\}.$$ Here we recall that $s\overline{A}$ is a graded $k$-module concentrated in degree $-1$.  The  differential is given by
\begin{equation*}
\begin{split}
(-1)^{p+1}\delta^{p}(f)(s\overline{a}_{1, i+1})=&a_1f(s\overline{a}_{2, i+1})+\sum^{i}_{j=1} (-1)^j f(s\overline{a}_{1, j-1}\otimes s\overline{a_ja_{j+1}} \otimes s\overline{a}_{j+2, i+1})\\
&+ (-1)^{i+1}f(s\overline{a}_{1, i})a_{i+1},
\end{split}
\end{equation*}
for $f\in C^{p}(A, M)$. The $p$-th  {\it Hochschild cohomology} of $A$ with coefficients in $M$, denoted by $\HH^p(A, M),$ is defined as the cohomology group $\frac{\Ker(\delta^p)}{\Imm(\delta^{p-1})}$ of $C^*(A, M)$. In particular,  we call $\HH^*(A, A)$ the Hochschild cohomology ring of $A$, and $C^*(A, A)$ the Hochschild cochain complex of $A$.

Similarly,  applying the functor $M\otimes_{A\otimes A^{\op}}-$ to $\Barr_*(A)$ and then using isomorphisms $M\otimes_{A\otimes A^{\op}}(A\otimes s\overline{A}^{\otimes p}\otimes A) \cong M\otimes s\overline{A}^{\otimes p}$, we obtain the {\it normalized Hochschild chain complex} $C_{*}(A, M)$ of $A$ with coefficients in $M$, with $C_p(A, M)=\bigoplus_{i\in\Z_{\geq 0} }(M\otimes (s\overline{A})^{\otimes i})^p.$ Here $m\otimes s\overline{a}_{1,i} \in C_p(A, M)$ if and only if $|m|-i=p$, where by $|m|$ we mean the degree of $m$. The differential $b_p: C_p(A, M)\rightarrow C_{p-1}(A, M)$ is given by
\begin{eqnarray*}
b_p(m\otimes s\overline{a}_{1, i}\lefteqn{)=(-1)^mma_1\otimes s\overline{a}_{2, i}+}\\
&&\sum_{j=1}^{i-1} (-1)^{j+|m|} m\otimes s\overline{a}_{1, j-1}\otimes s\overline{a_ja_{j+1}}\otimes s\overline{a}_{j+2, i}+(-1)^pa_im\otimes s\overline{a}_{1, i-1}.
\end{eqnarray*}
The $p$-th {\it Hochschild homology} of $A$ with coefficients in $M$, denoted by $\HH_p(A, M)$, is defined as the homology group $\frac{\Ker(b_p)}{\Imm (b_{p+1)}}$ of $C_*(A, M)$. 

Since $\Barr_*(A)$ is a projective bimodule resolution of  $A$, it follows that $\HH^p(A, M)\cong \Ext_{A\otimes A^{\op}}^p(A, M)$ and $\HH_p(A, M)\cong \Tor_p^{A\otimes A^{\op}}(M, A)$.

\subsection{Gerstenhaber and Batalin-Vilkovisky algebras}\label{section-gerstenhaber}
In the 1960s when Gerstenhaber \cite{Ger} studied the deformation theory of algebras, he found that there is a rich structure on the Hochschild cochain complex $C^*(A, A)$. Besides the graded $k$-module structure, it has a differential graded (dg) associative algebra structure with the cup product $(g\cup f)(s\overline{a}_{1, m+n})=f(s\overline{a}_{1, m})g(s\overline{a}_{m+1,m+n} ),$ for $f\in C^m(A, A)$ and $g\in C^n(A, A)$. This cup product has a remarkable property: it is not (graded) commutative in $C^*(A, A)$, but the induced product on $\HH^*(A, A)$ is graded commutative. There is also a differential graded (dg) Lie algebra structure on $C^{*+1}(A, A)$ with the Gerstenhaber bracket defined as follows: for $f\in C^m(A, A)$ and $g\in C^n(A, A)$,
$$[f, g]:=f\circ g-(-1)^{(m-1)(n-1)} g\circ f$$ where
\begin{eqnarray*}
f\circ g(s\overline{a}_{1, m+n-1}):=\sum_{i=1}^m (-1)^{(i-1)(n-1)} f(s\overline{a}_{1, i-1} \otimes \overline{g}(s\overline{a}_{i, i+n-1}) \otimes s\overline{a}_{i+n, m+n-1}).
\end{eqnarray*}
Furthermore, the induced Gerstenhaber bracket in $\HH^*(A, A)$ satisfies the graded Leibniz rule with respect to the cup product. In summary, Gerstenhaber proved the following result.
\begin{thm}[\cite{Ger}]\label{thm-ger}
The Hochschild cohomology ring $\HH^*(A, A)$ is a {\lq\lq Gerstenhaber algebra\rq\rq}   in the following sense:
\begin{enumerate}[\upshape(i)]
\item $(\HH^*(A, A), \cup)$ is a graded commutative algebra with the unit $1\in \HH^0(A, A)$,
\item $(\HH^{*+1}(A, A), [\cdot, \cdot])$ is a graded Lie algebra,
\item The operations $\cup $ and $[\cdot, \cdot]$ are compatible through the (graded) Leibniz rule,
$$[f, g\cup h]=[f, g]\cup h+(-1)^{(m-1)n} g\cup [f, h],$$
where $f \in C^m(A, A)$ and $g\in C^n(A, A)$.
\end{enumerate}
\end{thm}

\begin{rem}\label{rem2.2}
In general, we call a graded $k$-module $G=\bigoplus\limits_{i\in \Z} G^i$, equipped with two operations $(\cup, [\cdot,\cdot]$) satisfying the above conditions (i), (ii) and (iii), {\it Gerstenhaber algebra}.  A nontrivial example is the Batalin-Vilkovisky algebra, motivated by quantum field theory.
\end{rem}

\begin{defin}
Let $V^{\bullet}=\bigoplus\limits_{n\in\Z} V^n$ be a graded commutative (associative) algebra. We say that $V^{\bullet}$, equipped with a differential $\Delta: V^{\bullet}\rightarrow V^{\bullet-1}$ of degree $-1$, is a {\it Batalin-Vilkovisky} (BV) algebra if the following conditions hold,
\begin{enumerate}
\item $\Delta(1)=0$ and $\Delta^2=0$,
\item for any $a\in V^m, b\in V^n$ and $c\in V^{\bullet}$,
\begin{eqnarray*}\lefteqn{\Delta(abc)=\Delta(ab)c+(-1)^ma\Delta(bc)+(-1)^{n(m-1)}b \Delta(ac)}\\
&&\ \ \ \ \  \ \ \  \ \ -\Delta(a)bc-(-1)^ma\Delta(b)c-(-1)^{m+n}ab\Delta( c).
\end{eqnarray*}
\end{enumerate}
\end{defin}
To each BV algebra, one can associate a graded Lie bracket $[\cdot, \cdot]$ as  the obstruction of $\Delta$ being a (graded) derivation with respect to the multiplication of $V^{\bullet}$. Explicitly, $[a, b]:=(-1)^m(\Delta(ab)-\Delta(a)b-(-1)^ma\Delta(b))$ for $a\in V^m$ and $b\in V^n$. This is called BV identity of $V^{\bullet}$. It follows from \cite[Proposition 1.2]{Get} that the graded commutative algebra $V^{\bullet}$, endowed with this Lie bracket $[\cdot, \cdot]$, is a Gerstenhaber algebra.

Almost all the examples  of BV algebras in literature (cf. e.g.  \cite{ChSu, Get}) are strongly inspired by quantum field theory and string theory. A typical example is the Hochschild cohomology $\HH^*(A, A)$ of a symmetric algebra $A$.
\begin{thm}[\cite{Tra, Men, Kau07}]\label{thm-bv2}
Let $A$ be a symmetric algebra. Then  $(\HH^*(A, A), \cup, [\cdot, \cdot])$ is a BV algebra whose BV operator $\Delta$ is  the dual of the Connes' B operator.
\end{thm}

Recall that a finite-dimensional algebra $A$ is {\it symmetric} if there is an associative, symmetric and non-degenerate bilinear form $\langle\cdot, \cdot \rangle: A\times A\rightarrow k$. More explicitly, $\langle a, bc\rangle=\langle ab, c\rangle, \langle a, b\rangle =\langle b, a \rangle$ for all $a, b, c\in A$, and the map $A\rightarrow D(A), a \mapsto \langle a, -\rangle$ from $A$ to  the $k$-linear dual $D(A)$ is an isomorphism.   Note that the pairing $\langle\cdot, \cdot\rangle$ on $A$  induces a graded pairing (still denoted by $\langle\cdot, \cdot\rangle$),
\begin{equation}\label{equation-pairing}
\langle\cdot, \cdot\rangle: C^*(A, A)\times C_*(A, A) \rightarrow k
\end{equation}
defined by
$ \langle f, a_0\otimes s\overline{a}_{1, m}\rangle:= \langle a_0, f(s\overline{a}_{1, m})\rangle,$
for any $f\in C^m(A, A)$ and $a_0\otimes s\overline{a}_{1, m}\in C_m(A, A)$. The BV operator $\Delta$ on $\HH^*(A, A)$ is determined by
$(-1)^{m}\langle \Delta(f)(s\overline{a}_{1, m}), a_0\rangle=\langle B(a_0\otimes s\overline{a}_{1, m}), f\rangle,
$ where $B$ is the Connes' $B$ operator defined by
\begin{equation}\label{B-operator}B(a_0\otimes s\overline{a}_{1,m})=\sum^{m+1}_{i=1} (-1)^{mi}1\otimes s\overline{a}_{i, m}\otimes s\overline{a_0}\otimes s\overline{a}_{1, i-1}.
\end{equation}


\subsection{Noncommutative differential forms}
There are several ways to define noncommutative differential forms of an associative $k$-algebra (not necessarily commutative). In the following, let us recall two of the (equivalent) definitions appeared in \cite{CuQu, Gin}.


The first definition  is originally due to Cuntz-Quillen \cite{CuQu}. Let $A$ be an $k$-algebra.   The noncommutative differential forms of $A$  is  the graded $k$-module $\Omega_{\nc}^{\bullet}(A):=\bigoplus\limits_{n\in \Z_{\geq 0}} A\otimes s\overline{A}^{\otimes n}$. There is a product on $\Omega_{\nc}^{\bullet}(A)$ defined by
\begin{eqnarray*}
(a_0\otimes s\overline{a}_{1, m})\lefteqn{(a_{m+1}\otimes s\overline{a}_{m+2, m+n+1})}\\
&:=\sum\limits_{i=0}^m(-1)^{m-i} a_0\otimes s\overline{a}_{1, i-1}\otimes s\overline{a_ia_{i+1}}  \otimes s\overline{a}_{i+2, m+n+1},
\end{eqnarray*}
where $a_0\otimes  s\overline{a}_{1, m}\in A\otimes s\overline{A}^{\otimes m}$ and $a_{m+1}\otimes  s\overline{a}_{m+2, m+n+1}\in A\otimes s\overline{A}^{\otimes n}$. It is clear that this product gives rise to a (graded) $A$-$A$-bimodule structure on $\Omega_{\nc}^{\bullet}(A)$. The left action is given by the multiplication of $A$ and the right action (denoted by $\blacktriangleleft$) is by $$(a_0\otimes s\overline{a}_{1, n})\blacktriangleleft a_{n+1}=\sum_{i=0}^n (-1)^{n-i}a_0\otimes s\overline{a}_{1, i-1}\otimes s\overline{a_ia_{i+1}} \otimes s\overline{a}_{i+2, n+1},$$ for any $ a_{n+1}\in A$ and $a_0\otimes \overline{a}_{1, n}\in A\otimes s\overline{A}^{\otimes n}$. There is a natural isomorphism of $A$-$A$-bimodules $\Omega_{\nc}^p(A)\otimes_A\Omega_{\nc}^q(A)\cong \Omega_{\nc}^{p+q}(A).$
The following lemma will be used frequently throughout this paper.
\begin{lem}\label{lem0}
For any $r, s\in \Z_{>0}$, the following identity holds in $\Omega_{\nc}^{r+s}(A)$.
\begin{eqnarray*}
 (a_0\otimes \lefteqn{s\overline{a}_{1, r+s-1})\blacktriangleleft a_{r+s}}\\
&=&(-1)^{s-1}(a_0\otimes s\overline{a}_{1, r})\blacktriangleleft a_{r+1}\otimes s\overline{a}_{r+2, r+s}\\
&&+ \sum_{i=1}^{s-1} (-1)^{s+i-1}a_0\otimes  s\overline{a}_{1, r+i-1}\otimes s\overline{a_{r+i}a_{r+i+1}}\otimes s\overline{a}_{r+i+2, r+s}.
\end{eqnarray*}
\end{lem}
\begin{proof}
This follows from a  straightforward computation.
\end{proof}

The other (equivalent) definition is as follows.   For any $p\in \Z_{\geq 0}$, denote by $\Omega_{\sy}^p(A)$ the cokernel $\coker( \Barr_{p+1}(A)\xrightarrow{d_p} \Barr_p(A))$  in the normalized bar resolution $\Barr_*(A)$ (cf. Section \ref{bar-resolution}). Clearly, $
 \Omega^p_{\sy}(A)$ is a graded $A$-$A$-bimodule concentrated in degree $-p$. Observe that $\Omega_{\sy}^0(A)\cong \Omega_{\nc}^0(A)=A$. Generally, we have the following result.

\begin{lem}\label{lemma1}
There is a natural isomorphism $\alpha: \Omega_{\sy}^{\bullet}(A) \rightarrow \Omega_{\nc}^{\bullet}(A)$ of graded $A$-$A$-bimodules.
\end{lem}

\begin{proof} For $p\in\Z_{\geq 0}$, we define the morphism $\alpha_p: \Omega_{\sy}^p(A) \rightarrow \Omega^p_{\nc}(A)$ to be the composition $\Omega_{\sy}^p(A) \hookrightarrow \Barr_{p-1}(A)=A\otimes s\overline{A}^{\otimes p-1} \otimes A\xrightarrow{\id_{A}\otimes \id_{s\overline{A}}^{\otimes (p-1)}\otimes \pi} A\otimes s\overline{A}^{\otimes p},$
where $\pi: A\rightarrow s\overline{A}$ is the canonical projection (of degree -1) and  $\id_V$ is the identity morphism of a $k$-module $V$.  It is straightforward  that $\alpha_p$ is invertible with the inverse $\alpha_p^{-1}: A\otimes s\overline{A}^{\otimes p} \xrightarrow{\iota_p} A\otimes s\overline{A}^{\otimes p}\otimes A\xrightarrow{d_p} \Omega^p_{\sy}(A)$, where the first morphism $\iota_p$ sends $x$ to $(-1)^px\otimes 1$. Thus it remains to show that $\alpha_p$ is an $A$-$A$-bimodule homomorphism. Indeed, we have that
\begin{eqnarray*}
\alpha_p^{-1}\lefteqn{(a_0(a_1\otimes s\overline{a}_{2, p})\blacktriangleleft a_{p+1})}\\
&=&\sum_{i=1}^p (-1)^{i}d_p(a_0a_1\otimes s\overline{a}_{2, i-1}\otimes \overline{a_ia_{i+1}}\otimes s\overline{a}_{i+2, p+1}\otimes 1)\\
&=& (-1)^pd_p(a_0a_1\otimes s\overline{a}_{2, p}\otimes a_{p+1})+d_p\circ d_{p+1}(a_0a_1\otimes  s\overline{a}_{2, p+1}) \otimes 1\\
&=&a_0 \alpha_p^{-1}(a_1\otimes  s\overline{a}_{2, p}) a_{p+1},
\end{eqnarray*}
where the last identity follows from $d^2=0$. This proves the lemma.
\end{proof}

Based on Lemma \ref{lemma1}, we will identify  $\Omega_{\sy}^{\bullet}(A)$ with $\Omega_{\nc}^{\bullet}(A)$ as graded $A$-$A$-bimodules.

\section{Tate-Hochschild cohomology}
Let $k$ be a field. We construct a cochain complex, called singular Hochschild cochain complex, for any associative $k$-algebra $A$.  The $i$-th cohomology group is isomorphic to  the morphism spaces from $A$ to $s^iA$ in the singularity category $\DD_{\sg}(A\otimes A^{\op})$.
\subsection{Singular Hochschild cochain complex}\label{subsection3.1}
Recall that  $\Omega_{\nc}^p(A)$ is a graded $A$-$A$-bimodule concentrated in degree $-p$. We consider a family, indexed by $p\in \Z_{\geq 0}$, of Hochschild cochain complexes $C^{*}(A, \Omega^p_{\nc}(A))$. 
For any $p\in \Z_{\geq 0}$, we define an embedding of cochain complexes (of degree zero), $$\theta_p: C^{*}(A, \Omega_{\nc}^p(A))\hookrightarrow C^{*}(A, \Omega_{\nc}^{p+1}(A)), f\mapsto f\otimes \id_{s\overline{A}}.$$ Here we recall that $C^m(A, \Omega_{\nc}^p(A))=\Hom((s\overline{A})^{\otimes m+p}, A\otimes (s\overline{A})^{\otimes p})$  for $m\in\Z$. 
\begin{lem}\label{lemma3.1}
$\theta_{p}\circ \delta=\delta\circ \theta_p.$\end{lem}
\begin{proof}
For  $f\in C^{m}(A, \Omega^p_{\nc}(A))$ and $n:=m+p$, we have that
\begin{eqnarray*}
\lefteqn{(-1)^{m+1}(\theta_p\circ \delta)(f)(s\overline{a}_{1, n+2})}\\
&=& a_1f(s\overline{a}_{2, n+1})\otimes s\overline{a}_{n+2}+\sum^{n}_{i=1} (-1)^i f(s\overline{a}_{1, i-1}\otimes  s\overline{a_ia_{i+1}}  \otimes s\overline{a}_{i+2, n+1})\otimes s\overline{a}_{n+2}\\
&& +(-1)^{n+1}f(s\overline{a}_{1, n}) \otimes s\overline{a_{n+1}a_{n+2}}+(-1)^{n+2} (f(s\overline{a}_{1, n})\otimes s\overline{a}_{n+1})\blacktriangleleft a_{n+2}\\
&=&(-1)^{m+1}\delta (\theta_p(f))(s\overline{a}
_{1, n+2}),
\end{eqnarray*}
where we used Lemma \ref{lem0} in the first identity.  This proves  $\theta_p\circ \delta=\delta\circ \theta_p$.
\end{proof}

\begin{defin}\label{defin3.2}
Let $A$ be an associative $k$-algebra. Then the {\it singular Hochschild cochain complex} of $A$, denoted by $C_{\sg}^*(A, A)$, is defined as the colimit of the inductive system in the category of cochain complexes of $k$-modules,
$$0\hookrightarrow C^*(A, A)\stackrel{\theta_0}{\hookrightarrow}  C^{*}(A, \Omega_{\nc}^1(A))\xhookrightarrow{\theta_1}\cdots \stackrel{\theta_{p-1}}{\hookrightarrow}  C^{*}(A, \Omega_{\nc}^p(A))\stackrel{\theta_p}{\hookrightarrow} \cdots.$$
Namely, $C_{\sg}^*(A, A):=\colim_{\theta_p} C^{*}(A, \Omega_{\nc}^p(A))$. Its cohomology groups are  denoted by $\HH_{\sg}^*(A, A)$. \end{defin}
\begin{rem}\label{rem3.3}
Since the map  $\theta_p$ is injective for any $p\in \Z_{\geq 0}$, there is a (bounded below) filtration of cochain complexes of $C_{\sg}^*(A, A)$,
$$0\subset C^*(A, A)\subset\cdots \subset C^{*}(A, \Omega^p_{\nc}(A))\subset\cdots \subset C_{\sg}^*(A, A). $$
This yields a natural map  $\rho: \HH^*(A, A)\rightarrow \HH_{\sg}^*(A, A)$.  In the following, we will see that $\rho$ is in fact a morphism of Gerstenhaber algebras (cf. Corollary \ref{cor5.3}). In \cite{RiWa}, we generalized  the definition of $C_{\sg}^*(A, A)$ to any dg associative algebra $A$,  in order to understand the relevance of the algebraic structures discussed in this paper in symplectic geometry and string topology.  \end{rem}
In \cite{JoSt} the authors formalize the use of graphs in tensor categories.  Morphisms in a tensor category are presented by graphs,  and operations (e.g. compositions and tensor products) on morphisms are presented by operations (e.g. gratings and unions) on graphs. Since the category $k$-$\modu$ of $k$-modules, which we are basically working on in this paper, is particularly a tensor category with the tensor product $\otimes_k$, we will use graphs to present morphisms and operations  in $k$-$\modu$.  For more details on  graph theory,  one may refer to \cite{JoSt, Kau05, Kau07a}.

\begin{figure}
\centering
  \includegraphics[width=120mm]{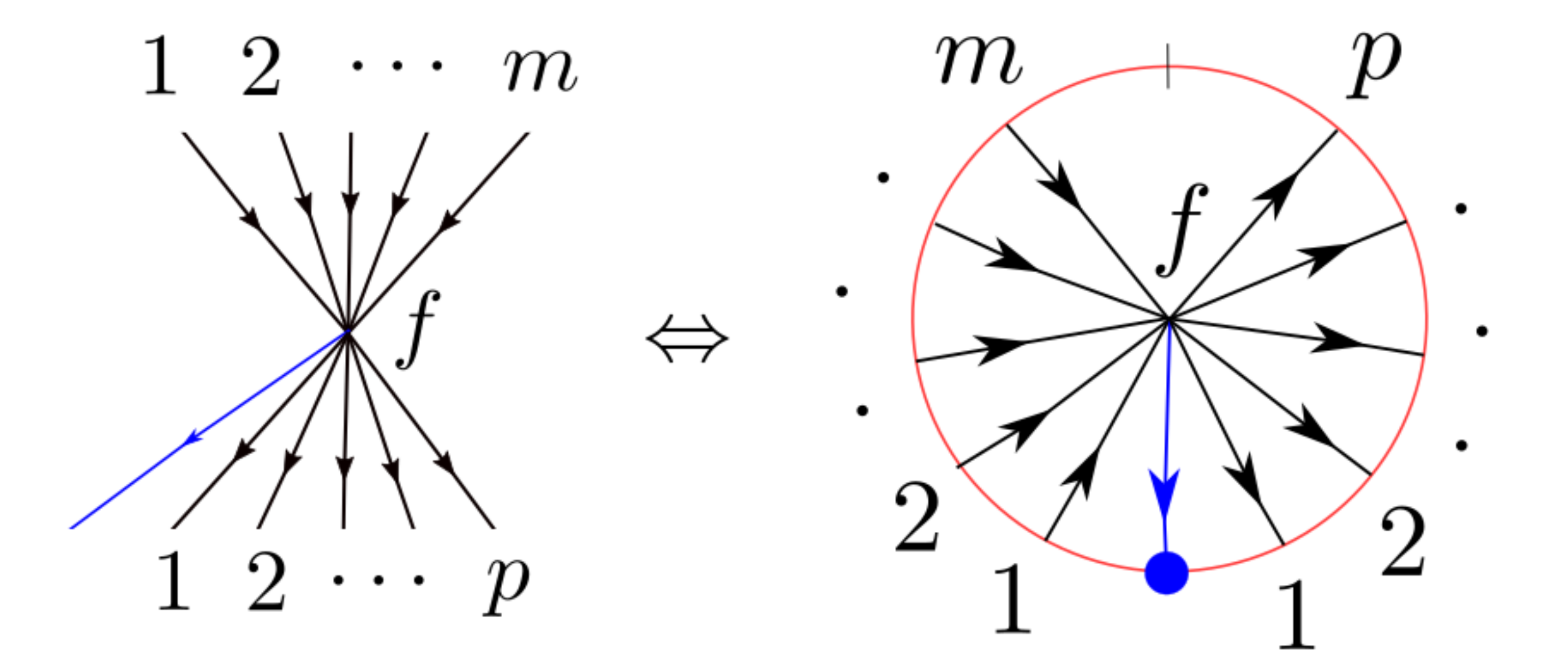}
  \caption{Two types of  graphic presentations of  $f\in C^{m-p}(A,   A\otimes s\overline{A}^{\otimes p}).$ The left one is  {\it treelike} presentation and the right one is {\it cactus-like} presentation. }
  \label{Tree-presentation}
\end{figure}

Figure \ref{Tree-presentation} illustrates two types (tree-like and cactus-like) of graphic presentations of  $f\in C^{m-p}(A,  A\otimes (s\overline{A})^{\otimes p})$. The tree-like presentation is the usual graphic presentation of morphisms in tensor categories used in \cite{JoSt}. We read the graph from top to bottom and  left to right. The inputs $(s\overline{A})^{\otimes m}$ are ordered from left to right at the top, while the outputs $A\otimes (s\overline{A})^{\otimes p}$ are ordered in the same way at the bottom. We use the color blue to distinguish the special output $A$. The orientations of edges are from top to bottom. To study the $B_{\infty}$-algebra structure on $C_{\sg}^*(A, A)$ (cf. Section \ref{subsection-brace}),  we also need to use the cactus-like presentation. An element $f\in C^{m-p}(A, \Omega_{\nc}^p(A))$ is presented as follows.
\begin{proc}\label{proc3.4}
First, the image  of $0\in \R$ in $S^1: = \R/\Z$ is decorated by a blue dot (cf. Figure \ref{Tree-presentation}). We call the image zero point of $S^1$. The blue radius pointing towards the dot represents the special output $A$. Then the inputs $(s\overline{A})^{\otimes m}$ are indicated by $m$ (black) radii on the left semicircle pointing towards the center of $S^1$ in clockwise, and finally the other outputs are indicated by $p$ radii on the right semicircle pointing outwards the center of $S^1$ in  counterclockwise.
\end{proc}


 By Definition \ref{defin3.2}, two elements $f\in C^{m}(A, \Omega_{\nc}^{p_1}(A))$ and $g\in C^{m}(A, \Omega_{\nc}^{p_2}(A))$ for $p_1\geq p_2$ represent the same element $[f]=[g]$ in $C_{\sg}^m(A, A)$ if and only if $f=g\otimes \id_{s\overline{A}}^{\otimes (p_2-p_1)}$, as depicted in  Figure \ref{theta-map}. This allows us to add (or remove) some vertical lines on the right of the tree-like graph or chords on the circle. 
  In the following sections, we will see that the two types of graphic presentations have different advantages. It is  easier to write down the corresponding morphisms from tree-like presentations, while   it is  more convenient to construct operations on $C_{\sg}^*(A, A)$ using cactus-like presentations.

\begin{figure}
\centering
  \includegraphics[width=120mm]{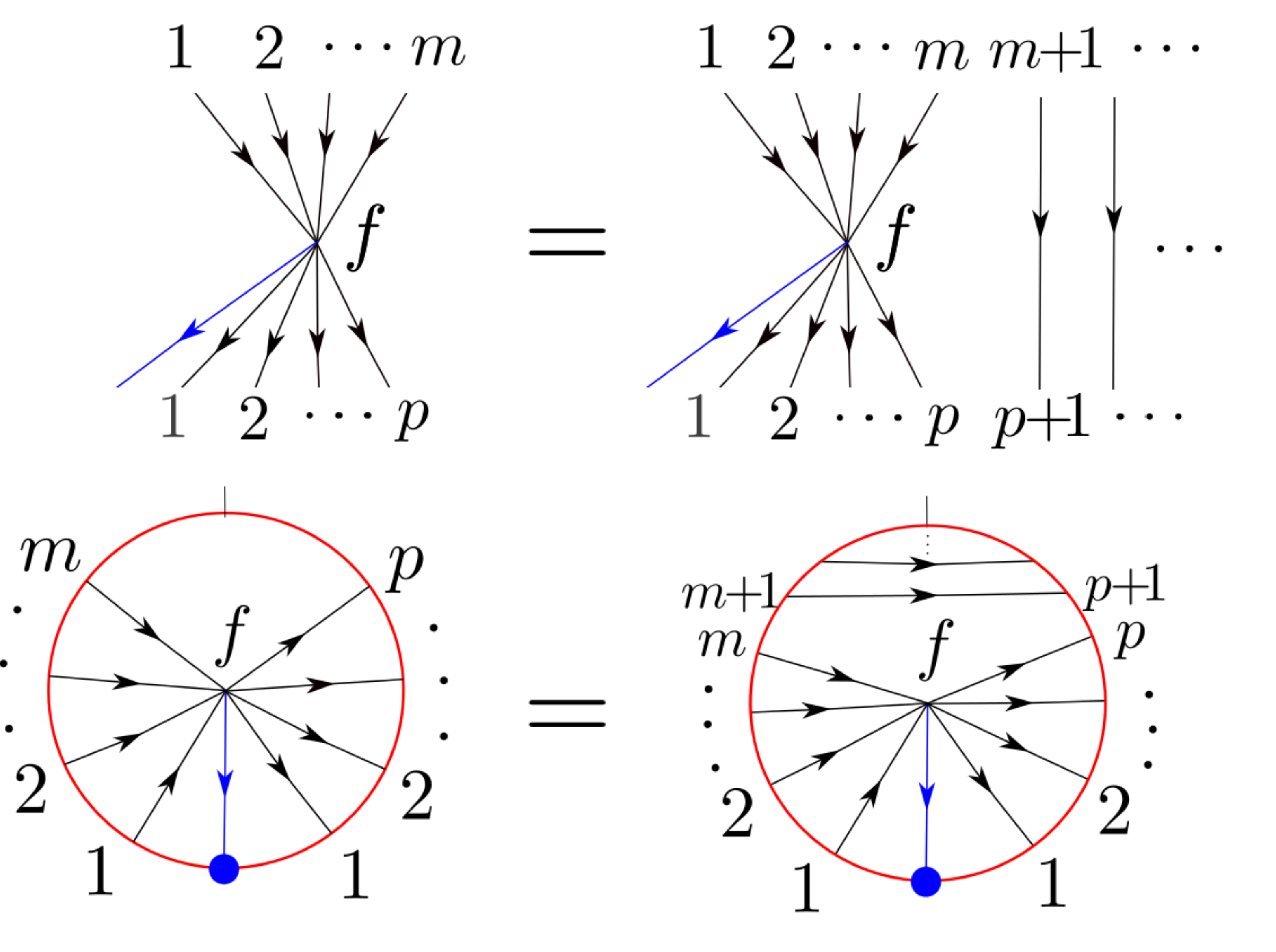}
  \caption{The representatives of $f\in C_{\sg}^{m-p}(A, A)$. The vertical lines on the upper right treelike graph represent identities $\id_{s\overline{A}}$ of $s\overline{A}$. Similarly, the identities are also represented by horizontal chords on the lower right circle. }\label{theta-map}
\end{figure}

\subsection{Relationship with singularity category}
In this section, we fix a (both left and right) Noetherian algebra $A$ over a field $k$. Let $\DD^b(A)$ be the bounded derived category of finitely generated left $A$-modules. Let $\Perf(A)$ denote the full subcategory consisting of those complexes which are quasi-isomorphic to bounded complexes of finitely generated projective $A$-modules. Then the {\it singularity category} $\DD_{\sg}(A)$ is defined as the Verdier quotient of the triangulated category $\DD^{b}(A)$ by $\Perf(A)$.

\begin{rem}
The notion of singularity category was introduced by Buchweitz in an unpublished manuscript \cite{Buc}. He proved that the singularity category $\DD_{\sg}(A)$ is triangle equivalent to the stable category $\underline{\MCM}(A)$ of maximal Cohen-Macaulay  modules when the algebra $A$ is Gorenstein.  Later, Orlov \cite{Orl04} independently rediscovered a global version of singularity category motivated by homological mirror symmetry.

Buchweitz in the same manuscript,  provided a general framework for Tate cohomology. Let $M, N$ be two modules over a Gorenstein algebra $S$. The $i$-th {\it Tate cohomology} group of $M$ with values in $N$ is defined as  $\Hom_{\DD_{\sg}(S)}(M, s^iN)$. In loc. cit. Buchweitz denoted it by $\underline{\Ext}_S^i(M, N).$  Clearly, this notion generalizes the Tate cohomology of finite groups. Under this framework, it is  natural to define  {\it Tate-Hochschild cohomology groups} as $\underline{\Ext}^*_{A\otimes A^{\op}}(A, A)$  for a Noetherian algebra $A$, compared with Hochschild cohomology. 
\end{rem}

\begin{thm}\label{thm1}
Let $A$ be a Noetherian $k$-algebra.
Then there exists a natural isomorphism $\Phi:  \HH_{\sg}^*(A, A) \rightarrow \underline{\Ext}_{A\otimes A^{\op}}^*(A, A)$.
\end{thm}
\begin{proof}
First, let us fix an integer $m\in\Z$. From the fact that the colimit commutes with the cohomology functor in the category of  cochain complexes, it follows that
\begin{equation}\label{equation1}
\HH_{\sg}^m(A, A)\cong \colim_{H^m(\theta_p)} \HH^{m}(A, \Omega_{\nc}^p(A)).
\end{equation}
We observe that the  map $H^m(\theta_p): \HH^{m}(A, \Omega_{\nc}^p(A))\rightarrow \HH^{m}(A, \Omega^{p+1}_{\nc}(A))$ coincides with the connecting morphisms in the long exact sequence $$\cdots\rightarrow \HH^{m}(A, \Barr_p(A))\rightarrow \HH^{m}(A, \Omega_{\nc}^p(A))\rightarrow \HH^{m+1}(A, s^{-1}\Omega_{\nc}^{p+1}(A))\rightarrow\cdots $$  induced by the short exact sequence $0\rightarrow s^{-1}\Omega^{p+1}_{\nc}(A)\rightarrow \Barr_p(A)\rightarrow \Omega_{\nc}^p(A)\rightarrow 0$. Here, we identify $\Omega_{\nc}^p(A)$ with $\Omega_{\sy}^p(A)$ by Lemma \ref{lemma1} . Note that there is a natural isomorphism between $\HH^{m+1}(A, s^{-1}\Omega_{\nc}^{p+1}(A))$ and $\HH^m(A, \Omega^{p+1}_{\nc}(A))$.

From \cite[Corollary 3.3]{Bel} and \cite{Buc}, it follows that
\begin{equation}\label{equation2}
\underline{\Ext}^m_{A\otimes A^{\op}}(A, A) \cong \colim_{\theta_p'} \underline{\Hom}_{A\otimes A^{\op}}(s^{-p-m}\Omega_{\sy}^{p+m}(A), s^{-p}\Omega_{\sy}^p(A))
\end{equation}
where  $\underline{\Hom}_{A\otimes A^{\op}}$ represents the morphism spaces in the stable category $A\otimes A^{\op}$-$\underline{\modu}$ of $A$-$A$-bimodules; the map $\theta_p'$ is induced by the fact that $A\otimes A^{\op}$-$\underline{\modu}$ is a left triangulated category with left shift functor the syzygy functor $\Omega^1_{\sy}$.
Combining the isomorphisms (\ref{equation1}) and (\ref{equation2}), it is sufficient to show that  $\colim_{H^m(\theta_p)} \HH^{m}(A, \Omega_{\nc}^p(A))$  is isomorphic to  $\colim_{\theta_p'} \underline{\Hom}_{A\otimes A^{\op}}(s^{-p-m}\Omega_{\sy}^{p+m}(A), s^{-p}\Omega_{\sy}^p(A)).$ First, let us define a morphism between them. Note that there is a canonical map
$$\Phi_{m, p}: \HH^{m}(A, \Omega_{\nc}^p(A))\rightarrow  \underline{\Hom}_{A\otimes A^{\op}}(s^{-p-m}\Omega_{\sy}^{p+m}(A), s^{-p}\Omega_{\sy}^p(A))$$
since  any cocycle $f\in C^{m}(A, \Omega_{\nc}^p(A))$ can be represented by an $A$-$A$-bimodule morphism $\overline{f}: \Omega_{\sy}^{m+p}(A)\rightarrow \Omega_{\sy}^{p}(A)$ and any coboundary factors through (see Diagram \ref{equation3}) the projective $A$-$A$-bimodule $\Barr_{m+p-1}(A)$.
\begin{equation}\label{equation3}
\xymatrix@R=1pc{
\Barr_{m+p+1}(A)\ar[d]_{d_{m+p+1}} &\\
\Barr_{m+p}(A) \ar[d]_{d_{m+p}}\ar[r]^-{f}&  \Omega_{\sy}^p(A) \\
\Barr_{m+p-1}(A)\ar@{.>}[ru] &\\
}
\end{equation}
We observe that  both of the  two maps $H^m(\theta_p)$ and $\theta_p'$ (cf. (\ref{equation1}) and (\ref{equation2})) correspond to the same lifting from the bottom horizontal maps $f$ to the top horizontal maps $\widehat{f}$:
\begin{equation}\label{equation4}
\xymatrix@C=3pc@R=1.5pc{
\Omega_{\sy}^{m+p+1}(A) \ar@{_(->}[d]\ar[r]^{\widehat{f}} &\Omega_{\sy}^{p+1}(A)\ar@{_(->}[d]\\
\Barr_{m+p}(A)\ar@{->>}[d]\ar@{.>}[r]&\Barr_{p}(A)\ar@{->>}[d]\\
\Omega_{\sy}^{m+p}(A) \ar[r]^-{f}&  \Omega_{\sy}^p(A) \\
}
\end{equation}
where $\widehat{f}$ is given by
$\widehat{f}(a_0\otimes s\overline{a}_{1, m+p})=f(a_0\otimes \overline{a}_{1, m+p-1})\otimes s\overline{a_{m+p}}.$
Again we use the identification of $\Omega_{\sy}^p(A)$ with $\Omega^p_{\nc}(A)$ by Lemma \ref{lemma1}. Therefore we get that the maps $\Phi_{m, *}$ are compatible with the colimit constructions  and then we have a canonical map
$$\Phi_m: \colim_{\theta_p} \HH^{m}(A, \Omega_{\nc}^p(A))\rightarrow \colim_{\theta_p'} \underline{\Hom}_{A\otimes A^{\op}}(s^{-p-m}\Omega_{\sy}^{p+m}(A), s^{-p}\Omega_{\sy}^p(A)).$$

Claim that $\Phi_m$ is surjective. Indeed, assuming   $$f\in  \colim_{\theta'_p} \underline{\Hom}_{A\otimes A^{\op}}(s^{-p-m}\Omega_{\sy}^{p+m}(A), s^{-p}\Omega_{\sy}^p(A)),$$ then there exists $p_0\in \Z_{\geq 0}$ such that $f$  can be represented by a certain element $f'\in \Hom_{A\otimes A^{\op}}(\Omega_{\sy}^{m+p_0}(A), \Omega_{\sy}^{p_0}(A)).$ Thus we obtain a Hochschild cocyle $\alpha:=f'\circ d_{m+p_0}\in \Hom_{A\otimes A^{\op}}(\Barr_{m+p_0}(A), \Omega^{p_0}_{\sy}(A)).$ By the definition of $\Phi_m$, we have  $\Phi_m(\alpha)=f$. This proves that  $\Phi_m$ is surjective.

It remains to show that $\Phi_m$ is injective. Suppose   $\beta\in \colim_{H^m(\theta_p)} \HH^{m}(A, \Omega_{\nc}^p(A))$ such that $\Phi_m(\beta)=0$. Since $\Phi_m$ is surjective, $\beta$ can be represented by an element $\beta'\in \HH^{m}(A, \Omega_{\sy}^{p_0}(A))$ for some $p_0$ such that $\Phi_{m, p_0}(\beta')=0$.  This means that
$\Phi_{m, p_0}(\beta')$ factors through the differential $d_{p_0}: \Barr_{p_0}(A)\twoheadrightarrow \Omega_{\sy}^{p_0}(A)$.\begin{equation}\label{equation5}
\xymatrix@C=4pc{
\Omega_{\sy}^{m+p_0}(A) \ar[r]^-{\Phi_{m, p_0}(\beta')} \ar[rd]_{\sigma}& \Omega_{\sy}^{p_0}(A)\\
& \Barr_{p_0}(A)\ar[u]_{d_{p_0}}
}
\end{equation}
Consider the long exact sequence
\begin{equation*}
\xymatrix@C=2pc{
\cdots\ar[r]&\HH^{m}(A, \Barr_{p_0}(A)) \ar[r]^-{d_{p_0}^*}& \HH^{m}(A, \Omega_{\sy}^{p_0}(A)) \ar[r]\ar[r]^-{H^m(\theta_{p_0})}& \HH^{m+1}(A, s^{-1}\Omega^{p_0+1})\ar[r] &\cdots}
\end{equation*}
From Diagram (\ref{equation5}), it follows  that  $[\sigma]\in \HH^m(A, \Barr_{p_0}(A))$ and  $$d_{p_0}^*([\sigma])=\beta'\in \HH^{m}(A, \Omega_{\sy}^{p_0}(A)).$$
Then  $0=H^m(\theta_{p_0})d_{p_0}^*([\sigma])=H^m(\theta_{p_0})(\beta')$, thus $\beta=0$ in $\colim_{H^m(\theta_p)} \HH^{m}(A, \Omega_{\nc}^p(A))$. Therefore $\Phi_m$ is injective. This proves the theorem.
\end{proof}
\begin{rem}
From the proof, we have  the following commutative diagram.
\begin{equation}\label{equation7}
\xymatrix{
\Ext_{A\otimes A^{\op}}^*(A, A) \ar[r]^{\rho'} & \underline{\Ext}_{A\otimes A^{\op}}^*(A, A)\\
\HH^*(A, A)\ar[r]^-{\rho}\ar[u]^{\cong} & \HH_{\sg}^*(A, A)\ar[u]_-{\Phi_*}
}
\end{equation}
where $\rho'$ is induced by the quotient functor from the bounded derived category $\DD^b(A\otimes A^{\op})$ to the singularity category $\DD_{\sg}(A\otimes A^{\op})$.
\end{rem}

\section{Gerstenhaber algebra structure}\label{section4}
 In this and the next section, we will prove  that there is a  Gerstenhaber algebra structure on $\HH_{\sg}^*(A, A)$ to make the natural map $\rho: \HH^*(A, A)\rightarrow \HH_{\sg}^*(A, A)$ into a morphism of Gerstenhaber algebras. 

\subsection{Cup product}
For any $m, n, p, q\in \Z_{\geq 0}$,
the {\it cup product}
\begin{equation}\label{equation-cup-product}
\cup: C^{m-p}(A, \Omega_{\nc}^p(A))\otimes C^{n-q}(A, \Omega_{\nc}^q(A))\rightarrow C^{m+n-p-q}(A, \Omega_{\nc}^{p+q}(A))
\end{equation}
is defined by the following formula,
$$f\cup g:=\left(\mu\otimes \id_{s\overline{A}}^{\otimes p+q}\right)\left(\id_{A}\otimes f\otimes \id_{s\overline{A}}^{\otimes q}\right)  \left(g\otimes \id_{s\overline{A}}^{\otimes m}\right),
$$
for any $f\in C^{m-p}(A, \Omega_{\nc}^p(A))$ and $g\in C^{n-q}(A, \Omega_{\nc}^q(A))$. Here  $\id_{s\overline{A}}$ is  the identity morphism of $s\overline{A}$.  When $p=q=0$, we recover the cup product on $C^*(A, A)$ (cf. Section \ref{section-gerstenhaber}). The cup product can be depicted by the treelike or cactus-like presentation (cf. Figure \ref{Cup-product}).

\begin{lem}\label{lemma4.1}
For any $f\in C^{m-p}(A, \Omega_{\nc}^p(A))$ and $g\in C^{n-q}(A, \Omega_{\nc}^q(A)),$ we have
$$\delta(f\cup g)=\delta(f)\cup g +(-1)^{m-p} f\cup \delta(g).$$
\end{lem}
\begin{proof} Without loss of  generality, we may  assume that $m\geq q$. Then we have
\begin{eqnarray*}
\lefteqn{(-1)^{\epsilon}\left(\delta(f\cup g)-(-1)^{m-p}f\cup \delta(g)\right)(s\overline{a}_{1, m+n+1})}\\
&=&\sum_{i=n+1}^{m+n} \sum_j (-1)^ic_0^j f \left(s\overline{c}^j_{1, q}\otimes s\overline{a}_{n+1, m+n-q}\right)\otimes \cdots \otimes s\overline{a_ia_{i+1}} \otimes s\overline{a}_{i+2, m+n+1}\\
&& +(-1)^{n}\sum_j \left(\mu\otimes \id_{s\overline{A}}^{\otimes p+q}\right)\left(\id_{A}\otimes f \otimes \id_{s\overline{A}}^{\otimes q}\right)\left(\left(c_0^j \otimes s\overline{c}_{1, q}^j\right)\blacktriangleleft  a_{n+1}\otimes s\overline{a}_{n+2, m+n+1}\right)\\
&& +(-1)^{m+n+1} \sum_j \left(c_0^j f\left(s\overline{c}_{1, q}^j\otimes s\overline{a}_{n+1, m-q}\right)\otimes s\overline{a}_{m-q+1, m+n}\right)\blacktriangleleft a_{m+n+1}
\end{eqnarray*}
where $\epsilon=m+n-p-q+1$ and $g(s\overline{a}_{1, n}):=\sum_j c_0^j\otimes s\overline{c}_{1, q}^j$.  Then it follows from Lemma \ref{lem0}  that the right hand side of the above identity equals to $(-1)^{\epsilon}\delta(f)\cup g(s\overline{a}_{1, m+n+1})$. Therefore, $\delta(f\cup g)=\delta(f)\cup g+(-1)^{m-p}f\cup \delta(g)$.
\end{proof}
Since $\theta_{p}(f)\cup g=f\cup \theta_{q}(g)=\theta_{p+q}(f\cup g),$ the cup product (still denoted by $\cup$) is well-defined on $C_{\sg}^*(A, A)$.
\begin{prop}\label{prop4.2}
The complex $(C_{\sg}^*(A, A), \delta)$, equipped with the cup product $\cup$, is a dg (unital associative) algebra.
\end{prop}
\begin{proof}
Since  the cup product is associative, this proposition follows from Lemma \ref{lemma4.1}.
\end{proof}

\begin{figure}
\centering
  \includegraphics[width=140mm]{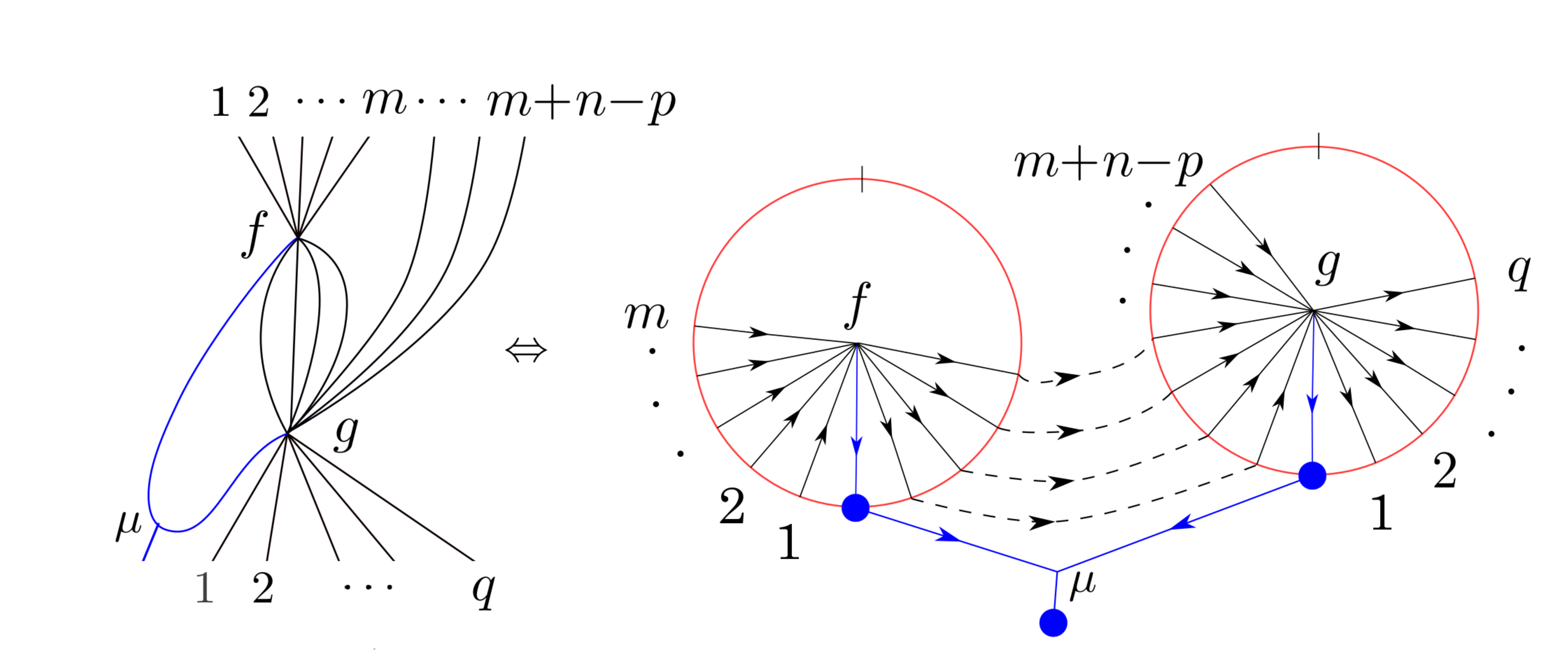}
  \caption{Cup product $g\cup f$ in $C_{\sg}^*(A, A)$ for $f\in C^{m-p}_{\sg}(A, A), g\in C^{n-q}_{\sg}(A,A)$.  For simplicity, the orientation arrows (from top to bottom) in the tree-like presentation are omitted. In the cactus-like presentation, by the blue arrows connecting blue radii, we mean the multiplication of $A$.}
  \label{Cup-product}
\end{figure}

\begin{figure}
\centering
  \includegraphics[width=140mm]{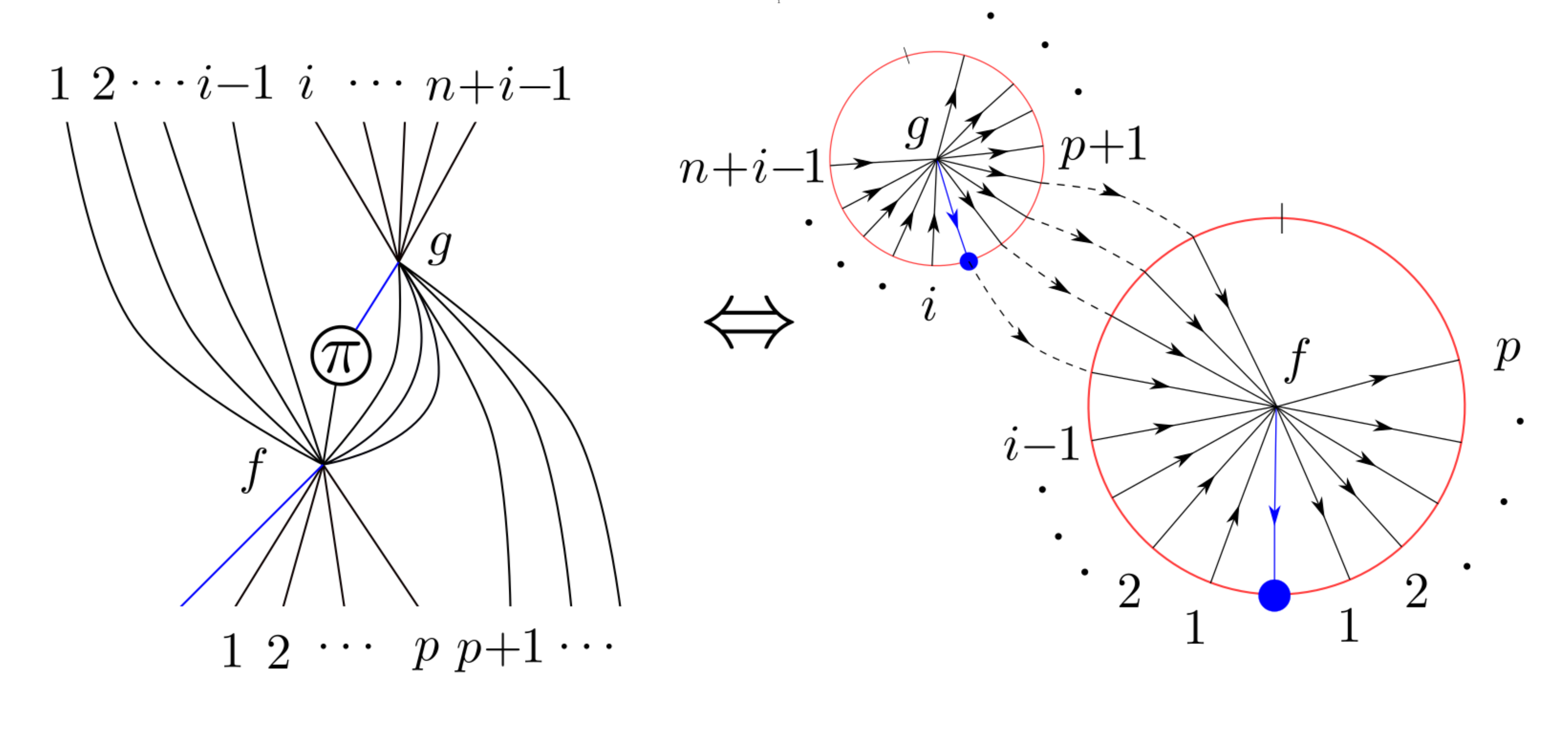}
  \caption{The treelike and cactus-like presentations of $f\circ_{i} g$ for $1\leq i\leq m$. For simplicity, the projection $\pi: A \rightarrow s\overline{A}$ is omitted in the cactus-like presentation. 
  }
    \label{Lie-bracket}
\end{figure}

\subsection{Lie bracket}
For any $m, n, p, q\in \Z_{\geq 0},$ we define  the {\it Lie bracket}
\begin{equation}\label{equation-Lie-brakcet}
[\cdot, \cdot]: C^{m-p}(A, \Omega_{\nc}^p(A))\otimes C^{n-q}(A, \Omega_{\nc}^q(A))\rightarrow C^{m+n-p-q-1}(A, \Omega_{\nc}^{p+q}(A)).
\end{equation}
as follows. For any $f\in C^{m-p}(A, \Omega_{\nc}^p(A))$ and $g\in C^{n-q}(A, \Omega^q_{\nc}(A))$, denote
\begin{equation*}
f\circ_i g:=
\begin{cases}
(f\otimes\id_{s\overline{A}}^{\otimes q})(\id_{s\overline{A}}^{\otimes i-1} \otimes \overline{g}\otimes \id_{s\overline{A}}^{\otimes m-i}) & \mbox{ for $1\leq i\leq m$}, \\
(\id_A\otimes \id_{s\overline{A}}^{\otimes -i-1} \otimes \overline{g} \otimes \id_{s\overline{A}}^{\otimes p+i}) (f\otimes \id_{s\overline{A}}^{\otimes n-1}) & \mbox{for $-p\leq i\leq -1$},
\end{cases}
\end{equation*}
where $\overline{g}:=(\pi\otimes \id_{s\overline{A}}^{\otimes q}) g$ and $\pi:A\rightarrow s\overline{A}$ is the canonical projection of degree $-1$. We set
$$f\circ g:=\sum_{i=1}^m (-1)^{(n-q-1)(i-1)} f\circ_i g-\sum_{i=1}^p (-1)^{(n-q-1)(i-m-p-1)}f\circ_{-i}g.$$ The Lie bracket is given by $[f,g]:=f\circ g-(-1)^{(m-p-1)(n-q-1)} g\circ f.$
When $p=q=0$, the Lie bracket $[\cdot, \cdot]$ coincides with the classical Gerstenhaber bracket (cf. Section \ref{section-gerstenhaber}) on $C^*(A, A)$.
Since $\theta_p(f)\circ g=f\circ \theta_q(g)=\theta_{p+q}(f\circ g)$, the circle product is well-defined on $C_{\sg}^*(A, A)$  and so is the Lie bracket $[\cdot, \cdot]$. Figures \ref{Lie-bracket} and \ref{Lie-bracket1} illustrate the graphic presentations of  the circle product. Clearly, $[\cdot, \cdot]$ is graded skew-symmetric:  $[f, g]=-(-1)^{(|f|-1)(|g|-1)}[g, f]$.  We observe that the differential $\delta$ of $C_{\sg}^*(A, A)$ can be expressed by the Lie bracket $[\cdot, \cdot]$ and the multiplication $\mu$ of $A$, namely,
\begin{equation}\label{equation-delta}
\delta(f)=[\mu, f],
\end{equation}
for any $f\in C_{\sg}^*(A, A)$.
\begin{rem}\label{remark4.5-mu}
Readers may note that the multiplication $\mu$ is not in $C^2(A, A)$. Recall that we have a natural projection $\pi: A\twoheadrightarrow \overline{A}$. We take a $k$-linear split injection $\iota: \overline{A}\hookrightarrow A$ such that $\pi \iota=\id_{\overline{A}}$. Denote  $\overline{\mu}:=\mu (\iota\otimes \iota)$ in $C^2(A, A)$. Then we have $\delta(f)=[\overline{\mu}, f]$  since $\delta(f)$ is  independent on the choice of the split injection $\iota$.  By abuse of notation, we write $\delta(f)=[\mu, f]$.
\end{rem}

\begin{figure}
\centering
  \includegraphics[width=140mm]{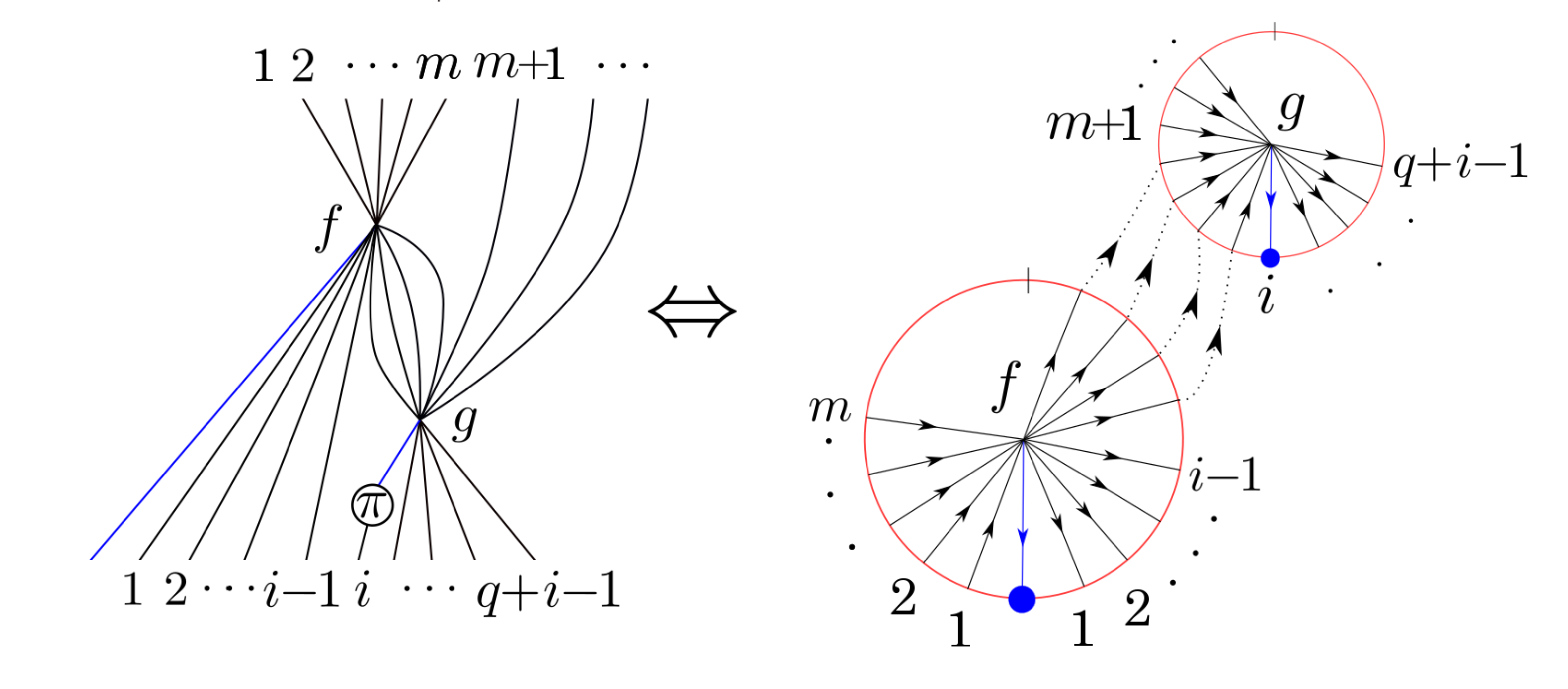}
  \caption{The treelike and cactus-like  presentations of $f\circ_{-i} g$ for $1\leq i\leq p$.  }
    \label{Lie-bracket1}
\end{figure}

\begin{prop}\label{prop4.4}
For any $f\in C^{m-p}(A, \Omega_{\nc}^p(A))$ and $g\in C^{n-q}(A, \Omega_{\nc}^q(A))$, we have
$$ (-1)^{m-p-1} f\circ \delta(g)-\delta(f\circ g)+\delta(f)\circ g=(-1)^{m-p-1}(f\cup g-(-1)^{(m-p)(n-q)} g\cup f). $$
\end{prop}
\begin{proof}  Firstly, let us denote the left hand side of the above identity  by $B(f, g)$. Set $r:=m-p-1$ and $s:=n-q-1$. For $i>0$,  we denote  
\begin{equation*}
\begin{split}
B^{>0}_i(f, g)&:=(-1)^{(i-1)(s-1)+r}f\circ_i\delta(g)-(-1)^{(i-1)s}\delta(f\circ_i g)+(-1)^{(i-1)s}\delta(f)\circ_i g,\\
B^{<0}_{-i}(f, g)&:=(-1)^{(s-1)(i-r)+r}f\circ_{-i} \delta(g)-(-1)^{s(i-r)}\delta(f\circ_{-i} g)+(-1)^{s(i-r-1)} \delta(f)\circ_{-i} g.
\end{split}
\end{equation*}
 Here  $f\circ_ig: =0$ if it is not well-defined. Clearly, we have   $$B(f, g)=\sum_{i=1}^{m+1} B^{>0}_i(f, g)-\sum_{i=1}^{p}B_{-i}^{<0}(f, g).$$

Secondly, we deal with the first term $\sum\limits_{i=1}^{m+1} B^{>0}_i(f, g)$.
For $0\leq i\leq m$, set 
\begin{eqnarray*}
C_i^{>0}(f, g\lefteqn{):=(-1)^{i(s-1)+r-1}}\\
&&\left(\mu\otimes \id_{s\overline{A}}^{\otimes p+q }\right)\left(\id_{A}\otimes f\otimes \id_{s\overline{A}}^{\otimes q}\right) \left((1\otimes s\overline{a}_{1, i})\blacktriangleleft g(\overline{a}_{i+1, i+n})\otimes \overline{a}_{i+n+1, m+n}\right).
\end{eqnarray*}
From a straightforward computation, we get $B_{i}^{>0}(f, g)=C_{i}^{>0}(f, g)-C_{i-1}^{>0}(f, g)$  for $1\leq i\leq m$.  Since $B_{m+1}^{>0}(f, g)=(-1)^{ms}\delta(f)\circ_{m+1} g$, we have
\begin{equation}\label{equation4.1}
\begin{split}
\sum_{i=1}^{m+1} B_i^{>0}(f, g)=(-1)^{ms}\delta(f)\circ_{m+1} g +C_{m}^{>0}(f, g)-C_0^{>0}(f, g),
\end{split}
\end{equation}

Finally, we need to simplify the second term $\sum\limits_{i=1}^{p}B_{-i}^{<0}(f, g)$. For $0\leq i\leq p$, we set 
\begin{eqnarray*}
\lefteqn{C_i^{<0}(f, g):=
(-1)^{(s-1)(i-r-1)+r-1}}\\
&&\sum_j\left(c^j_0\otimes \overline{c^j}_{1, i}\right) \blacktriangleleft \left(g\otimes \id_{s\overline{A}}^{\otimes p+q-i }\right) \left( s\overline{c^j}_{i+1, p}\otimes s\overline{a}_{m+1, m+n}\right),
\end{eqnarray*}
where $f(s\overline{a}_{1, m}):=\sum_j c_0^j\otimes \overline{c^j}_{1, p}.$ For $1\leq i\leq p$, we have $B_{-i}^{<0}(f, g)=C^{<0}_{i}(f, g)-C_{i-1}^{<0}(f, g)$. Thus we get
\begin{equation}\label{equation4.2}
\sum_{i=1}^{p}B_{-i}^{<0}(f, g)=C_{p}^{<0}(f, g) -C_0^{<0}(f, g).
\end{equation}
Since
$C_0^{>0}(f,g)=(-1)^{r-1} (f\cup g)(s\overline{a}_{1, m+n})$ and $C_0^{<0}(f, g)=(-1)^{s(r-1)}(g\cup f)(s\overline{a}_{1, m+n}),$
combining (\ref{equation4.1}) and (\ref{equation4.2}), we obtain
\begin{eqnarray}\label{equation4.3}
B(f, g)\lefteqn{=(-1)^{r} (f\cup g-(-1)^{(s-1)(r-1)} g\cup f)(s\overline{a}_{1, m+n})}\nonumber\\
&&+(-1)^{ms}\delta(f)\circ_{m+1} g +C_{m}^{>0}(f, g)-C_{p}^{<0}(f, g).
\end{eqnarray}
From (\ref{equation4.3}),  it is enough to verify  $(-1)^{ms}\delta(f)\circ_{m+1} g =C_p^{<0}(f,g)-C_{m}^{>0}(f, g).$
This identity follows from a straightforward computation. 
This proves the proposition.
\end{proof}

\begin{cor}\label{cor-graded-algebra}
The cup product $\cup$ in $\HH_{\sg}^*(A, A)$ is graded commutative.
\end{cor}

\begin{proof}
It is a direct consequence of Proposition \ref{prop4.4}.
\end{proof}

\begin{rem}
Recall that the {\it usual} cup product (we denote it by $\cup'$ in this paper)
$$\cup': C^m(A, N)\otimes C^n(A, M)\rightarrow C^{m+n}(A, M\otimes_AN)$$ is given by
$f\cup' g(s\overline{a}_{1, m+n})=g(s\overline{a}_{1, n})\otimes_A f(\overline{a}_{n+1, m+n})$ for any two $A$-$A$-bimodules $M$ and $N$. 
In particular, we have
$$\cup': C^m(A, \Omega_{\nc}^p(A))\otimes C^n(A, \Omega_{\nc}^q(A))\rightarrow C^{m+n}(A, \Omega^{p+q}_{\nc}(A))$$
by the canonical isomorphism $\Omega^q_{\nc}(A)\otimes_A \Omega^p_{\nc}(A)\cong \Omega^{p+q}_{\nc}(A)$.
Generally, $f\cup' g$ is not equal to $f\cup g$ in $C^{m+n}(A, \Omega_{\nc}^{p+q}(A)).$ Since $\cup'$ is not compatible with the maps  $\theta_p$,  it is not well-defined in $C_{\sg}^*(A, A)$. In this sense, the cup product $\cup$ may be more interesting than the usual one $\cup'$.  \end{rem}

\begin{figure}
\centering
  \includegraphics[width=100mm]{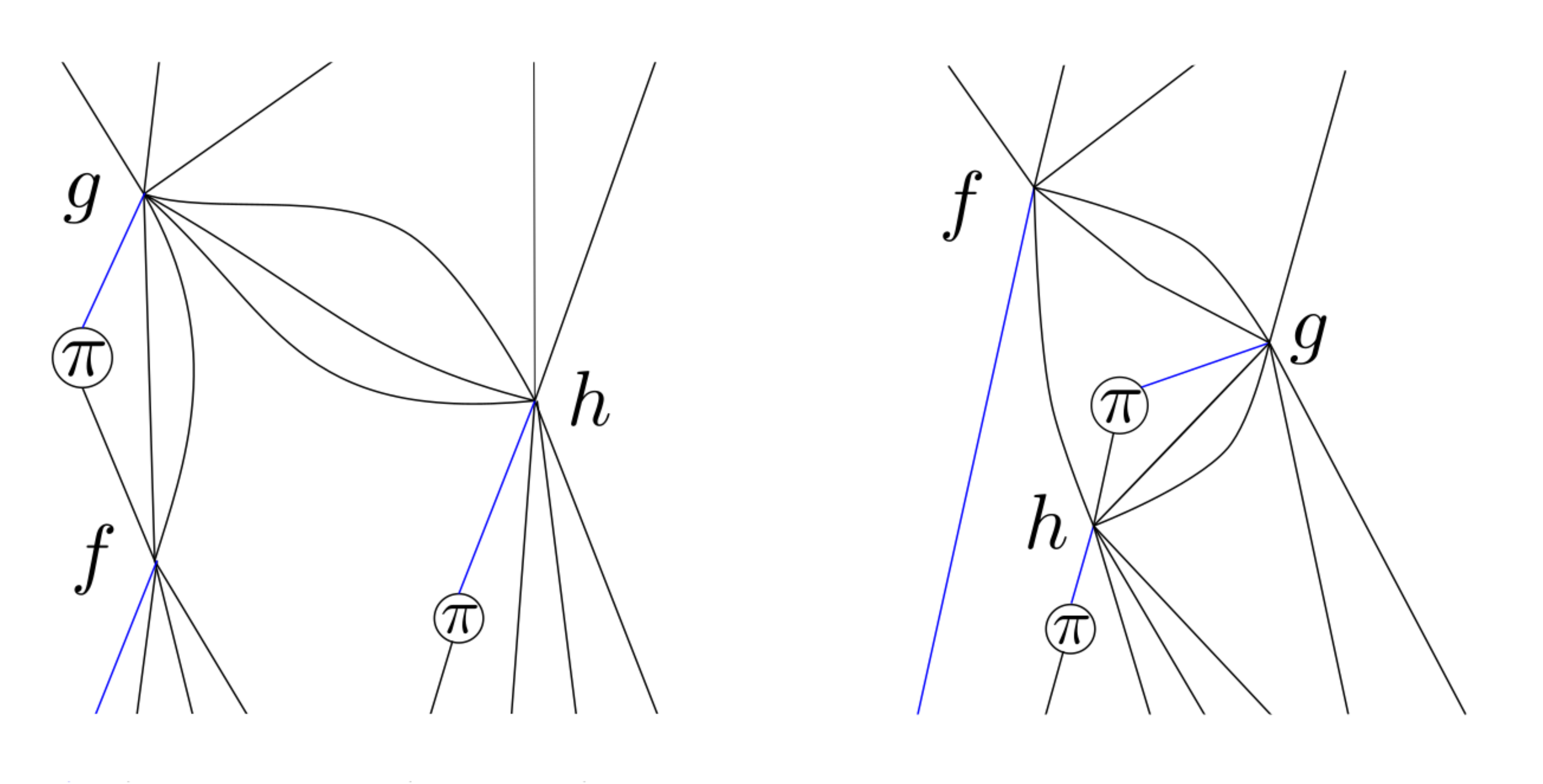}
 \caption{Two of the summands  in the brace operation $f\{g, h\}$. }
    \label{Level-tree}
\end{figure}

\begin{prop}\label{prop4.4-1}
Let $A$ be a Noetherian $k$-algebra. Then the map  $\Phi_*: \HH_{\sg}^*(A, A)\rightarrow \underline{\Ext}^*_{A\otimes A^{\op}}(A, A)$ (cf. Theorem \ref{thm1}) is an isomorphism of graded algebras, where the algebra structure on $\underline{\Ext}^*_{A\otimes A^{\op}}(A, A)$ is given by the Yoneda product. \end{prop}
\begin{proof}
For any $f\in C^{m-p}(A, \Omega^p(A))$ and $g\in C^{n-q}(A, \Omega^q(A))$, by a similar computation as in the proof of Proposition \ref{prop4.4},  we have 
$$f\cup g-f\cup' g=\sum_{i=1}^q (-1)^{(m-p-1)(i-1)}\delta(g\circ_{-i} f-\delta(g)\circ_{-i}f-(-1)^{n-q-1} g\circ_{-i}\delta(f)).$$ Thus  $\cup$ coincides with $\cup'$ at the cohomology level. Therefore, this proposition follows from the  fact that the usual cup product $\cup'$ corresponds to the Yoneda product.\end{proof}

\begin{prop}\label{prop-lie-algebra}
$(C_{\sg}^*(A, A), \delta, [\cdot, \cdot])$ is a dg Lie algebra of degree $-1$.
\end{prop}
\begin{proof}
It follows from Proposition \ref{prop4.4}  that $[\cdot, \cdot]$ is compatible with the differential $\delta$. Namely,  we have $\delta([f, g])=(-1)^{|f|-1}[f, \delta(g)]+[\delta(f), g]$  for any $f, g\in C_{\sg}^*(A, A)$.  It is sufficient to verify the Jacobi identity,
$$(-1)^{(|f|-1)(|h|-1)}[f, [g, h]]+(-1)^{(|f|-1)(|g|-1)}[g, [h, f]]+(-1)^{(|g|-1)(|h|-1)}[h, [f, g]]=0,$$
 where we recall that  $|f|$ is the degree of $f$. Note that to verify the Jacobi identity is equivalent to verify the so-called {\it pre-Lie identity} (cf. \cite{Ger}),
\begin{equation*}
f\circ(g\circ h)-(f\circ g)\circ h =(-1)^{(|g|-1)(|h|-1)}(f\circ (h\circ g)-(f\circ h)\circ g).
\end{equation*}
From Theorem \ref{thm5.2} and the identity (\ref{equation-B1}) in the following, we have   $$(f\circ g)\circ h-f\circ (g\circ h)=(-1)^{(|g|-1)(|h|-1)} f\{g, h\}+f\{h, g\},$$ where $f\{g, h\}$ is the brace operation on $C_{\sg}^*(A, A)$. Roughly speaking, the summands of   $f\{g, h\}$ consist of tree-like graphs with three vertices $f, g$ and $h$ such that the special output (i.e.  the blue output)  is given  by $f$, and the level of $g$ is higher than the level of $h$ (cf. Figure \ref{Level-tree}).
This yields the pre-Lie identity.
\end{proof}

\section{$B_{\infty}$-algebra and Deligne's conjecture on $C_{\sg}^*(A, A)$}
Throughout this section, we fix an associative algebra $A$ over a field $k$. The aim of this section is to  prove the following  two results.
\begin{thm}\label{thm5.2}
There is a $B_{\infty}$-algebra structure on  $C_{\sg}^*(A, A)$ such that the normalized Hochschild cochain complex $C^*(A, A)$ is a $B_{\infty}$-subalgebra. \end{thm}
\begin{thm}\label{thm5.1},
The complex  $C_{\sg}^*(A, A)$ is an algebra over the operad of chains  of the little $2$-disc operad. Equivalently, the Deligne's conjecture  holds for $C^*_{\sg}(A, A)$.
\end{thm}

Combining  Propositions \ref{prop4.2} and \ref{prop-lie-algebra}, and Theorem \ref{thm5.2}, we obtain the following result.
\begin{cor}\label{cor5.3}
Let $A$ be a $k$-algebra. Then the Tate-Hochschild cohomology $\HH_{\sg}^*(A, A)$, equipped with the cup product $\cup$ and  Lie bracket $[\cdot, \cdot]$, is a Gerstenhaber algebra. Moreover, the natural map $\rho: \HH^*(A, A)\rightarrow \HH^*_{\sg}(A, A)$ (cf. Remark \ref{rem3.3}) is a morphism of Gerstenhaber algebras.
\end{cor}
Throughout this section, we consider the opposite cup product $f\cup^{\op} g:=(-1)^{|f||g|}g\cup f$ on $C_{\sg}^*(A, A)$.  Since $\cup$ is graded commutative on $\HH_{\sg}^*(A, A)$, we have $\cup=\cup^{\op}$.

\subsection{$B_{\infty}$-algebras}
The brace operations on $C^*(A, A)$, described by Kadeishvili \cite{Kad} and Getzler \cite{Get93},  are a natural generalization of the Gerstenhaber circle product $\circ$ (cf. Section \ref{section-gerstenhaber}). \begin{defin}\label{defin5.41}
For $f\in C^m(A, A)$ and $g_i\in C^{n_i}(A, A)$ where $i=1, \cdots, k$, the {\it brace operation} is defined as
\begin{eqnarray}\label{eqnarray-brace}
\lefteqn{f\{g_1, \cdots, g_k\}(s\overline{a}_{1, N})}\\
&&=\sum\limits_{\substack{1\leq i_1\leq \cdots \leq i_k\leq m\\ i_j+n_j\leq i_{j+1}}} (-1)^{\epsilon} f(s\overline{a}_{1, i_1}, \overline{g_1}(s\overline{a}_{i_1+1, i_1+n_1}), \cdots, s\overline{a}_{i_k}, \overline{g_k}(s\overline{a}_{i_k+1, i_k+n_k}), \cdots, s\overline{a}_N)\nonumber
\end{eqnarray}
where $\epsilon:=\sum_{j=1}^k (n_j-1)i_j$ and $N=m+\sum_{i=1}^k n_k-k$. Recall that $\overline{g_i}:= \pi\circ g_i$ where $\pi: A\rightarrow s\overline{A}$ is the natural projection of degree $-1$.
\end{defin}
Obviously, the brace operation $f\{g_1, \cdots, g_k\}$ is of degree $-k$ and $f\{g_1\}=f\circ g_1$. 
\begin{defin}[\cite{Bau}]\label{def5.4}
A {\it $B_{\infty}$-algebra} structure on a graded vector space $V:=\bigoplus_{n\in\Z} V^n$ is the structure of a dg bialgebra on the tensor coalgebra $(T(sV):=\bigoplus (sV)^{\otimes p}, \Delta)$ such that the element $1\in k=(sV)^{\otimes 0}$ is the unit of $T(sV)$. Here $\Delta: T(sV)\rightarrow T(sV)\otimes T(sV)$ is defined by
$$\Delta(sa_1\otimes \cdots \otimes sa_p)=\sum_{i=0}^p (sa_1\otimes \cdots \otimes sa_i) \otimes (sa_{i+1}\otimes \cdots \otimes sa_p).$$
\end{defin}
Since the tensor coalgebra is cofree and both the differential $D: T(sV)\rightarrow T(sV)$ and the product $m: T(sV)\otimes T(sV)\rightarrow T(sV)$ are compatible with the coproduct, they are determined by a collection of $k$-linear maps $D_p: (sV)^{\otimes p} \rightarrow sV$  of degree $1$ and $m_{p, q}: (sV)^{\otimes p}\otimes (sV)^{\otimes q}\rightarrow (sV)^{\otimes p+ q}$ of degree zero for $p, q\in\Z_{\geq 0}$, subject to  some relations (called $B_{\infty}$-relations) \cite[Section 2.2]{Vor}.

On $C^*(A, A)$, we take $D_1=\delta, D_2=\cup^{\op}$ and $D_p=0$ for $p\neq 1, 2$. Let $m_{1, 0}=m_{0, 1}=\id$ and $m_{1, q}$ be the brace operation. For other $p, q,$  we set $m_{p, q}=0$. Then this collection  $(D_p, m_{p, q})$  defines a $B_{\infty}$-algebra structure on $C^*(A, A)$ (cf. \cite[Theorem 3.1]{Vor}). In this case, the $B_{\infty}$-algebra relations are simplified as follows.
\begin{enumerate}
\item $(C^*(A, A), D_1, D_2)$ is a dg associative algebra.
\item Higher pre-Jacobi identities.
\begin{eqnarray}\label{equation-B1}
x\lefteqn{\{y_{1, m}\}\{ z_{1, n}\}}\\
&&=\sum_{0\leq i_1\leq \cdots \leq i_m\leq n} (-1)^{\epsilon} x\{z_{1, i_1}, y_1\{z_{i_1+1},\cdots\}, \cdots, z_{i_m}, y_m\{z_{i_m+1}, \cdots\}, \cdots, z_n\}\nonumber
\end{eqnarray}
where $\epsilon:=\sum_{p=1}^m \left(\left(|y_p|-1\right)\sum_{q=1}^{i_q}\left(|z_q|-1\right)\right)$.
\item Distributivity.
\begin{eqnarray}\label{equation-B2}
(x_1\cdot x_2)\{y_{1, n}\}=\sum_{k=0}^n(-1)^{|x_2|\sum_{p=1}^k (|y_p|-1)} (x_1\{y_{1, k}\})\cdot (x_2\{y_{k+1, n}\}),
\end{eqnarray}
\item Higher homotopies.
\begin{eqnarray}\label{equation-B3}
\delta(x\{y_{1, l}\}\lefteqn{)-(-1)^{|x|(|y_1|-1)} y_1\cdot (x \{y_{2, l}\})+(-1)^{\epsilon_{l-1}} (x\{y_{1, l-1}\})\cdot y_l}\nonumber\\
&=& \delta(x) \{y_{1, l}\}-\sum_{i=1}^{l-1} (-1)^{\epsilon_i}x\{y_{1, i}, \delta(y_{i+1}),y_{i+2, l}\}\nonumber\\
&&-\sum_{i=1}^{l-2} (-1)^{\epsilon_{i+1}+1}x \{y_{1, i}, y_{i+1}\cdot y_{i+2}, y_{i+3, l}\},
\end{eqnarray}
where $\epsilon_i:=|x|+\sum_{p=1}^i(|y_p|-1)$. For simplicity, we denote $x\cdot y:=x\cup^{\op}y$.
\end{enumerate}
\begin{rem}\label{remark-B}
Conversely,  a collection of $k$-linear maps $(D_p, m_{p, q})$ on a graded space $V$ with $D_{p}=0, p\neq 1, 2$ and $m_{p,q}=0, p>1$, satisfying the above relations (1)-(4), defines a $B_{\infty}$-algebra structure on $V$.
For more details on brace operations and $B_{\infty}$-algebras, one may refer to \cite{Vor, Kau07a, Kel, MaShSt}.
\end{rem}

\begin{figure}
\centering
  \includegraphics[width=90mm]{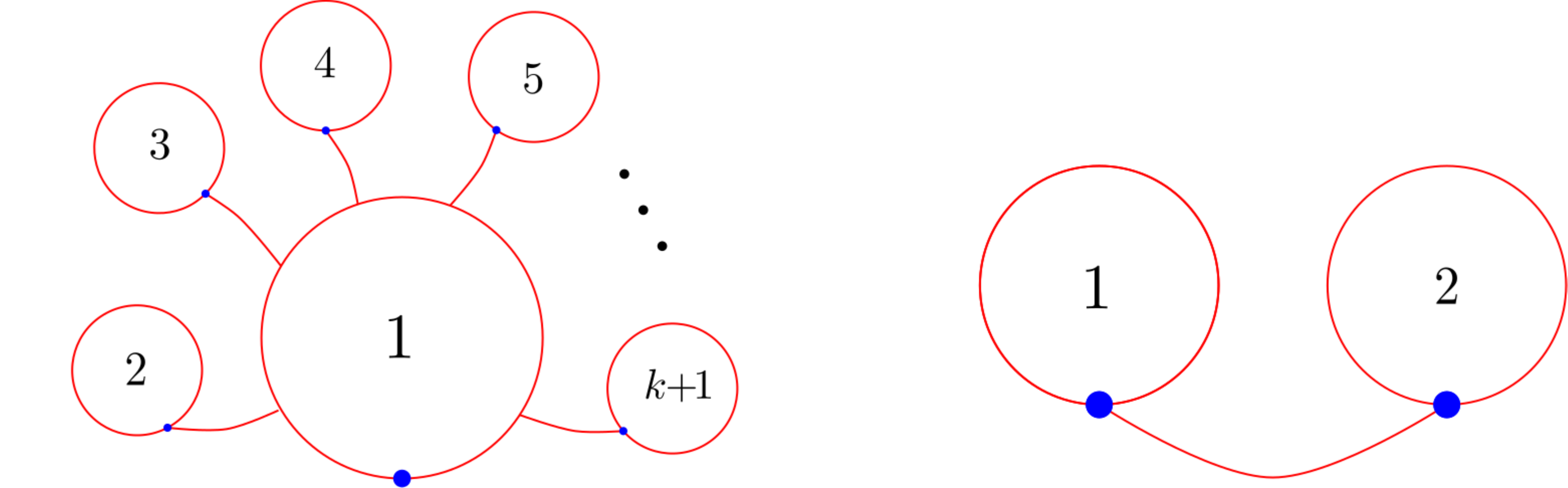}
  \caption{The left  cell  in $CC_k(\Cact(k+1))$  corresponds to the brace operation of degree $-k$ and the right  one in $CC_0(\Cact(2))$ corresponds to the (opposite) cup product.
  }
  \label{Brace-operation}
\end{figure}


\subsection{Brace operations on $C_{\sg}^*(A, A)$}\label{subsection-brace}
In this section, we will extend brace operations on $C^*(A, A)$ to $C_{\sg}^*(A, A)$, using the cactus-like presentations. We prove that the brace operations, with the opposite cup product $\cup^{\op}$, define a $B_{\infty}$-algebra structure on $C_{\sg}^*(A, A)$. 

 Fix $k\in\Z_{\geq 1}$. Let $f'\in C^{m'}_{\sg}(A, A)$ and $g'_i\in C^{n'_i}_{\sg}(A, A)$ for $m', n_i'\in \Z$ and  $i=1, 2, \cdots, k$. Take representatives $f\in C^{m-p}(A, \Omega_{\nc}^{p})$ of $f'$ and $g_i \in C^{n_i-q_i}_{\sg}(A, \Omega_{\nc}^{q_i})$  of $g'_i$, respectively. Here we have $m-p=m'$ and $n_i-q_i=n_i'$ for $i=1, \cdots, k$. 
The brace operation $f'\{g'_1, \cdots, g'_k\}$  is defined as
\begin{equation}
f'\{g_1', \cdots, g'_k\}=\sum\limits_{\substack{0\leq j\leq k\\1\leq i_1< i_2 <\cdots < i_j\leq m\\ 1\leq l_1\leq l_2\leq  \cdots \leq l_{k-j}\leq p }} (-1)^{\epsilon}B^{(i_1, \cdots, i_j)}_{(l_1, \cdots, l_{k-j})}(f; g_1, \cdots, g_k).
\end{equation}
where $B^{(i_1, \cdots, i_j)}_{(l_1, \cdots, l_{k-j})}(f; g_1, \cdots, g_k)$ is illustrated in Figure \ref{Brace-action}, and \begin{eqnarray*}\lefteqn{\epsilon:=\sum_{r=1}^j (n_r'-1)(i_r-r+n_1'+n_2'+\cdots +n_{r-1}')+k-j}\\
&&\ \ \ +\sum_{r=1}^{k-j} (n'_{r+j}-1)(l_{k-j+1-r} +m' +n_1'+\cdots +n_{r+j-1}'+r+j).\end{eqnarray*}

\begin{figure}
\centering
  \includegraphics[width=130mm]{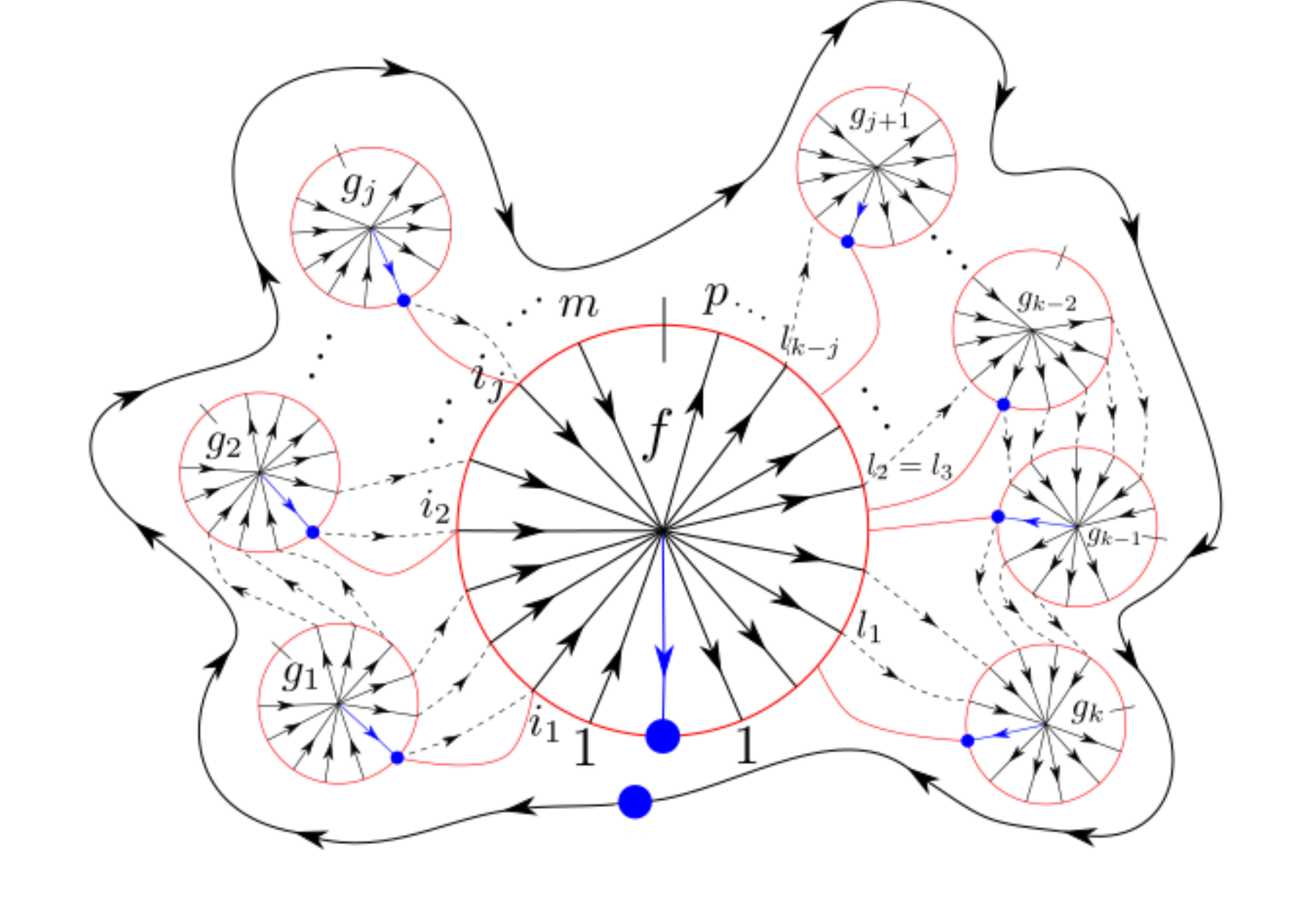}
  \caption{A summand   $B^{(i_1, \cdots, i_j)}_{(l_1, \cdots, l_{k-j})}(f; g_1, \cdots, g_k)$ in the brace operation $f\{g_1, \cdots, g_k\}$. }
  \label{Brace-action}
\end{figure}

Let us  describe the summand $B^{(i_1, \cdots, i_j)}_{(l_1, \cdots, l_{k-j})}(f; g_1, \cdots, g_k)$ in detail. Firstly,  we fix an integer array  $(i_1, \cdots, i_j; l_1, \cdots, l_{k-j}),$ where $0\leq j\leq k$, $1\leq i_1< \cdots <i_j\leq m$ and $1\leq l_1\leq l_2\leq \cdots \leq l_{k-j}\leq p.$  Secondly, we use the cell on the left in Figure \ref{Brace-operation}. We put $f$ into the circle $1$  and $g_i$ into the circle $i+1$, respectively. The inputs and outputs are then  placed according to Process \ref{proc3.4}  described in Section \ref{subsection3.1}, as shown in Figure \ref{Brace-action}. 
 For each $1\leq r\leq j$, the zero point (i.e.   blue dot) of the circle $g_r$ is connected with  the $i_{r}$-th radius in the left semi-circle of $f$ via a red  curve.  For each $1\leq r \leq k-j$, the zero point of the circle $g_{j+r}$ is connected  with the open arc   between the   $(l_{k-j-r+1}-1)$-th   and $l_{k-j-r+1}$-th radii in the right semi-circle via a red  curve. Thirdly, we need  to identify some inputs with outputs. For each $1\leq r\leq j$, add a dashed arrow from the zero point of $g_r$ to the $i_r$-th radius. Starting from the global zero point (i.e.  the zero point of $f$), walk clockwise along the red path (i.e.  the outside circles and the red  curves)  and record the inputs and outputs (including the special outputs of $g_i$) in order as a sequence. When an input is found closely behind an output in this sequence, we call this pair {\it out-in}.  Let us define a process.
\begin{proc}\label{proc5.5}
 Once the pair out-in appears in the sequence, we add a dashed arrow from  the corresponding output to  input  in the graph.  Delete this pair and renew the sequence. Then repeat the above operations until no pair out-in left.
 \end{proc} 
 After applying this Process, we obtain a final sequence with all inputs preceding all outputs. Finally, we translate the updated cactus-like graph  into a treelike graph by putting the inputs (in the final sequence) on the top and outputs on the bottom (e.g. Figure \ref{Brace-action1}).  We therefore get the $k$-linear map  $B^{(i_1, \cdots, i_j)}_{(l_1, \cdots, l_{k-j})}(f; g_1, \cdots, g_k)$ from $s\overline{A}^{\otimes s}$ to $A\otimes s\overline{A}^{\otimes t},$ where $s$ and $t$ are the numbers of the input and  output  in the final sequence, respectively.

 From Lemma \ref{lem5.9} below, it follows that the brace operation $f'\{g_1', \cdots, g_k'\}$  is well-defined, namely,  it does not depend on the choice of representatives $f, g_1, \cdots g_k$.
\begin{rem}\label{rem5.5}
If  $f\in C^m(A, A)$ and $g_i \in C^{n_i}(A, A)$ for $1\leq i\leq k$, we recover the original brace operation $f\{g_1, \cdots, g_k\}$ on $C^*(A, A)$ (cf. Definition \ref{defin5.41}).  The cup product $\cup^{\op}$ on $C_{\sg}^*(A, A)$ can be interpreted as
\begin{equation}\label{equation-cup-interpretation}
f\cup^{\op} g=\mu\{f, g\}
\end{equation}
for $f, g\in C_{\sg}^*(A, A),$ where $\mu$ is the multiplication of $A$ (cf. Remark \ref{remark4.5-mu}). 
\end{rem}

\begin{figure}
\centering
  \includegraphics[width=140mm]{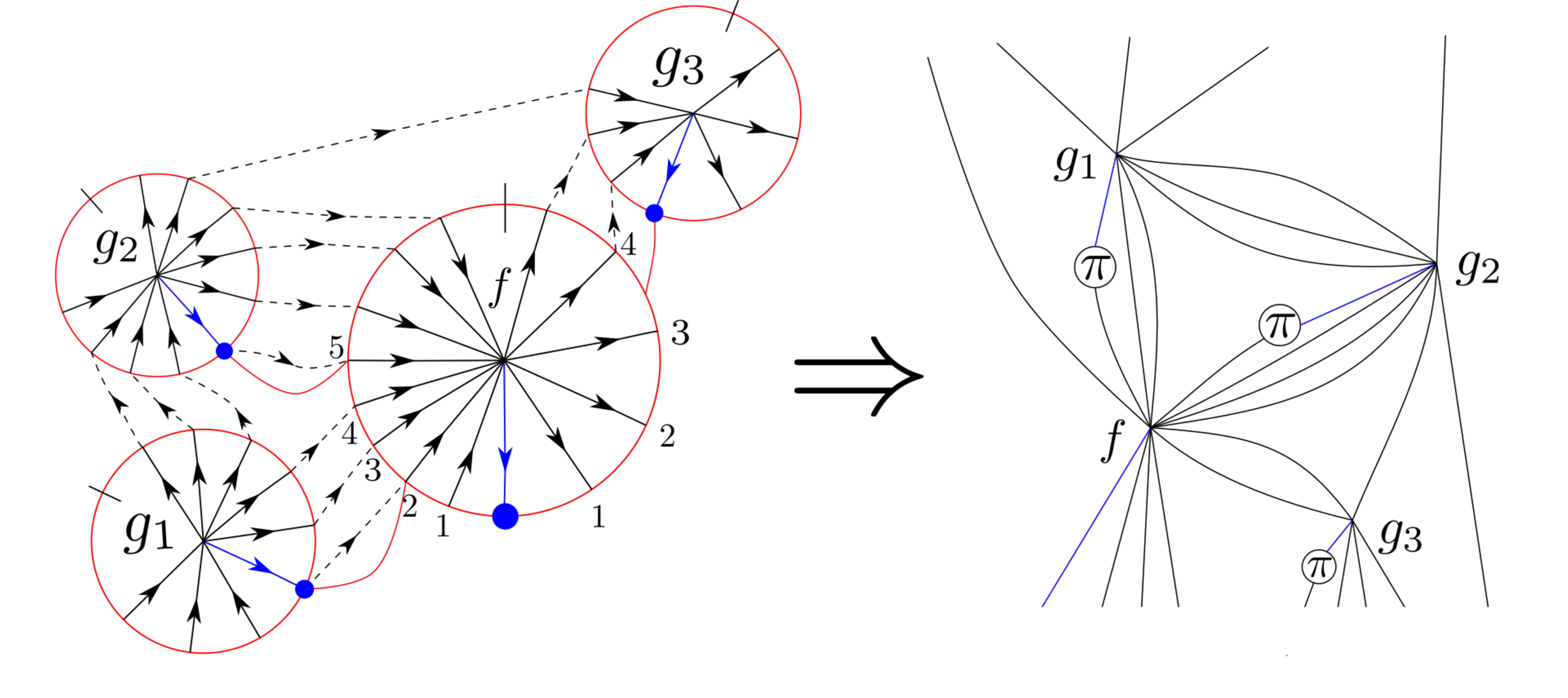}
  \caption{An example of $B^{2, 5 }_{4}(f; g_1, g_2, g_3)$. It corresponds to the linear map $(\id_{s\overline{A}}^{\otimes 4}\otimes \overline{g_3}\otimes \id_{s\overline{A}})\circ (f\otimes \id_{s\overline{A}}^{
  \otimes 2})\circ (\id_{s\overline{A}}^{\otimes 4}\otimes \overline{g_2})\circ (\id_{s\overline{A}}\otimes \overline{g_1} \otimes \id_{s\overline{A}}$).}
  \label{Brace-action1}
\end{figure}


\begin{proof}[Proof of Theorem \ref{thm5.2}]
From  Proposition \ref{prop4.2} and Remark \ref{remark-B},
it is sufficient to verify the identities  (\ref{equation-B1})-(\ref{equation-B3}) for $C_{\sg}^*(A, A)$. From (\ref{equation-delta}) and (\ref{equation-cup-interpretation}) it follows that  (\ref{equation-B3}) is a special case of the identity (\ref{equation-B2}). Hence it remains to check  (\ref{equation-B1}) and (\ref{equation-B2}) for $C_{\sg}^*(A, A)$. The verifications can be done directly by graphic presentations. This proves the theorem.
\end{proof}

\begin{proof}[Proof of Corollary \ref{cor5.3}]
From Corollary \ref{cor-graded-algebra} and Proposition \ref{prop-lie-algebra}, it remains to check the Leibniz rule $[f, g\cup h]=[f, g]\cup h+(-1)^{(|f|-1)|g|} g\cup [f, h]$ for $f, g, h\in \HH_{\sg}^*(A, A)$.
From (\ref{equation-B2}), we have $(g\cup^{\op}h)\circ f=(g\circ f)\cup^{\op} h +g\cup^{\op} (h\circ f)$ in $C_{\sg}^*(A, A)$.  It follows from (\ref{equation-B3}) that $$\delta(\{f\}\{g, h\})-(-1)^{|f|(|g|-1)}g\cup^{\op} (f\circ h)+(-1)^{|f|+|g|+1}((f\circ g)\cup^{\op}  h-f\circ(g\cup^{
\op}h))=0$$ since $\delta(f)=\delta(g)=\delta(h)=0$, thus $f\circ (g\cup^{\op}h)=(f\circ g)\cup^{\op}h +(-1)^{(|f|-1)|g|}g\cup (f\circ h)$ on $\HH^*_{\sg}(A, A)$. This verifies the Leibniz rule. Therefore $(\HH_{\sg}^*(A, A), \cup, [\cdot, \cdot])$ is a Gerstenhaber algebra.
\end{proof}


\subsection{An action of $CC_*(\Cact)$ on $C_{\sg}^*(A, A)$}
In this section we will generalize the brace action to any cell in the cellular chain model $CC_*(\Cact)$. 

In the series of papers \cite{Kau05, Kau07, Kau08}, the author introduced the (topological) operad $\Cact$ of spineless cacti.  He constructed a natural action of the cellular chain model $CC_*(\Cact)$ on $C^*(A,A)$. Let $\Brace$ be the dg suboperad of the endomorphism operad $\Endop(C^*(A, A))$, generated by the cup product and brace operations on $C^*(A, A)$. The author proved that $CC_*(\Cact)$ is  isomorphic to $\Brace$ (cf. \cite[Proposition 4.9]{Kau07a}), and     equivalent to  the operad of chains of the little $2$-discs operad  (cf. \cite[Theorem 3.11]{Kau07a}). As a conclusion, he provided a proof of Deligne's conjecture for $C^*(A, A)$.

\begin{figure}
\centering
  \includegraphics[width=120mm]{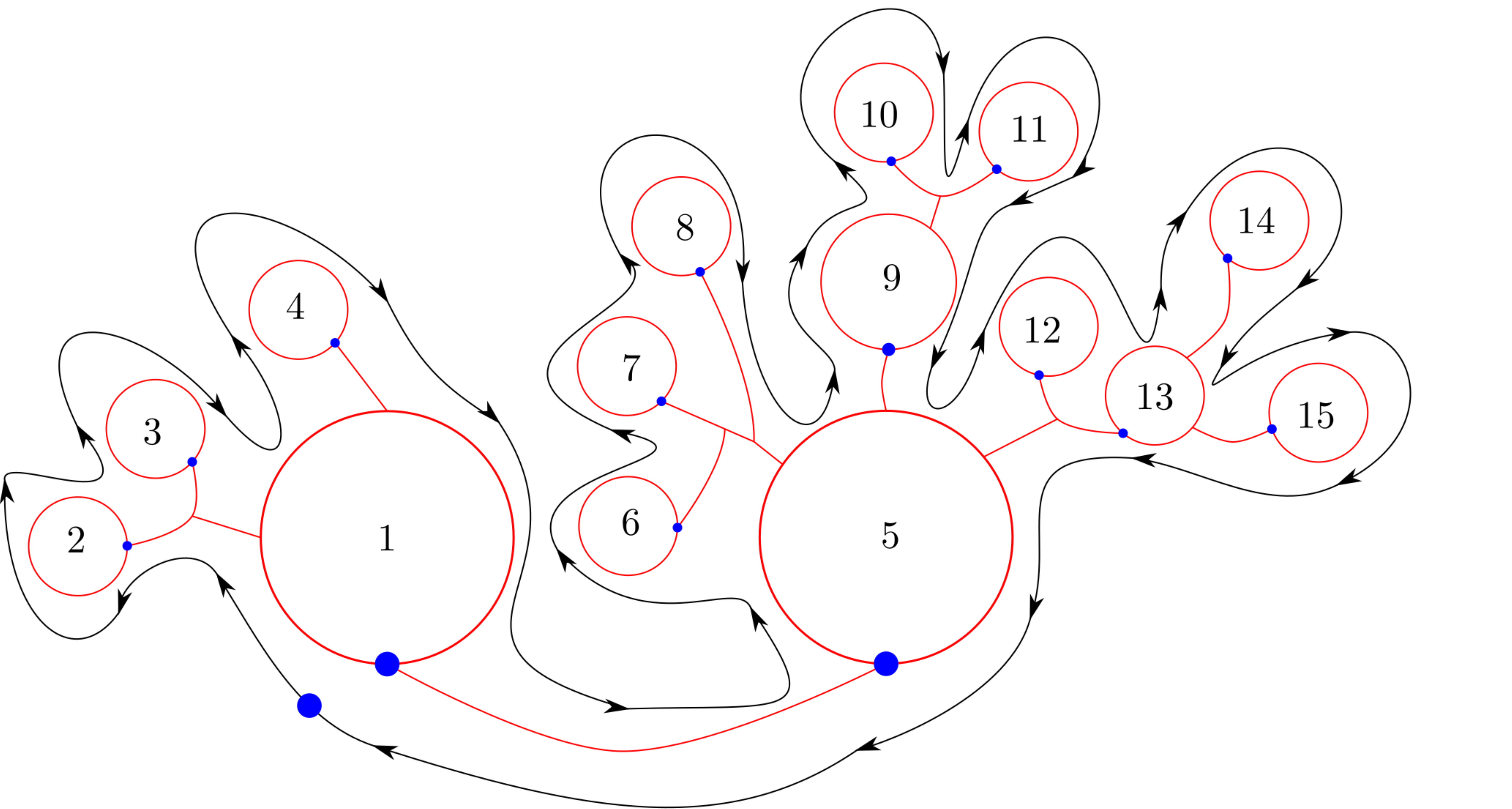}
  \caption{An example of a cell  in $CC_*(\Cact(15))$. It corresponds to the operation (of degree $8$) in $C_{\sg}^*(A, A)$: $f_1\{f_2\cup^{\op} f_3, f_4\})\cup^{\op} (f_5\{f_6\cup^{\op} f_7\cup^{\op} f_8, f_9\{f_{10}\cup^{\op} f_{11}\}, f_{12}\cup^{\op} (f_{13}\{f_{14}, f_{15}\})$.}
  \label{Brace-operation1}
\end{figure}

From the above analysis, any cell in $CC_*(\Cact)$ induces  an action on $C^*(A, A)$. For instance, the cell on the left in  Figure  \ref{Brace-operation} corresponds to  the brace operation of degree $-k$ while the cell on the right corresponds to the (opposite) cup product $\cup^{\op}$. More generally, any cell in $CC_*(\Cact)$  can be represented by a cactus-like graph as shown in  Figure \ref{Brace-operation1}.  Explicitly, the zero point of a circle is indicated by the blue dot. By  a (red) curve connecting two circles,  we mean that the two circles intersect at the endpoints of the  curve (called the  {\it intersection point}). Note that  at least one endpoint of each curve should coincide with the  zero point. In other words, we do {\it not} allow that two circles intersect at non-zero points. 
We  allow that three or more circles intersect at one common point.  For a cell,  there is  only one global zero point (called {\it root}) which may or may not be an intersection point.  In fact, the root is the only  zero point which is not necessary to be an intersection point. If the root indeed is an intersection point, the other endpoint(s) of the red curve(s)  must be zero point(s). Note that
the degree (or dimension) of the cell equals  the number of the intersection points (except the global zero point).  

\begin{rem}
The cactus-like presentations of cells in $CC_*(\Cact)$ described above are slightly different from the ones in  the original papers \cite{Kau05, Kau07a}. We use red curves to indicate  the intersection points. This modification could make it more convenient to define the action of a cell on $C_{\sg}^*(A,A)$.  For more details on $\Cact$ and  $CC_*(\Cact)$, refer to \cite{Kau05, Kau07a}. 
\end{rem}

\begin{figure}
\centering
  \includegraphics[width=150mm]{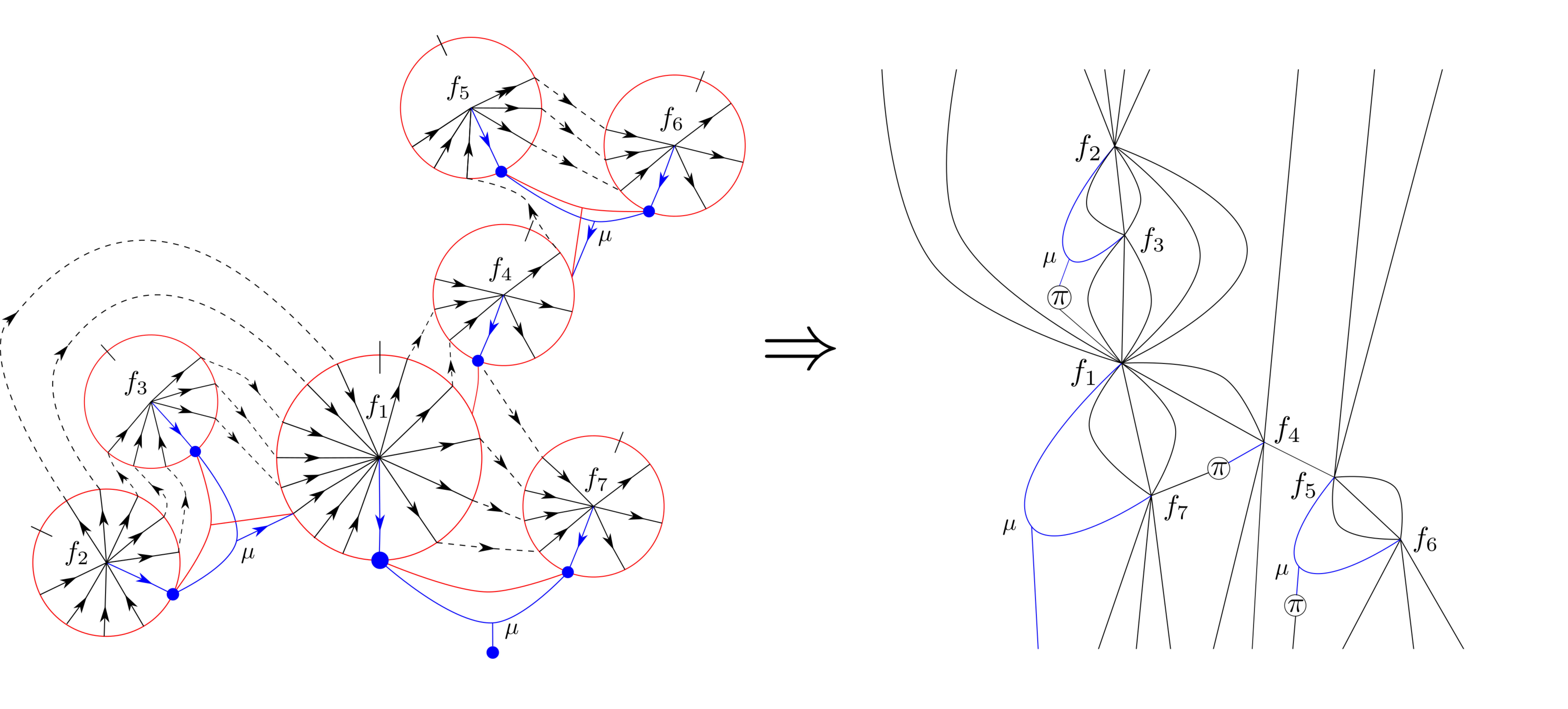}
  \caption{An action   corresponding to the map $(\mu\otimes \id_{s\overline{A}}^{\otimes 5}\otimes (\overline{f_5\cup^{\op} f_6})) (\id_{A}\otimes f_7\otimes \id_{s\overline{A}}^{\otimes 5})(\id_{A}\otimes \id_{s\overline{A}}^{\otimes 3}\otimes \overline{f_4} \otimes \id_{s\overline{A}}^{\otimes 2}
  )  (f_1\otimes \id_{s\overline{A}}^{
  \otimes 3}) (\id_{s\overline{A}}^{\otimes 2}\otimes (\overline{f_2\cup^{\op} f_3}))$ in $C_{\sg}^*(A, A)$.}
  \label{Brace-action2}
\end{figure}

Let us now  generalize the brace operations to any cell in $CC_*(\Cact)$.  Let $\tau$ be a cell in $ CC_l(\Cact(k)) $ of degree $l\in\Z_{\geq 0}$ (e.g. Figure \ref{Brace-operation1}). Take any $k$ elements $f_i'\in C^{m_i'}_{\sg}(A, A)$ for $ 1\leq i\leq k$, the action of $\tau$ on $f'_1\otimes\cdots\otimes f'_k$ , denoted by $\tau(f'_1, \cdots, f'_k)$, is defined as follows.

{\bf First step:}  Choose a representative $f_i\in C^{m_i-p_i}(A, \Omega_{\nc}^{p_i}(A))$ for each $f_i'$, where $m_i-p_i=m_i'$ for $1\leq i\leq k$.  We put $f_i$ into the $i$-th circle of $\tau$ according to Process \ref{proc3.4}. 
Recall that there is only one nonzero endpoint of the red curves at any intersection point (except the global zero point).
We need to fix the {\it type} of an intersection point by moving the nonzero endpoint of the red curves  along the circle so that it either  coincides with the $j$-th radius (i.e.  input, $1\leq j\leq m_i$) of $f_i$,  or  located in the open arc between the $(l-1)$-th and $l$-th radii (i.e.  outputs, $1\leq l\leq p_i$) of $f_i$. Accordingly, we say that the intersection point has type $j$ or $-l$. 
If an intersection point contains more than one red curves, 
 we multiply in order all the corresponding special outputs (i.e.  blue radii), and  then get a new output. 
We stress that when moving the endpoints of  curves along a fixed circle, the order (starting from the zero point in clockwise) of the intersection points must be preserved.


Once the types of all  intersection points are fixed, we arrange them as an integer array, labelled by  curves. This sequence is called a {\it type} of $\tau$.   
Denote the set of all intersection points (except the global zero point) by $I(\tau)$.  Clearly, a type is a map from $I(\tau)$ to $\Z^{l}$, where $l$ is the degree of $\tau$. We denoted by $\mathfrak{T} (\tau)(f_1, \cdots, f_k)$ the set of all types of $\tau$  associated with  $f_1\otimes \cdots \otimes f_k$.


{\bf Second step:}  
For any fixed type $\Phi\in\mathfrak{T}(\tau)(f_1, \cdots, f_k)$, we need to add dashed arrows from  inputs to outputs. 
Starting from the global zero point, walk along the red path (i.e.  outside circles and the red curves); record  the inputs and outputs, except  those already connected by dashed arrows,  as a sequence. We apply Process \ref{proc5.5} to  get  a final cactus-like graph and a sequence in which all inputs precede outputs.

{\bf Third step:} 
 By translating the  above cactus-like graph into the treelike graph, we get a $k$-linear map $\tau(\Phi; f_1, \cdots, f_k): s\overline{A}^{\otimes s} \rightarrow A\otimes s\overline{A}^{\otimes t},$ where $s$ and $t$ are the numbers of inputs and outputs in the final sequence, respectively.  It is clear that $s-t=\sum_{i=1}^km_i'-l$.

Therefore, we have the following definition.
\begin{defin}\label{defin-5.9}
The action of  $\tau\in CC_l(\Cact(k))$  on $f_1'\otimes \cdots \otimes f_k'$ is defined as
$$\tau(f_1', \cdots, f_k'):=\sum_{\Phi \in \mathfrak{T}(\tau)(f_1,\cdots, f_k)} (-1)^{\epsilon(\Phi)}\tau(\Phi; f_1, \cdots, f_k),
$$
where the sign $(-1)^{\epsilon(\Phi)}$ is determined by signs in brace operations since   $\tau$ is generated by  cells corresponding to the cup product and brace operations.
\end{defin}

 \begin{figure}
\centering
  \includegraphics[width=150mm]{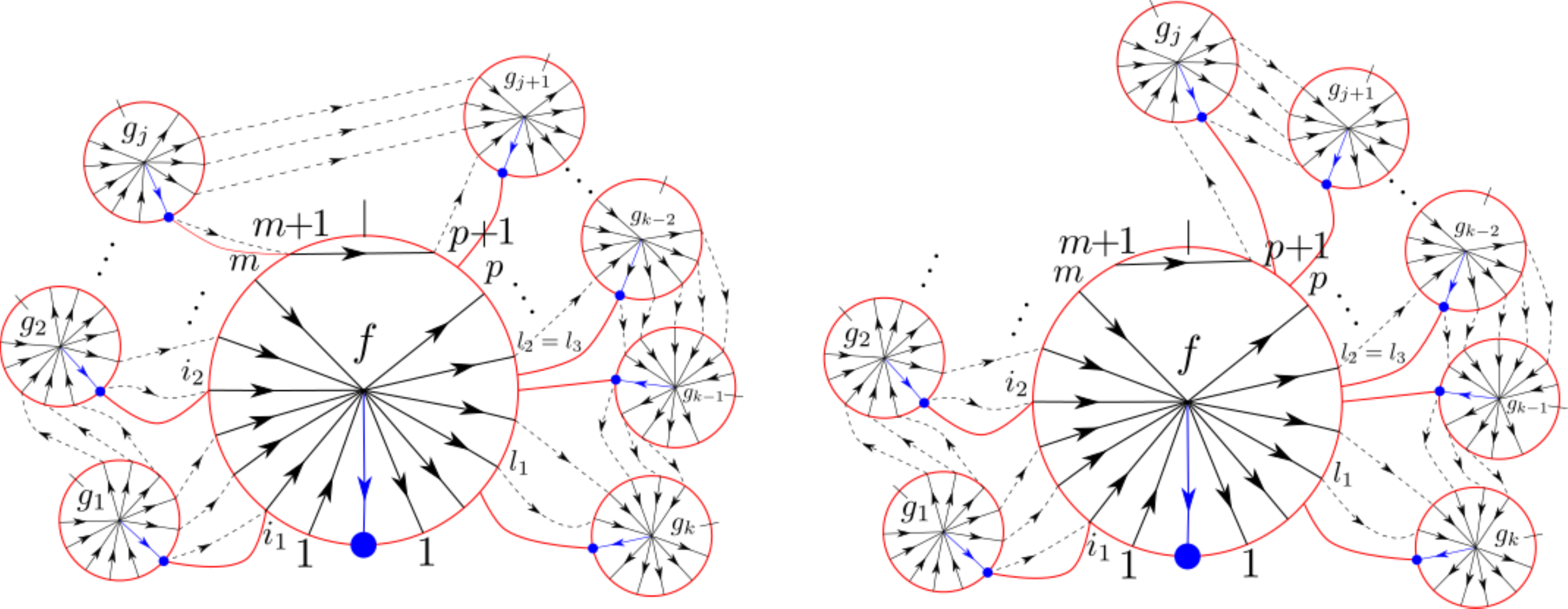}
  \caption{The type $(i_1, \cdots, i_{j-1}, m+1; l_1, \cdots, l_{k-j})$ on the left is cancelled with the type $(i_1, \cdots, i_{j-1}; l_1, \cdots, l_{k-j}, p+1)$ on the right in the identity (\ref{equation5.2}).}
  \label{Indpendent}
 \end{figure}

\begin{lem}\label{lem5.9}
$\tau(f_1', \cdots, f_k')$ is well-defined, namely, it does not depend on the choice of representatives $f_1, \cdots, f_k$.
\end{lem}
\begin{proof}
Since the action $\tau$ on $f_1, \cdots, f_k$ can be written as the (opposite) cup product and compositions of brace operations,  it is sufficient to prove that $\cup^{\op}$ and brace operations  are independent of the choice of representatives. From Proposition \ref{prop4.2}, it follows that the cup product $\cup^{\op}$ is well-defined on $C_{\sg}^*(A, A)$. Thus it remains to check the following  identities on $C_{\sg}^*(A, A),$
\begin{equation}\label{equation5.2}
\begin{split}
f\{g_1, \cdots, g_k\}&=(f\otimes \id_{s\overline{A}})\{g_1, \cdots, g_k\}\\
f\{g_1, \cdots, g_k\}&=f\{g_1, \cdots, g_i\otimes \id_{s\overline{A}}, \cdots, g_k\}
\end{split}
\end{equation}
where $1\leq i\leq k$ and $f\in C^{m-p}(A, \Omega_{\nc}^p(A)), g_i\in C^{n_i-q_i}(A, \Omega_{\nc}^{q_i}(A)).$  Let us check the first identity. Observe that all the terms on the left hand side are cancelled out by terms on the right hand side. We need to cancel out  the remaining terms on the right hand side. Note that  the cactus-like presentation of each remaining term has the following property: there is a red curve connecting  with the $(m+1)$-th input or  the open arc between the $p$-th and $(p+1)$-th outputs of $f$.  Assume that the circle $g_j$ intersects with $f$ at the $(m+1)$-th input via a red curve (cf. the left graph in Figure \ref{Indpendent}), this term will cancel with the one whose cactus-like presentation is obtained by just moving the red chord   into  the open arc between $p$-th and $(p+1)$-th output of $f$ (cf. the right graph in Figure
\ref{Indpendent}). In this way,  all the remaining terms  cancel out. This verifies the first identity.
The second identity can  be verified by the same argument.  This proves the lemma.
 \end{proof}

\begin{figure}
\centering
 \includegraphics[width=130mm]{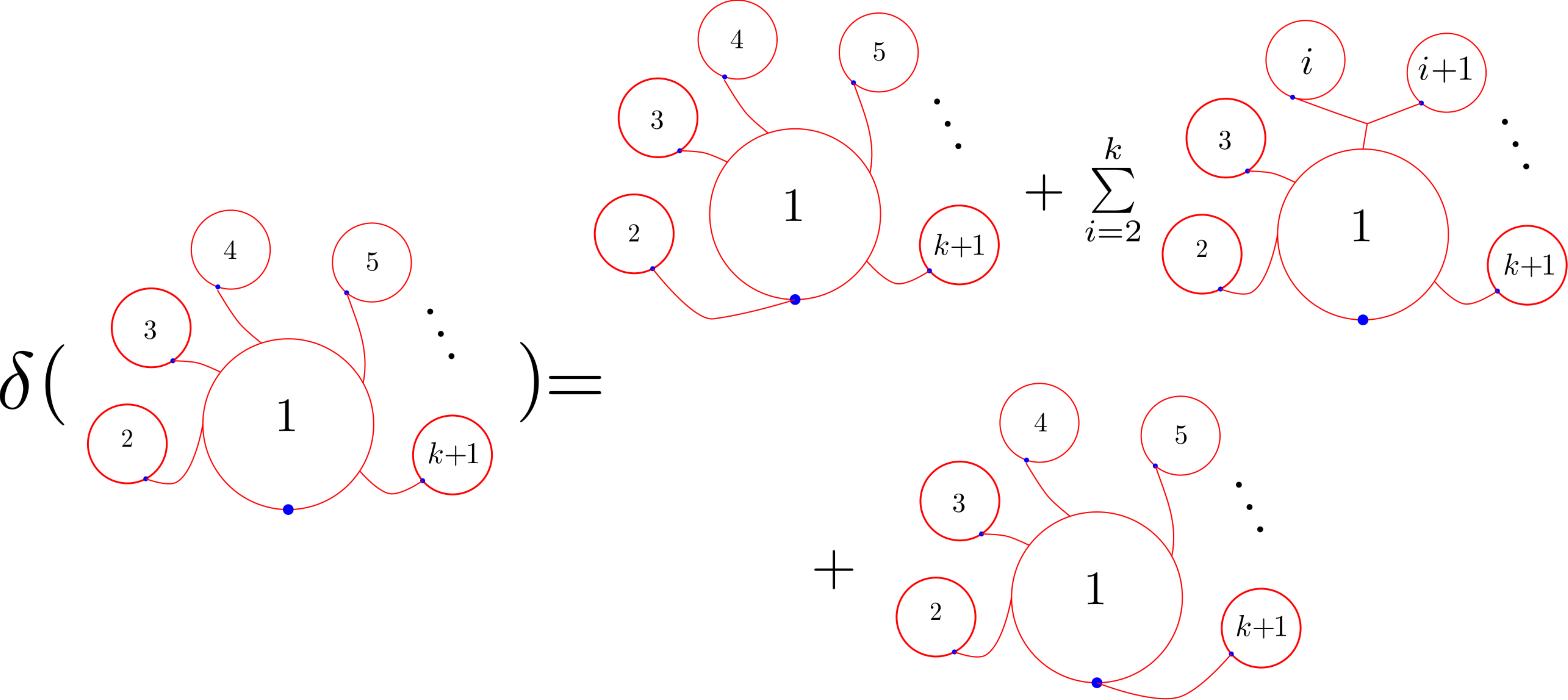}
  \caption{The differential in $CC_*(\Cact)$.} \label{BV}
  \end{figure}

 \begin{proof}[Proof of Theorem \ref{thm5.1}]
 Since $CC_*(\Cact)$ is equivalent to the operad of chains of the little $2$-discs operad (cf. \cite[Proposition 4.9]{Kau07a}), it is sufficient to prove that the action of $CC_*(\Cact)$ (cf.  Definition \ref{defin-5.9}) induces a morphism of dg operads $\varphi: CC_*(\Cact)\rightarrow \Endop(C_{\sg}^*(A, A)).$ It is not difficult to show that $\varphi$ is compatible with the compositions. Let us prove that $\varphi$ is  compatible with the differentials. Since $CC_*(\Cact)$ is generated by the cells as shown in Figure \ref{Brace-operation}, it is sufficient to check $\varphi(\delta(\tau))=\delta(\varphi(\tau))$, where $\tau$ is the cell corresponding to the brace operation.  From Figure \ref{BV} it follows that to prove the above identity  is equivalent to prove (\ref{equation-B3}). This proves the theorem.
\end{proof}


\section{An application to self-injective algebras}\label{section6}
\subsection{Generalized Tate-Hochschild complex} Before  the case of self-injective algebras, let us  start with a more general setting. Let $A$ be an associative algebra over a field $k$. Denote by $A^{\vee}:=\Hom_{A\otimes A^{\op}}(A, A\otimes A^{\op}).$ 
It is clear that $A^{\vee}$ is isomorphic to the zeroth Hochschild cohomology $\HH^0(A, A\otimes A^{\op})$. Thus we have
$$A^{\vee}\cong \left\{ \sum_i x_i\otimes y_i \in A\otimes A \ | \ \mbox{  $\sum_i ax_i\otimes y_i=\sum_i x_i\otimes y_ia,$ for any $a\in A$} \right\},$$
where the isomorphism sends $\alpha\in A^{\vee}$ to $\alpha(1).$ Note that $A^{\vee}$ has an $A$-$A$-bimodule structure: For any $\sum_i x_i\otimes y_i\in A^{\vee}$ and $a, b \in A$, the action is   $a\cdot \left(\sum_i x_i\otimes y_i\right)\cdot b:=\sum_i x_ib\otimes ay_i.$

Recall that  $(C_*(A, A^{\vee}), b)$ is the Hochschid chain complex with coefficients in $A^{\vee}$.  We now construct an unbounded complex 
\begin{equation}\label{equation-Tate-complex}
\calD^*(A, A): \cdots\xrightarrow{b_2} C_1(A, A^{\vee})\xrightarrow{b_1} A^{\vee}\xrightarrow{\mu} A\xrightarrow{\delta^0} C^1(A, A)\xrightarrow{\delta^1} \cdots
\end{equation}
where $C^i(A, A)$ is in degree $i$ and $\mu: A^{\vee}\rightarrow A$ is given by the multiplication of $A$.
Let us denote the cohomology of $\calD^*(A, A)$ by $\THH^*(A, A)$.
\begin{lem}\label{lemma6.1}
There is a natural embedding of complexes $\iota^*: \calD^*(A, A)\hookrightarrow C_{\sg}^*(A, A)$.
\end{lem}
\begin{proof}
For $i\geq 0$, it is known from Definition \ref{defin3.2}  that $C^i(A, A)$ is a subspace of $C_{\sg}^i(A, A)$.  For $i<0$,  we define a map
$\iota^{i}: C_{-i-1}(A, A^{\vee}) \rightarrow C^{i}_{\sg}(A, A) $ by the following formula  $$\iota^i(\alpha):=\sum_j x_j\otimes \overline{a}_{1, -i-1}\otimes \overline{y_j}\in \Omega_{\nc}^{-i}(A)\subset C^{-i}_{\sg}(A, A), $$ where $\alpha=\left(\sum_j x_j\otimes y_j\right)\otimes \overline{a}_{1, -i-1}$.
Then we need to  check $\iota^*\circ \partial^*=\partial^*\circ \iota^*$. For $i\geq 0$, it is clear that  $\iota^{i+1}\circ \partial^i=\delta^{i}\circ \iota^i$. For $i<0$, we claim  $\theta_0\circ\iota^{i+1}\circ \partial^i=\delta^i\circ \iota^{i}$. Indeed, 
\begin{equation*}
\begin{split}
(\theta_0\circ \iota^{i+1}\circ \partial^i(\alpha))(\overline{b})=&(-1)^{i-1}\sum_j (x_j\otimes \overline{a}_{1, -i-1})\blacktriangleleft y_j\otimes \overline{b}\\
=&(-1)^{i}\sum_j (x_j\otimes \overline{a}_{1, -i-1}\otimes \overline{ y_j})\blacktriangleleft \overline{b}-(-1)^{i}x_j\otimes \overline{a}_{1, -i-1}\otimes \overline{ y_jb}\\
=&\delta^i\circ \iota^i(\alpha)(\overline{b}),
\end{split}
\end{equation*}
where the second identity follows from Lemma \ref{lem0} and the third one follows from  $\sum\limits_jx_j\otimes y_jb=\sum\limits_jbx_j\otimes y_j$.  This proves the lemma.
\end{proof}
\subsection{$\star$-product on $\calD^*(A, A)$}\label{section-star2}


Let $A$ be an associative algebra (not necessarily, self-injective) over a field $k$.
We construct a product (of degree zero), called {\it $\star$-product}, on $\calD^*(A, A)$
\begin{equation}\label{equation-Star-product}
\star: \calD^*(A, A)\otimes \calD^*(A,A)\rightarrow \calD^*(A, A),
\end{equation}
which extends the cup product on $C^*(A, A)$ and the cap product between $C^*(A, A)$ and $C_*(A, A^{\vee})$.
\begin{enumerate}[\upshape (i)]
\item For $p, q\geq 0$, define
$\star: C_q(A, A^{\vee})\otimes C_p(A, A^{\vee})\rightarrow C_{p+q+1}(A, A^{\vee})$  by  $$\alpha\star\beta=\sum_{i, j}(x_j'x_i \otimes y_j') \otimes s\overline{a}_{1, p}\otimes  s\overline{y_i}\otimes s\overline{b}_{1, q},$$ where  $\alpha=\left(\sum_ix_i\otimes y_i\right) \otimes s\overline{a}_{1, p}$ and $\beta=\left(\sum_j x_j'\otimes y_j'\right)\otimes s\overline{b}_{1,q}.$
\item For $m, p\in\Z_{\geq 0}$ such that $p\geq m$, define $\star: C_p(A, A^{\vee})\otimes C^m(A, A)\rightarrow C_{p-m}(A, A^{\vee})$ as the usual cap product. Namely, for $f\in C^m(A,A)$ and $\alpha=\left(\sum_ix_i\otimes y_i\right) \otimes s\overline{a}_{1, p}$, we have
$$\alpha\star f:=\sum_i (x_i\otimes f(s\overline{a}_{p-m+1, p})y_i)\otimes s\overline{a}_{1, p-m}.$$
Similarly, we define $\star: C^m(A, A)\otimes C^p(A, A^{\vee})\rightarrow C_{p-m}(A, A^{\vee})$ by $$f\star \alpha:=\sum_i (x_if(s\overline{a}_{1, m})\otimes y_i)\otimes s\overline{a}_{m+1, p}.$$
\item For $m, p\in \Z_{\geq 0}$ such that $p<m$, define $\star: C^m(A, A)\otimes C_p(A, A^{\vee})\rightarrow C^{m-p-1}(A, A)$ by the following formula,
$$ f\star \alpha(s\overline{b}_{1, m-p-1}):=\sum_if(s\overline{b}_{1, m-p-1}\otimes s\overline{x_i}\otimes s\overline{a}_{1, p})y_i.$$
Similarly, $\star: C_p(A, A^{\vee})\otimes C^m(A, A)\rightarrow C^{m-p-1}(A, A)$ is defined by
$$\alpha\star f(s\overline{b}_{1, m-p-1})=\sum_i x_if(s\overline{a}_{1, p}\otimes s\overline{y_i}\otimes s\overline{b}_{1, m-p-1}).$$
\item For $m, n\in\Z_{\geq 0}$, we define $f\star g:=f\cup g,$ for  $f\in C^m(A, A)$ and $g\in C^n(A, A)$.
\end{enumerate}
\begin{lem}
The $\star$-product is compatible with the differential $\partial$ in $\calD^*(A, A)$.
As a result, it induces a well-defined product (still denoted by $\star$) on the cohomology $\THH^*(A, A)$.
\end{lem}
\begin{proof}
This follows from straightforward computations.
\end{proof}
\begin{rem}\label{remark6.7}
In general,  the  $\star$-product restricted to the complex $C_*(A, A^{\vee})$ is not a chain map since $\partial(\alpha\star \beta)\neq \pm\alpha\star\partial( \beta)$ if $\alpha\in C_{0}(A, A^{\vee})$. In order to make it well-defined, we have to extend the $\star$-product from $C_*(A, A^{\vee})$ to  $\calD^*(A, A)$.
Assume that $A$ is a commutative symmetric algebra. The $\star$-product restricted to $C_{>0}(A, A)$ coincides with the so-called Abbaspour product (see \cite[Theorem 6.1]{Abb}) motivated by certain operations in string topology. For more details and  further investigation, one may refer to \cite{RiWa}.
\end{rem}
\begin{rem}
In general, the $\star$-product on $\calD^*(A, A)$ is {\bf not} associative although it is well-known that the associativity holds when restricted to either $\calD^{\geq 0}(A, A)$ or $\calD^{<0}(A, A)$. For instance, let $\alpha:=(\sum_ix_i\otimes y_i)\otimes s\overline{a}_{1, p},  \beta:=(\sum_j x_j'\otimes y_j')\otimes s\overline{b}_{1, q},$ and $f\in C^m(A, A)$, where $p, q>m>0$. We have
\begin{eqnarray*}(\alpha\star f)\star \beta-\alpha\star(f\star \beta)\lefteqn{=\partial m_3(\alpha, f, \beta)-m_3(\partial(\alpha), f, \beta)-}\\&(-1)^{p-1}m_3(\alpha, \partial(f), \beta)-(-1)^{m-p-1}m_3(\alpha, f, \partial(\beta)),
\end{eqnarray*}
where \begin{eqnarray*}m_3(\alpha, f, \beta\lefteqn{)=\sum_{i, j}\sum_{k=1}^m (-1)^{(m-1)(m-k-1)}}\\&& (x_j'x_i\otimes y_j')\otimes s\overline{a}_{1, p-m+k}\otimes f(s\overline{a}_{p-m+k+1, p} \otimes s\overline{y_i}\otimes s\overline{b}_{1, k-1})\otimes s\overline{b}_{k, p}.\end{eqnarray*}
This means that the associativity holds up to homotopy. From this point of view, it might be interesting to ask whether this extends to an $A_{\infty}$-algebra structure with $(\partial, \star, m_3, \cdots)$ on $\calD^*(A, A)$. In \cite[Proposition 6.5]{RiWa}, we give an affirmative answer to this question in the case where $A$ is a (dg) symmetric algebra. For general cases, further investigations are needed. \end{rem}

\begin{prop}\label{prop6.9}
The $\star$-product  is graded commutative and associative on $\THH^*(A, A)$.
\end{prop}

\begin{proof} Let us first verify the graded commutativity in the following cases.
\begin{enumerate}
\item For $\alpha\in C_p(A, A^{\vee})$ and $\beta\in C_q(A, A^{\vee})$,  denote
\begin{eqnarray*}
\beta\bullet\alpha:=\sum_{i, j} \sum_{k=1}^{p+1}(-1)^{qk} (x_i\otimes y_i)\otimes s\overline{a}_{1,k-1} \otimes s\overline{x_j'}\otimes s\overline{b}_{1, q}\otimes  s\overline{y_j'}\otimes s\overline{a}_{k, p}.\end{eqnarray*} Then we have
\begin{eqnarray}\label{equation-star} \lefteqn{\partial(\beta\bullet \alpha)}\\
 &=& (-1)^q\beta\star \alpha+\sum_{i, j} \sum_{k=1}^{p+1}\sum_{l=0}^{k-2}(-1)^{qk+l} (x_i\otimes y_i)\otimes s\overline{a}_{1, l-1}\otimes s\overline{a_la_{l+1}}\nonumber \\
&& \otimes  s\overline{a}_{l+2, k-1}\otimes s\overline{x_j'}\otimes s\overline{b}_{1, q}\otimes s\overline{y_j'}\otimes  s\overline{a}_{k, p}+\sum_{i, j} \sum_{k=1}^{p+1}\sum_{l=k}^{p}(-1)^{q(k-1)+l}\nonumber \\
&&(x_i\otimes y_i)\otimes s\overline{a}_{1, k-1}\otimes s\overline{x_j'}\otimes s\overline{b}_{1, q}\otimes s\overline{y_j'}\otimes s\overline{a}_{k, l-1}\otimes  s\overline{a_la_{l+1}}\otimes   s\overline{a}_{l+2, p}\nonumber\\
&&  -(-1)^{(p+1)(q+1)+q} \alpha\star \beta+\partial(\beta)\bullet \alpha\nonumber\\
&=&(-1)^q(\beta\star \alpha-(-1)^{(p+1)(q+1)} \alpha \star \beta)+(-1)^{q}\beta\bullet \partial(\alpha)+\partial(\beta)\bullet\alpha. \nonumber
\end{eqnarray}
Thus on $\THH^*(A, A)$, we have $\beta\star \alpha-(-1)^{(p-1)(q-1)}  \alpha\star \beta=0.$
\item For $m, p\in\Z_{\geq 0}$ such that $p\geq m-1$, denote
\begin{eqnarray*}
f\bullet \alpha:=\sum_{i}\sum_{k=1}^{p-m+1}(-1)^{(m-1)k}(x_i\otimes y_i)\otimes s\overline{a}_{1, k-1}\otimes f(s\overline{a}_{k, k+m-1})\otimes s\overline{a}_{k+m, p}.
\end{eqnarray*}
Then we have
\begin{eqnarray*}
\partial(f\bullet \alpha)
=(-1)^{m-1}(f\star\alpha-(-1)^{m(p-1)}\alpha\star f)+\partial(f)\bullet \alpha+(-1)^{m-1} f\bullet \partial(\alpha).\end{eqnarray*}
Thus on $\THH^*(A, A)$, we have  $f\star\alpha -(-1)^{m(p-1)}\alpha\star f=0.$
\item For $m, p\in \Z_{\geq 0}$ such that $p\leq m$, denote
\begin{eqnarray*}(f\bullet \alpha)(s\overline{b}_{1, m-p}):=\sum_i \sum_{k=1}^{m-p+1}(-1)^{pk+m-1}f(s\overline{b}_{1, k-1}\otimes s\overline{x_i}\otimes s\overline{a}_{1, p}\otimes  s\overline{y_i}\otimes s\overline{b}_{k, m-p}).
\end{eqnarray*}
By a similar computation, we have $f\star\alpha-(-1)^{(p-1)m} \alpha\star f=0$ on $\THH^*(A, A)$. 
\end{enumerate}
It remains to verify the associativity. Since $\star$ is graded commutative on $\THH^*(A, A)$, it is enough to verify  $(x\star y)\star z=x\star (y\star z)$ for $x, y\in H^*(\calD^{<0}(A, A))$ and for $x, y\in H^*(\calD^{\geq 0}(A, A))$. But for these two cases, from a direct computation it follows that the above identity already holds on $\calD^*(A, A)$. This proves the proposition.
\end{proof}

\begin{cor}
 $\widetilde{\iota}^{*}: \THH^*(A, A)\rightarrow \HH_{\sg}^*(A, A)$ is a morphism of graded algebras.  
\end{cor}
\begin{proof}
Observe that $\iota^*$ is compatible with the products $\star$ and $\cup$ at the cochain level. Thus
the result follows from Proposition \ref{prop6.9}. 
\end{proof}
\begin{rem}
In general, the morphism  $\widetilde{\iota}^{*}$ is not an isomorphism. For instance, consider the radical square zero algebra $A=kQ/\langle Q_2\rangle$ of the quiver $Q$ with only one vertex and two loops.  We prove in \cite[Section 5]{Wan} that  $\HH_{\sg}^*(A, A)$   is of infinite dimension in each degree,  while  $\THH^*(A, A)$ is of finite dimension in each degree.  Nevertheless, in the following section, we will prove  that $\widetilde{\iota}^*$ is an isomorphism if $A$ is a self-injective algebra.  \end{rem}

\subsection{The case of self-injective algebras}
In this section, we fix a finite dimensional self-injective algebra $A$ over a field $k$. Recall that $A$ is {\it self-injective} if $A$  itself is injective as a left (or equivalently, right) $A$-module. Clearly,  symmetric algebras are naturally self-injective. Self-injective algebras play an important role in representation theory, mainly due to the fact that their stable module categories have a natural triangulated structure (cf. e.g. \cite[Section 5.1.4]{Zim}).  Moreover, we have the following result.
\begin{thm}[{\cite[Theorem2.1]{Ric}}]\label{thm-ric1}
Let $A$ be a self-injective algebra. Then the canonical functor $F_A: \mbox{$A$-$\underline{\modu}$}\rightarrow \DD_{\sg}(A)$ is an equivalence between triangulated categories.
\end{thm}
Since $A$ is self-injective, so is $A\otimes A^{\op}$.  Thus from Theorem \ref{thm-ric1}, it follows that there is an equivalence $F_{A\otimes A^{\op}}: \mbox{($A\otimes A^{\op})$-$\underline{\modu}$}\rightarrow \DD_{\sg}(A\otimes A^{\op})$ of triangulated categories.  In particular, it induces an isomorphism \begin{equation}\label{equation-ric2} \underline{\Hom}_{A\otimes A^{\op}}(A, \Omega_{\sy}^p(A)) \xrightarrow{\cong} \underline{\Ext}^p_{A\otimes A^{\op}}(A, A).\end{equation} Based on this isomorphism, we prove the following result.
\begin{prop}\label{prop6.4}
The embedding $\iota^*: \calD^*(A, A)\hookrightarrow C_{\sg}^*(A,A)$ is a quasi-isomorphism.
\end{prop}
\begin{proof}
First we note that $\iota^p$, for $p\in \Z_{>0}$, induces an isomorphism at the cohomology level from Theorem \ref{thm-ric1} and the proof of Theorem \ref{thm1}. Let us  prove that  this also holds for $p\in \Z_{\leq 0}$. Indeed, we note that  $D^{p}(A, A)\cong \HH^0(A, \Barr_{-p-1}(A))$ and  the differential $\partial^{p}: D^{p}(A,A)\rightarrow D^{p+1}(A, A)$ coincides with the differential $d_{-p-1}':=\HH^0(A, d_{-p-1}):\HH^0(A, \Barr_{-p-1}(A))\rightarrow \HH^0(A, \Barr_{-p-2}(A))$. Hence we have the following commutative diagram,
\begin{equation}\label{equation6.1}
\xymatrix@C=1.5pc{
\cdots \ar[r]& D^p(A, A)\ar[d]^{\cong} \ar[r]^-{\partial^p} & D^{p+1}(A) \ar[r] \ar[d]^{\cong}& \cdots \ar[r] & D^{-1}(A, A) \ar[r] \ar[d]^{\cong}& \HH^0(A, A)\ar[d]^{=}\\
\cdots\ar[r] & \HH^0(\Barr_{-p-1}) \ar[r]^-{d_{-p-1}'} &\HH^0(\Barr_{-p-2}) \ar[r]& \cdots \ar[r]& \HH^0(\Barr_0)\ar[r]& \HH^0(A, A)
}
\end{equation}
where for simplicity we write $\HH^0(A, \Barr_{-p}(A))$ as $\HH^0(\Barr_{-p})$.  Observe that the $p$-th cohomology of the lower complex  is isomorphic to  $\underline{\Hom}_{A\otimes A^{\op}}(A, \Omega^{-p}_{\sy}(A))$. Thus we have an isomorphism between $H^{-p}(\calD^*(A, A))$ and $\underline{\Hom}_{A\otimes A^{\op}}(A, \Omega^{-p}_{\sy}(A))$. 
Therefore, from (\ref{equation-ric2}) and  Theorem \ref{thm1},  $\iota^p$ induces an isomorphism in cohomology.  This proves the proposition.
\end{proof}

\begin{rem}
This proposition shows that  $\HH_{\sg}^*(A, A)$ can be computed by  $\calD^*(A,A)$ if $A$ is a self-injective algebra. Thus we have
\begin{equation*}
\HH^i_{\sg}(A, A)\cong
\begin{cases}
\HH^i(A, A) & \mbox{for $i>0$}, \\
\HH_{-i-1}(A, A^{\vee}) & \mbox{for $i<-1$}
\end{cases}
\end{equation*}
and for $i=-1, 0$, we have an exact sequence,
$$0\rightarrow \HH_{\sg}^{-1}(A, A) \rightarrow A^{\vee}\otimes_{A\otimes A^{\op}}A \xrightarrow{\tau} \HH^0(A, A)\rightarrow \HH_{\sg}^0(A, A)\rightarrow 0.$$
 In fact, this result is a special case of \cite[Corollary 6.4.1]{Buc} since self-injective algebras are naturally Gorenstein. Hence the quasi-isomorphism $\iota^*$ is viewed as a lifting of Buchweitz's result to the cochain level.
\end{rem}

\subsection{The case of symmetric algebras}
From now on, we fix a symmetric algebra $(A, \langle\cdot, \cdot\rangle)$ over a field $k$. Recall that   there is a natural isomorphism $A\xrightarrow{\cong} A^{\vee}, x\mapsto \sum_{\lambda}\el x\otimes \eu $ of $A$-$A$-bimodules with inverse $A^{\vee}\rightarrow A, \sum_i x_i\otimes y_i \mapsto \sum_i\langle y_i, 1\rangle x_i$, where  $\{\el \}$ is a basis of $A$ and $\{\eu \}$ is its dual basis with respect to the pairing $\langle\cdot, \cdot \rangle$ (cf. \cite{Bro}).  Under this isomorphism,  $\calD^*(A, A)$ is naturally isomorphic to the following complex (still denoted by $\calD^*(A, A)$),
$$\calD^*(A, A):= (\cdots \xrightarrow{b_2} C_1(A, A)\xrightarrow{b_1} C_0(A, A) \xrightarrow{\tau} C^0(A, A)\xrightarrow{\delta^0} C^1(A, A) \xrightarrow{\delta^1}\cdots),$$
where $\tau(x):=\sum_{\lambda} \el x\eu$. 
One may easily write down  the star product $\star$ (cf. Section \ref{section-star2}) on the new complex $\calD^*(A, A)$.

\subsubsection{Lie bracket on $\calD^*(A, A)$}
 Recall that there is a non-degenerate pairing $\langle\cdot, \cdot\rangle$ on $\calD^*(A, A)$ (cf. (\ref{equation-pairing})),
$$\langle f, \alpha\rangle=\langle \alpha, f\rangle=\begin{cases}
\langle f(s\overline{a}_{1, m}), a_0\rangle & \mbox{if $m=n$}, \\
0  & \mbox{Otherwise,}
\end{cases}$$
where $f\in C^m(A, A)$ and $\alpha:=a_0\otimes s\overline{a}_{1, n}\in C_n(A, A)$.

\begin{lem}\label{lem6.11}
\begin{enumerate}[(i)]
\item  For $x, y\in \calD^*(A, A)$, we have
$\langle \partial(x), y\rangle=(-1)^{|x|-1}\langle x, \partial(y)\rangle$.
 \item For $x, y, z\in \calD^*(A,A)$, we have $\langle x\star y, z\rangle=\langle x, y\star z\rangle.$
\end{enumerate}
\end{lem}
\begin{proof}
This follows from a straightforward computation.
\end{proof}

We now define a Lie bracket $\{\cdot, \cdot\}$ (of degree -1) on $\calD^*(A, A)$ in the following cases. 
\begin{enumerate}[\upshape(i)]
\item For $p, q\in \Z_{\geq 0}$, define $\{\cdot, \cdot\}: C_p(A, A)\otimes C_q(A, A)\rightarrow C_{p+q+2}(A, A)$ as $\{\alpha, \beta\}:=\alpha\bullet \beta-(-1)^{pq}\beta\bullet \alpha,$
 where $$\beta\bullet \alpha:=\sum^{p+1}_{i=1} (-1)^{qi} a_0\otimes s\overline{a}_{1, i-1}\otimes s\overline{\el b_0}\otimes s\overline{b}_{1, q}\otimes s\overline{\eu} \otimes s\overline{a}_{i, p}.$$
\item For $f, g\in C^*(A, A)$, define $\{f, g\}$ to be the classical Gerstenhaber bracket $[f, g]$.
\item For $m, p\in \Z_{\geq 0}$ such that $p\geq m-1$, define $\{\cdot, \cdot\}: C_p(A, A)\otimes C^m(A, A)\rightarrow C_{p-m+1}(A, A)$ as $\langle \{\alpha, f\}, g\rangle:=(-1)^{m-1}\langle \alpha, [f, g]\rangle,$ for all $g\in C^{p-m+1}(A, A)$.  Since the pairing is non-degenerate, the  above identity uniquely determines  the Lie bracket $\{\alpha, f\}$. Similarly, we define $\{f, \alpha\}$ by   $\langle \{f, \alpha\}, g\rangle:=(-1)^{m-1}\langle \alpha, [g, f]\rangle.$
\item For  $p\leq m-2$,  the brakcet $\{\cdot, \cdot\}: C_p(A, A)\otimes C^m(A, A) \rightarrow C^{m-p-2}(A, A)$ is uniquely determined by $\langle \{f, \alpha\}, \beta\rangle:=(-1)^{p}\langle f, \{\alpha, \beta\}\rangle.$
Similarly, $\langle \{\alpha, f\}, \beta\rangle:=(-1)^{p}\langle f, \{\beta, \alpha\}\rangle$ determines the Lie bracket $\{\alpha, f\}$.
\end{enumerate}
It is clear that $\{\cdot, \cdot\}$ is graded skew-symmetric.
\begin{lem} For  $x, y\in \calD^*(A, A)$, we have
$\partial (\{x, y\})=\{\partial(x), y\}+(-1)^{|x|-1}\{x, \partial(y)\}.$
\end{lem}
\begin{proof}
This follows from  Lemma \ref{lem6.11} and (\ref{equation-star}).
\end{proof}

\begin{lem}
Let $\alpha:=a_0\otimes s\overline{a}_{1, p}\in C_p(A,A)$ and $f\in C^m(A, A)$.   Then
\begin{enumerate}
\item if $p\geq m-1$, we have
\begin{eqnarray*}
\{\alpha, f\lefteqn{\}=-\sum_{i=1}^{p-m+1}(-1)^{(m-1)(p+i)} a_0\otimes s\overline{a}_{1, i-1}\otimes s\overline{f}(s\overline{a}_{i,i+m-1})\otimes s\overline{a}_{i+m, p}}\\
&&+\sum_{\lambda}\sum_{i=1}^m (-1)^{(i-1)(p-m)+m-1}\langle a_0, f(s\overline{a}_{1, i-1}\otimes s\overline{\el} \otimes s\overline{a}_{i+p-m+1, p})\rangle \\
&&\eu \otimes  s\overline{a}_{i, i+p-m},
\end{eqnarray*}
\item if $p\leq m-2$,  we have
\begin{eqnarray*}
\{\alpha, f\}(s\overline{b}_{1, r}\lefteqn{)=-\sum_{i=1}^{m-p-1}(-1)^{p(m-i)} f(s\overline{b}_{1, i-1}\otimes s\overline{\el a_0}\otimes s\overline{a}_{1, p}\otimes s\overline{\eu}\otimes s\overline{b}_{i, r})+}\\
&&\sum_{\lambda, \mu}\sum_{i=1}^{p+1}(-1)^{ri+p} \langle a_0, f(s\overline{a}_{1, i-1} \otimes s\overline{\el e_{\mu}}\otimes s\overline{b}_{1, r}\otimes s\overline{\eu} \otimes s\overline{a}_{i, p})\rangle e^{\mu}
\end{eqnarray*}
where $r:=m-p-2$.
\end{enumerate}
\end{lem}
\begin{proof}
This follows from straightforward computations.
\end{proof}
\begin{rem}
We stress that the  Jacobi identity does {\it not} hold on $\calD^*(A, A)$, although it does indeed  when all three elements are restricted to either $\calD^{<0}(A, A)$ or $\calD^{\geq 0}(A, A)$ (through a direct computation). 
Nevertheless,  it follows from  Proposition \ref{prop6.19} below that the Jacobi identity holds on the cohomology level since  the Lie bracket $\{\cdot, \cdot\}$  coincides with $[\cdot, \cdot]$ on $\HH_{\sg}^*(A, A)$. As a subsequent investigation, we prove in \cite{RiWa} that the Lie bracket $\{\cdot, \cdot\}$ can be  extended to an $L_{\infty}$-algebra structure on $\calD^*(A,A).$
\end{rem}
\begin{rem}
The Lie bracket $\{\cdot, \cdot\}$ restricted to $C_{>0}(A, A)$ coincides with the bracket constructed in \cite[Theorem 6.1]{Abb},  if  $A$ is a commutative symmetric algebra. Moreover, Abbaspour proved in loc. cit.  that  the homology $H_*(C_{>0}(A, A))$, endowed with the $\star$-product and Lie bracket $\{\cdot, \cdot\}$, is a BV algebra (without unit) whose BV operator is the Connes' B operator. Namely, \begin{equation}\label{bv-identity}\{\alpha, \beta\}=(-1)^{p}(B(\alpha\star\beta)-B(\alpha)\star\beta-(-1)^{p-1}\alpha \star B(\beta)),\end{equation}
for any $\alpha\in H^p(C_{>0}(A, A))$ and $\beta\in H^q(C_{>0}(A, A))$, where $p, q\in\Z_{>0}$. In the following, we will prove that this identity also holds on $H^{\leq 0}(\calD^*(A, A))$ for a (not necessarily commutative) symmetric algebra $A$.
\end{rem}


\begin{prop}\label{prop6.15}
Let $A$ be a symmetric  $k$-algebra.  Then  $H^{\leq 0}(\calD^*(A, A))$, equipped with the $\star$-product and Lie bracket $\{\cdot, \cdot\}$, is a BV algebra (with unit) whose BV operator is the Connes' B operator.
\end{prop}
\begin{proof} The Jacobi identity for $\{\cdot, \cdot\}$ on $\calD^{\leq 0}(A, A)$ can be verified by a direct computation.  Let us prove the BV identity (\ref{bv-identity}).
Since the proof is completely  analogous to the one of \cite[Theorem 6.1]{Abb} if $\alpha, \beta\in H^{\leq -2}(\calD^*(A, A))$, we omit this proof here.  Thus it is sufficient to consider the cases where at least one of $\alpha, \beta$ is either in $H^{-1}(\calD^*(A, A))$ or $H^0(\calD^*(A, A))$. We define the operator $B|_{\calD^0(A, A)}=0$.
\begin{enumerate}
\item If $\alpha, \beta \in H^0(\calD^*(A, A))$, then $\{\alpha, \beta\}=0$.  So the BV identity holds.
\item If only $\alpha\in H^0(\calD^*(A,A))$, then we have $\langle \{\alpha, \beta\}, f\rangle=(-1)^q\langle \beta, \{f, \alpha\}\rangle$ for all $f\in C^{q+1}(A, A)$. Thus the BV identity for $\{\alpha, \beta\}$ follows from  that for $[f, \alpha]$  on $\HH^*(A, A)$ (cf. Section \ref{section-gerstenhaber}).
\item If $\alpha\in H^{-1}(\calD^*(A, A))$, we write $\alpha:=a_0 \in C_0(A, A)$. Then we have $\partial(a_0)=\sum_{\lambda} \el a_0\eu=0$.  Observe  that
$$\partial (H_1(\alpha, \beta))=\beta\bullet \alpha+B_1(\alpha\star \beta)-B(\alpha)\star \beta$$
where $H_1(\alpha, \beta)=\sum_{\lambda} 1\otimes s\overline{a_0}\otimes s\overline{\el b_0} \otimes s\overline{b}_{1, q}\otimes s\overline{\eu}.$
Here  $B(\alpha\star \beta)=B_1(\alpha\star \beta)+B_2(\alpha\star \beta)$ and $B_1(\alpha\star \beta)=(-1)^{q-1}1\otimes s\overline{a_0\el b_0}\otimes s\overline{b}_{1, q}\otimes s\overline{\eu}.$
Therefore, it remains to verify the identity $\alpha\bullet \beta=B_2(\alpha, \beta)+\alpha\star B(\beta)$ in $H^{-q-3}(\calD^*(A, A))$. Let us construct a homotopy
\begin{eqnarray*}H_2(\alpha, \beta)=\sum_{\lambda}\sum_{0\leq j\leq i\leq q }(-1)^{(j-1)q} 1\otimes s\overline{b}_{j+1, i}\otimes s\overline{\el a_0}\otimes s\overline{\eu} \otimes s\overline{b}_{i+1, q}\otimes s\overline{b}_{0, j}.\end{eqnarray*}
Substituting $\partial(\alpha)=0$ and $\partial(\beta)=0$ into $\partial(H_2(\alpha, \beta))$,
we get three terms $-\alpha\bullet\beta, B_2(\alpha, \beta)$ and $\alpha\star B(\beta)$, which correspond to the terms when $j=0, i=q$ and $j=i,$ respectively.  Checking the sign, we have
$$\partial(H_2(\alpha, \beta))=-\alpha\bullet \beta+B_2(\alpha, \beta)+\alpha\star B(\beta).$$ 
\end{enumerate}
This verifies the BV identity. It remains to prove the Leibniz rule.
 Recall that the embedding $\iota: \calD^*(A, A)\hookrightarrow C_{\sg}^*(A, A)$ is a quasi-isomorphism. We note  that $\iota(\{\alpha, \beta\})=[\iota(\alpha), \iota(\beta)]$ for any $\alpha, \beta\in \calD^{\leq 0}(A, A)$. Thus the Leibniz rule for $H^{\leq 0}(\calD^*(A, A))$ is deduced from  that for $\HH^*_{\sg}(A, A)$ (cf. Corollary \ref{cor5.3}). This proves the proposition.
\end{proof}
\subsubsection{BV algebra structure}\label{subsection-BV}
From Theorem \ref{thm-bv2} and Proposition \ref{prop6.15},  it follows that both $\HH_{\sg}^{\geq 0}(A, A)$ and $\HH_{\sg}^{\leq 0}(A, A)$ have a BV algebra structure.
The aim of this section is  to prove the following result.
\begin{thm}\label{thm6.17}
Let $A$ be a symmetric  $k$-algebra. Then the Tate-Hochschild cohomology $\HH_{\sg}^*(A, A)$, equipped with the cup product $\cup$ and Lie bracket $[\cdot, \cdot]$ (cf. Section \ref{section4}), is a BV algebra whose BV operator $\widetilde{\Delta}^*$ is defined on $\calD^*(A, A)$ by
\begin{equation*}
\widetilde{\Delta}^i:=
\begin{cases}
\Delta^i  &  \mbox{for $i> 0$,}\\
0 & \mbox{for $i=0$,}\\
-B_{-i-1} & \mbox{for $i\leq -1$,}
\end{cases}
\end{equation*}
where $\Delta$ is determined by
$(-1)^{m-1}\langle \Delta(f)(s\overline{a}_{1, m-1}), a_0\rangle=\langle B(a_0\otimes s\overline{a}_{1, m-1}), f\rangle$.
\end{thm}
\begin{lem}\label{lem6.18}
For any $\alpha\in \THH^*(A, A)$ and $\beta\in \THH^*(A, A)$, we have
$$\{\alpha, \beta\}=(-1)^{|\alpha|}(\widehat{\Delta}(\alpha\star \beta)- \widehat{\Delta}(\alpha)\star \beta-(-1)^{|\alpha|}\alpha\star \widehat{\Delta}(\beta)),$$
\end{lem}
\begin{proof}
It follows from Theorem \ref{thm-bv2} and Proposition \ref{prop6.15} that the BV identity for $\{\alpha, \beta\}$ holds in the case $\alpha, \beta\in \THH^{\geq 0}(A, A)$ or $\alpha, \beta\in \THH^{<0}(A, A)$. 
As for other cases,  let $\alpha\in C_{p}(A, A)$ and $f\in C^m(A, A)$ such that $p\geq m-1$. By definition, we have
\begin{equation}\label{equation6.5}\langle \{\alpha, f\}, g\rangle=(-1)^{m-1}\langle \alpha, [f, g]\rangle\end{equation} for all $g\in C^{p-m+1}(A,A)$. From Theorem \ref{thm-bv2}, it follows  that  $$[f, g]=(-1)^m\left(\widehat{\Delta}(f\cup g)-\widehat{\Delta}(f)\cup g-(-1)^m f\cup \widehat{\Delta}(g)\right).$$ Substituting this into  (\ref{equation6.5}), we get
\begin{equation*}
\begin{split}
\langle\{\alpha, f\}, g\rangle&=(-1)^{m-1}\langle \alpha, [f, g]\rangle\\
& =-\left\langle \alpha, \widehat{\Delta}(f\cup g)- \widehat{\Delta}(f)\cup g-(-1)^m f\cup \widehat{\Delta}(g)\right\rangle\\
&=\left\langle (-1)^{p-1} \left(\widehat{\Delta}(\alpha\cup f)-\widehat{\Delta}(\alpha) \cup f-(-1)^{p-1}\alpha\cup \widehat{\Delta}(f)\right), g\right\rangle.
\end{split}
\end{equation*}
Thus, we have $\{\alpha, f\}=(-1)^{p-1}\left(\widehat{\Delta}(\alpha \cup f)-\widehat{\Delta}(\alpha)\cup f-(-1)^{p-1}\alpha\cup \widehat{\Delta}(f)\right).$ By the same argument, we obtain the BV identity for $p<m-1$. This proves the lemma.
\end{proof}
\begin{prop}\label{prop6.19}
The isomorphism $\widetilde{\iota}^*: \THH^*(A, A)\xrightarrow{\cong} \HH_{\sg}^*(A, A)$ is compatible with Lie brackets. Namely, $[\widetilde{\iota}^*(\alpha), \widetilde{\iota}^*(\beta)]=\widetilde{\iota}^*(\{\alpha, \beta\})$ for $\alpha, \beta\in \THH^*(A, A)$. In particular, the Jacobi identity for the Lie bracket $\{\cdot, \cdot\}$ holds on $\THH^*(A, A)$. 
\end{prop}
\begin{proof}
Clearly,  the  identity $[\iota^*(\alpha),\iota^*(\beta)]=\iota^*(\{\alpha, \beta\})$ holds on $C_{\sg}^*(A, A)$ for the two cases  where  $\alpha, \beta\in \calD^{\geq 0}(A, A)$ or $\alpha, \beta\in \calD^{\leq 0}(A, A)$. It remains to prove that \begin{equation}\label{equation6.6} [\widetilde{\iota}^*(\alpha), \widetilde{\iota}^*(f)]=\widetilde{\iota}^*(\{\alpha, f\})\end{equation} for $\alpha \in \THH^{-p-1}(A, A))$ and $f\in \THH^m(A, A))$, where $p, m\in \Z_{\geq 0}$. We need to consider the following two cases.
\begin{enumerate}
\item If $p\geq m-1$, to prove the identity (\ref{equation6.6}) is equivalent to prove the following commutative diagram.
\begin{equation*}
\xymatrix@R=3pc{
\THH^m(A, A)  \otimes  \THH^{-p-1}(A, A) \ar[d]^-{\cong}_-{\kappa_{m, 1}\otimes \kappa_{-p-1, p+2}}\ar[r]^-{\{\cdot, \cdot\}} & \THH^{m-p-2}(A, A)\ar[d]_-{\cong}^-{\kappa_{m-p-2, p+2}}\\
\HH^{m}(A, \Omega^1_{\nc}(A)) \otimes  \HH^{-p-1}(A, \Omega_{\nc}^{p+2})(A)\ar[r]^-{[\cdot, \cdot]}& \HH^{m-p-2}(A, \Omega^{p+2}_{\nc}(A))
}
\end{equation*}
Here the map $\kappa_{r, s}: \calD^{r}(A, A) \rightarrow C^{r}(A, \Omega_{\nc}^{s}(A))$ is defined as  $$\kappa_{r, s}:=\begin{cases}\theta_{s-1, r}\circ \cdots \circ \theta_{0, r} \circ \iota^{r} & \mbox{if $r-s\geq 0$}, \\ \theta_{s-1, r}\circ \cdots \circ \theta_{s-r, r} \circ \iota^{r} & \mbox{if $r-s<0$}, \end{cases}$$
where we recall that $\theta_{s,r}: C^{r}(A, \Omega^s_{\nc}(A))\rightarrow C^{r}(A, \Omega^{s+1}_{\nc}(A))$ sends $f$ to $f\otimes \id_{s\overline{A}}$.  Write $\alpha:=a_0\otimes s\overline{a}_{1, p}\in C_p(A, A)$. For any $s\overline{b}_{1, m+1}\in (s\overline{A})^{\otimes m+1},$ we have
\begin{eqnarray*}
\lefteqn{[\kappa_{m, 1}(f), \kappa_{-p-1, p+2}(\alpha)](s\overline{b}_{1, m+1})}\\
&=&\sum_{\lambda}\sum_{i=1}^m (-1)^{p(i-1)} f(s\overline{b}_{1, i-1}\otimes s\overline{\el a_0}\otimes  s\overline{a}_{1, m-i})\otimes s\overline{a}_{m-i+1, p}\otimes s\overline{\eu}\otimes s\overline{b}_{i, m+1}\\
&&+\sum_{\lambda} \sum_{i=1}^{p-m+1}(-1)^{(m-1)i}  \el a_0\otimes s\overline{a}_{1, i-1}\otimes s\overline{f}(s\overline{a}_{i, i+m-1})\otimes  s\overline{a}_{i+m, p}\otimes s\overline{\eu}\otimes \\
&& s\overline{b}_{1, m+1}
 + \sum_{\lambda}\sum_{i=p-m+2}^{p+1}(-1)^{(m-1)i}\el a_0\otimes s\overline{a}_{1, i-1} \otimes s\overline{f}(s\overline{a}_{i, p}\otimes s\overline{\eu}\otimes \\
 && s\overline{b}_{1, m-p+i-2})\otimes s\overline{b}_{m-p+i-1, m+1}.
\end{eqnarray*}
and
\begin{eqnarray*}
\lefteqn{\kappa_{m-p-2, p+2}(\{f, \alpha\})(s\overline{b}_{1, m+1})}\\
&=&\sum_{\lambda}\sum_{i=1}^{p-m+1}(-1)^{(m-1)i}
\el a_0\otimes s\overline{a}_{1, i-1}\otimes s\overline{f}(s\overline{a}_{i, i+m-1})\otimes  s\overline{a}_{i+m, p}\otimes s\overline{\eu}\otimes \\
&& s\overline{b}_{1,m+1} -\sum_{\lambda,\mu}\sum_{i=1}^{m} (-1)^{(m-i)(p-m)+m-1}\langle a_0, f(s\overline{a}_{1, i-1}\otimes s\overline{\el}\otimes s\overline{a}_{i+p-m+1, p})\rangle \\
&& e_{\mu}\eu\otimes s\overline{a}_{i, i+p-m}\otimes s\overline{e^{\mu}}\otimes s\overline{b}_{1, m+1},
\end{eqnarray*}
where $\{\el\}$ is a basis of $A$ and $\{\eu\}$ is the dual basis with respect to $\langle\cdot, \cdot\rangle$.
Comparing the above two identities, we get
\begin{eqnarray}\label{equation6.7}
\lefteqn{\left([\kappa_{m, 1}(f), \kappa_{-p-1, p+2}(\alpha)]-\kappa_{m-p-2, p+2}(\{f, \alpha\})\right)(s\overline{b}_{1, m+1})}\nonumber\\
&=&\sum_{\lambda}\sum_{i=1}^m (-1)^{p(i-1)} f(s\overline{b}_{1, i-1}\otimes s\overline{\el a_0}\otimes s\overline{a}_{1, m-i})\otimes s\overline{a}_{m-i+1, p}\otimes s\overline{\eu}\otimes s\overline{b}_{i, m+1}\nonumber\\
&&+\sum_{\lambda} \sum_{i=p-m+2}^{p+1}(-1)^{(m-1)i} \el a_0\otimes s\overline{a}_{1, i-1}\otimes  s\overline{f}(s\overline{a}_{i, p}\otimes s\overline{\eu}\otimes s\overline{b}_{1, m-p+i-2})\otimes \nonumber\\
&&s\overline{b}_{m-p+i-1, m+1}-\sum_{\lambda,\mu}\sum_{i=1}^{m} (-1)^{(m-i)(p-m)+m}\langle a_0, f(s\overline{a}_{1, i-1}\otimes s\overline{\el}\otimes s\overline{a}_{i+p-m+1, p})\rangle\nonumber\\
&& e_{\mu}\eu\otimes s\overline{a}_{i, i+p-m}\otimes s\overline{e^{\mu}}\otimes s\overline{b}_{1, m+1}.
\end{eqnarray}
For any $1\leq k\leq m$,  let us denote
\begin{eqnarray*}
\lefteqn{B_{k-1}(f, \alpha)(s\overline{b}_{1, m+1})}\\
&:=& \sum_{\lambda} (-1)^{p(k-1)}f(s\overline{b}_{1, k-1}\otimes s\overline{\el a_0}\otimes s\overline{a}_{1, m-k})\otimes s\overline{a}_{m-k+1, p}\otimes s\overline{\eu}\otimes s\overline{b}_{k,m+1}\\
&&+\sum_{\lambda} (-1)^{(m-1)(p-k-1)}  \el a_0\otimes s\overline{a}_{1, p'}\otimes s\overline{f}(s\overline{a}_{p'+1, p}\otimes s\overline{\eu}\otimes s\overline{b}_{1, k-1})\otimes s\overline{b}_{k, m+1},
\end{eqnarray*}
where $p':=p-m+k$. For any $0\leq k\leq m-1$, denote
\begin{eqnarray*}
C_{k}(f, \alpha)\lefteqn{(s\overline{b}_{1, m+1})}\\
&:=&\sum_{\lambda, \mu} \sum_{i=k+1}^m(-1)^{(m-i+k)(p-m)+m} \langle a_0, f(s\overline{a}_{1, i-k-1}\otimes s\overline{(\el\otimes s\overline{b}_{1, k})\blacktriangleleft e_{\mu}} \\
&&\otimes s\overline{a}_{i+p-m+1, p})\rangle \eu \otimes s\overline{a}_{i-k, i+p-m}\otimes  s\overline{e_{\mu}}\otimes s\overline{b}_{k+1, m+1}.
\end{eqnarray*}
In particular, we have
\begin{eqnarray*}
C_0(f, \alpha)(\lefteqn{s\overline{b}_{1, m+1})}\\
&=&\sum_{\lambda, \mu}\sum_{i=1}^m(-1)^{(m-i)(p-m)+m}\langle a_0, f(s\overline{a}_{1, i-1}\otimes s\overline{\el}\otimes s\overline{a}_{i+p-m+1, p})\rangle\\
&& e_{\mu}\eu\otimes s\overline{a}_{i, i+p-m}\otimes s\overline{e^{\mu}}\otimes s\overline{b}_{1, m+1}\end{eqnarray*}
and
\begin{eqnarray}\label{equation6.71}
[\kappa_{m, 1}(f), \kappa_{-p-1, p+2}(\alpha)]-\lefteqn{\kappa_{m-p-2, p+2}(\{f, \alpha\})}\\
&&=\sum_{k=1}^m B_{k-1}(f, \alpha)-C_0(f, \alpha).\nonumber
\end{eqnarray}
For any $1\leq k\leq m-1$, set
\begin{eqnarray*}
H_k(f, \alpha)(\lefteqn{s\overline{b}_{1, m})}\\
&:=& \sum_{\lambda, \mu} \sum_{i=k+1}^m(-1)^{(m-i-k)(p-m)+p+k}\langle a_0, f(s\overline{a}_{1, i-k-1}\otimes s\overline{\el}\otimes s\overline{b}_{1, k-1}\\
&&\otimes s\overline{e_{\mu}}\otimes s\overline{a}_{i+p-m+1, p})\rangle \eu \otimes s\overline{a}_{i-k, i+p-m}\otimes s\overline{e_{\mu}}\otimes s\overline{b}_{k, m}.\end{eqnarray*}
We claim that $\delta(H_k(f, \alpha))=C_{k}(f, \alpha)-C_{k-1}(f, \alpha)+B'_{k-1}(f, \alpha),$ where
\begin{eqnarray*}
B'_{k-1}\lefteqn{(f, \alpha)(s\overline{b}_{1, m+1})}\\
&:=&\sum_{\lambda}(-1)^{p(m-1)+m-k} f(s\overline{b}_{1, k-1}\otimes s\overline{\el}\otimes s\overline{a}_{p-m+k+1, p})a_0\otimes s\overline{a}_{1, p-m+k}\\
&&\otimes s\overline{\eu}\otimes s\overline{b}_{k, m+1}-(-1)^{k(p-m)+p} \sum_{\lambda}\eu\otimes s\overline{a}_{m-k+1, p}\otimes  a_0f(s\overline{a}_{1, m-k}\\
&&\otimes s\overline{\el}\otimes s\overline{b}_{1, k-1}) \otimes s\overline{b}_{k, m+1}.
\end{eqnarray*}
Indeed, $B_{k-1}'(f, \alpha)$ appears when $i=k+1$ and $i= m$ in  $\delta(H_k(f, \alpha))$, using $\partial(f)=0$. We note that the remaining terms in $\delta(H_k(f, \alpha))$ are cancelled by $C_k(f, \alpha)-C_{k-1}(f, \alpha)$, using $\partial(f)=\partial(\alpha)=0$.  This proves the claim.  Therefore, we have
$$\sum_{k=1}^{m-1}\delta(H_k(f, \alpha))=C_{m-1}(f, \alpha)-C_0(f, \alpha)+\sum^{m-1}_{k=1} B'_{k-1}(f, \alpha).$$
Substituting this identity  into (\ref{equation6.71}), we obtain
\begin{eqnarray*}
[\kappa_{m, 1}(f), \kappa_{-p-1, p+2}(\alpha)]-\kappa_{m-p-2, p+2}(\{f, \alpha\})
=\sum_{k=1}^m \left(B_{k-1}(f, \alpha)-B'_{k-1}(f, \alpha)\right),
\end{eqnarray*}
where  $B'_{m-1}(f, \alpha):=C_{m-1}(f, \alpha)$ for simplicity.
Thus it is sufficient to verify  $$\sum_{k=1}^m \left(B_{k-1}(f, \alpha)-B'_{k-1}(f, \alpha)\right)=0.$$
For $1\leq k\leq m$, denote
\begin{eqnarray*}
\lefteqn{H_{k-1}'(f, \alpha)(s\overline{b}_{1, m})}\\
&=&\sum_{\lambda} \sum_{i=1}^{p-m+k+1}(-1)^{p(k+i-1)+m-k-1}f(s\overline{b}_{1, k-1}\otimes s\overline{\el}\otimes  s\overline{a}_{i, m+i-k-1}) \otimes\ s\overline{a}_{m+i-k, p}\\
&&\otimes s\overline{a_0}\otimes s\overline{a}_{1, i-1} \otimes s\overline{\eu}\otimes s\overline{b}_{k, m}+\sum_{\lambda} \sum_{i=1}^{p-m+k+1} (-1)^{\epsilon_i}\el\otimes s\overline{a}_{p-i+2, p}\otimes  s\overline{a_0}\\
&&\otimes s\overline{a}_{1, p-i+k-m+1}\otimes s\overline{f}(s\overline{a}_{p-i+k-m+2, p-i+1}\otimes s\overline{\eu}\otimes s\overline{b}_{1, k-1})\otimes s\overline{b}_{k, m},
\end{eqnarray*}
where $\epsilon_i=p(i-1)+(m-1)(p-k-1)$. Then we have $$\sum_{k=1}^m\delta(H_{k-1}'(f, \alpha)) =\sum_{k=1}^m\left(B_{k-1}(f, \alpha)-B'_{k-1}(f, \alpha)\right),$$  since  we have the term $B_{k-1}(f, \alpha)$ when $i=1$, and  $B'_{k-1}(f, \alpha)$ when $i=p-m+k+1$. All other  terms in $\sum_{k=1}^m\delta(H_{k-1}'(f, \alpha))$ are cancelled by   $\partial(f)=0=\partial(\alpha)$. Therefore  the identity (\ref{equation6.6}) is verified in this case.
\item If $p\leq m-2$, we need to prove  the following commutative diagram
\begin{equation*}
\xymatrix{
\THH^m(A, A)  \otimes  \THH^{-p-1}(A, A) \ar[d]^-{\cong}_-{\id\otimes \kappa_{-p-1, p+2}}\ar[r]^-{\{\cdot, \cdot\}} & \THH^{m-p-2}(A, A)\ar[d]_-{\cong}^-{\kappa_{m-p-2, p+2}}\\
\HH^{m}(A, A) \otimes  \HH^{-p-1}(A, \Omega_{\nc}^{p+2})(A)\ar[r]^-{[\cdot, \cdot]}& \HH^{m-p-2}(A, \Omega^{p+2}_{\nc}(A)).
}
\end{equation*}
By an analogous computation as in Case 1, we get that
\begin{eqnarray}\label{equation6.81}
\lefteqn{([f, \kappa_{-p-1, p+2}(\alpha)]-\kappa_{m-p-2, p+2}(\{f, \alpha\}))(s\overline{b}_{1,m})}\nonumber \\
&=&\sum_{\lambda}\sum_{i=m-p}^m (-1)^{p(i-1)} f(s\overline{b}_{1, i-1}\otimes s\overline{\el a_0}\otimes s\overline{a}_{1, m-i})\otimes s\overline{a}_{m-i+1, p}\otimes s\overline{\eu}\otimes s\overline{b}_{i, m} \nonumber \\
&&+\sum_{\lambda} \sum_{i=1}^{p+1}(-1)^{(m-1)i} \el a_0\otimes s\overline{a}_{1, i-1}\otimes s\overline{f}(s\overline{a}_{i, p}\otimes s\overline{\eu}\otimes s\overline{b}_{1, p'})\otimes  s\overline{b}_{p'-1, m}\nonumber \\
&&-\sum_{\lambda, \mu}\sum_{i=1}^{p+1} (-1)^{(m-p)i+mp} \langle a_0,  f(s\overline{a}_{1, i-1}\otimes s\overline{\el e_{\mu}} \otimes s\overline{b}_{1, m-p-2}\otimes s\overline{\el}\otimes s\overline{a}_{i, p})\rangle \nonumber \\
&&e^{\mu} \otimes s\overline{b}_{m-p-1, m},
\end{eqnarray}
where $p':=m-p+i-2$.
For $0\leq k\leq p$, we denote
\begin{eqnarray*}
\lefteqn{B_{k}(f, \alpha)(s\overline{b}_{1, m})}\\
&=& \sum_{\lambda} (-1)^{p(m-k)}f(s\overline{b}_{1, m-p+k-1}\otimes s\overline{\el a_0}\otimes s\overline{a}_{1, p-k}) \otimes  s\overline{a}_{p-k+1, p}\otimes s\overline{\eu}\\
&&\otimes s\overline{b}_{m-p+k, m}+ \sum_{\lambda}(-1)^{(m-1)(k-1)}\el a_0 \otimes s\overline{a}_{1, k}\otimes s\overline{ f}(s\overline{a}_{k+1, p}\otimes s\overline{\eu}\\
&&\otimes s\overline{b}_{1, m-p+k-1})\otimes s\overline{b}_{m-p+k, m}.
\end{eqnarray*}
For $0\leq k\leq p$, we denote
\begin{eqnarray*}
\lefteqn{C_k(f, \alpha)(s\overline{b}_{1, m})}\\
&=& \sum_{\lambda, \mu}\sum_{i=1}^{p+1-k}(-1)^{(m-p)(k+i)+mp}\langle a_0,  f(s\overline{a}_{1, i-1}\otimes s\overline{ e_{\mu}} \otimes s\overline{b}_{1, m-p+k-2}\otimes s\overline{\el}  \\
&&\otimes s\overline{a}_{i+k, p})\rangle(e^{\mu} \otimes s\overline{a}_{i, i+k-1})  \blacktriangleleft \el  \otimes s\overline{b}_{m-p+k-1, m}.
\end{eqnarray*}
It follows from (\ref{equation6.81}) that
\begin{eqnarray}\label{equation6.82}
[f, \kappa_{-p-1, p+2}(\alpha)]-\kappa_{m-p-2, p+2}(\{f, \alpha\})=\sum_{k=0}^p B_k(f, \alpha)-C_0(f, \alpha).
\end{eqnarray}
For $0\leq k\leq p$, set
\begin{eqnarray*}
\lefteqn{H_k(f, \alpha)(s\overline{b}_{1, m-1})}\\
&=& \sum_{\lambda, \mu}\sum_{i=1}^{p+1-k}(-1)^{(m-p)(i-1)+mp+k}\langle a_0,  f(s\overline{a}_{1, i-1}\otimes s\overline{ e_{\mu}} \otimes s\overline{b}_{1, m-p+k-2}\otimes  s\overline{\el} \\
&&\otimes s\overline{a}_{i+k, p})\rangle e^{\mu} \otimes s\overline{a}_{i,i+k-1}\otimes s\overline{\eu}  \otimes s\overline{b}_{m-p+k-1, m-1}.
\end{eqnarray*}
We claim that $\delta(H_k(f, \alpha))=C_{k+1}(f, \alpha)-C_{k}(f, \alpha)+B'_k(f, \alpha),$
where
\begin{eqnarray*}
\lefteqn{B'_k(f, \alpha)(s\overline{b}_{1, m})}\\
&=& \sum_{\lambda}(-1)^{mp+k}f(s\overline{b}_{1, m-p+k-1}\otimes s\overline{\el}\otimes s\overline{a}_{k+1, p})a_0 \otimes s\overline{a}_{1, k}\otimes s\overline{\eu}\\
&&\otimes s\overline{b}_{m-p+k, m}- \sum_{\lambda}(-1)^{(m-p)k+m} \el \otimes s\overline{a}_{p+1-k, p}\otimes s\overline{a_0 f}(s\overline{a}_{1, p-k}\otimes s\overline{\eu}\\
&&\otimes s\overline{b}_{1, m-p+k-1})\otimes s\overline{b}_{m-p+k, m}
\end{eqnarray*}
and $C_{p+1}(f, \alpha):=0$.  Indeed,  the term $B_k'(f, \alpha)$ appears when $i=1$ and $i=p+1-k$ by $\partial(f)=0$ in $\delta(H_k(f, \alpha))$. The remaining terms are exactly equal to  $C_{k+1}(f, \alpha)-C_k(f, \alpha)$. Therefore, we have $$\sum^p_{k=0}\delta(H_k(f, \alpha))=-C_0+\sum_{k=0}^pB'_k(f, \alpha).$$
Substituting this identity into (\ref{equation6.82}), we get
\begin{eqnarray}\label{equation6.8}
[f, \kappa_{-p-1, p+2}(\alpha)]-\kappa_{m-p-2, p+2}(\{f, \alpha\})=
\sum_{k=0}^{p} \left(B_k(f, \alpha)- B'_k(f, \alpha)\right).
\end{eqnarray}
For $0\leq k\leq p$,  denote
\begin{eqnarray*}
\lefteqn{ H'_k(f, \alpha)(s\overline{b}_{1, m-1})}\\
&=& \sum_{\lambda}\sum_{j=1}^{k+1} (-1)^{p(m-k-j-1)+k-1}f(s\overline{b}_{1, p'}\otimes s\overline{\el }\otimes s\overline{a}_{j, j+p-k-1})\otimes s\overline{a}_{j+p-k, p}\otimes s\overline{a}_0\\
&& \otimes  s\overline{a}_{1, j-1}\otimes s\overline{\eu}\otimes s\overline{b}_{p'+1, m-1}+\sum_{\lambda}\sum^{k+1}_{j=1} (-1)^{p(j-1)+(m-1)(k-1)}\el \otimes s\overline{a}_{p-j+2, p}\\
&&\otimes s\overline{a}_0\otimes s\overline{a}_{1, k-j+1}\otimes  s\overline{f}(s\overline{a}_{k-j+2, p-j+1}\otimes s\overline{\eu}\otimes  s\overline{b}_{1, p'})\otimes s\overline{b}_{p'+1, m-1}.
\end{eqnarray*}
where $p'=m-p+k-1$.
Then using $\partial(f)=0=\partial(\alpha)$, we have
\begin{eqnarray*}
\sum_{k=0}^p\delta(H_k'(f, \alpha))=\sum_{k=0}^{p} \left(B_k(f, \alpha)-B'_k(f, \alpha)\right)
\end{eqnarray*}
since  $B_k(f, \alpha)$ appears  in $\delta(H_k'(f, \alpha))$ when $j=1,$ and $B'_k(f, \alpha)$ appears when $j=k+1$.  All other terms are cancelled by $\partial(f)=0=\partial(\alpha)$.
\end{enumerate}
This proves that $\widetilde{\iota}^*$ is compatible with  Lie brackets. \end{proof}
\begin{proof}[Proof of Theorem \ref{thm6.17}]
It follows from Proposition \ref{prop6.19} that the Jacobi identity for $\{\cdot, \cdot\}$ holds on $\THH^*(A, A)$, thus $(\THH^*(A, A), \star,  \{\cdot, \cdot\})$ is a Gerstenhaber algebra. Then Lemma \ref{lem6.18} infers that $(\THH^*(A, A), \star, \{\cdot, \cdot\}, \widehat{\Delta})$ is a BV algebra.  Since $\widetilde{\iota}^*$ is compatible with  products and Lie brackets,  we obtain that $(\HH^*_{\sg}(A, A), \cup, [\cdot, \cdot], \widetilde{\iota}\circ\widehat{\Delta}\circ\widetilde{\iota}^{-1})$ is a BV algebra and clearly, $\widetilde{\iota}^*$ is an isomorphism of BV algebras. This proves the theorem.
\end{proof}
\begin{rem}
It would be  interesting to wonder whether the cyclic Deligne's conjecture holds for $\calD^*(A,A)$, namely, whether $\calD^*(A, A)$ is an algebra over the frame little $2$-discs operad (cf. e.g. \cite{MaShSt, Kau08}). This conjecture would yield the result of Theorem \ref{thm6.17} from the operadic  point of view.
\end{rem}

\bigskip
\footnotesize
\noindent\textit{Acknowledgments.}
This work contains a part of results in author's PhD thesis and some results during the postdoc in Beijing International Center for Mathematical Research (BICMR). The author would like to thank his PhD supervisor Alexander Zimmermann for introducing this interesting topic and giving many valuable suggestions for the improvement. The author would also like to thank his PhD co-supervisor Marc Rosso for the constant support and encouragement during his career in mathematics.  He is grateful to Xiaowu Chen, Claude Cibils, Reiner Hermann, Murray Gerstenhaber, Manuel Rivera, Selene S\'anchez, Yongheng Zhang,  Guodong Zhou for many interesting discussions and useful comments. 

Part of the results in this work were presented at the workshop on Hochschild Cohomology in Algebra, Geometry, and Topology at Oberwolfach (MFO) in 2016. The author would like to thank the organizers (Luchezar L. Avramov, Ragnar-Olaf Buchweitz and Wendy Lowen) for  the opportunity to report the work.

The author would like to dedicate this paper to the memory of Professor  Ragnar-Olaf Buchweitz (1952-2017), who  shared  lots of original ideas on Tate-Hochschild cohomology and had a great  interest in this work.

\centering\footnotesize \noindent

\end{document}